\documentclass{article}
\usepackage{hyperref}
\usepackage{epsfig}
\usepackage[margin =1in]{geometry}
\usepackage{bm}
\usepackage{amsthm}
\usepackage[super]{nth}
\usepackage{amsmath,amssymb}
\usepackage{hyperref}
\usepackage{subfigure}
\usepackage{graphicx}
\usepackage{enumerate}
\usepackage{xcolor}
\usepackage{fancyhdr,graphicx,amsmath,amssymb}
\usepackage[ruled,vlined]{algorithm2e}
\usepackage{enumitem}
\usepackage[round]{natbib}
\usepackage{subfigure}

\usepackage{hyperref}
 \usepackage{fullpage}

\usepackage{bm}
 \usepackage{amsthm}
\usepackage{amsmath,amssymb}
\usepackage{mathtools}
\usepackage{epsfig}
\usepackage{hyperref}
\usepackage{graphicx}
\usepackage{enumerate}
\usepackage{xcolor}
\usepackage{fancyhdr,graphicx,amsmath,amssymb}
\usepackage[ruled,vlined]{algorithm2e}

\usepackage{xcolor}

\newcommand\Argmax{\mathop{\rm argmax}}

\newcommand\bz{\bm{0}}
\newcommand\bE{\mathbb{E}}
\newcommand\bP{\mathbb{P}}
\newcommand\te{\theta}
\newcommand{\norm}[1] {\left \| #1 \right \|} 

\newtheorem{theorem}{Theorem}
\newtheorem{lemma}{Lemma}
\newtheorem{proposition}{Proposition}

\newtheorem{assumption}{Assumption}

\usepackage{comment}
\usepackage{mathrsfs}
\newcommand{\btheta}{\bm{\theta}}
\newcommand{\bmu}{\bm{\mu}}
\newcommand{\bxi}{\bm{\xi}}
\newcommand{\bzeta}{\bm{\zeta}}

\newcommand{\R}{{\bm R}} 
\newcommand{\X}{\bm{X}}

\usepackage{tikz}

\usepackage[title]{appendix}
\author{}
\author{ Sina Baghal \thanks{E-mail: srezazadehbaghal@uwaterloo.ca}\and Courtney Paquette \thanks{E-mail: yumiko88@uw.edu; \url{cypaquette.github.io/}} \and Stephen A.~Vavasis \thanks{E-mail: vavasis@uwaterloo.ca; \url{https://www.math.uwaterloo.ca/\~vavasis}} }
\title{A termination criterion for stochastic gradient descent for binary classification \thanks{Department of Combinatorics and Optimization, University of Waterloo, Waterloo, ON, N2L 3G1, Canada; Research of Paquette was supported by
NSF DMS award 1803289 (Postdoctoral Fellowship) and research of
Vavasis was supported in part by an NSERC (Natural Sciences and
Engineering Research Council of Canada) Discovery Grant.} }
\date{\vspace{-5ex}}
\begin{document}
\maketitle

\begin{abstract}

We propose a new, simple, and computationally inexpensive termination test for constant step-size stochastic gradient descent (SGD) applied to binary classification on the logistic and hinge loss with homogeneous linear predictors. Our theoretical results support the effectiveness of our stopping criterion when the data is Gaussian distributed. This presence of noise allows for the possibility of non-separable data. We show that our test terminates in a finite number of iterations and when the noise in the data is not too large, the expected classifier at termination nearly minimizes the probability of misclassification. Finally, numerical experiments indicate for both real and synthetic data sets that our termination test exhibits a good degree of predictability on accuracy and running time.

\end{abstract}
\section{Introduction}
Minimization of an expected loss objective function using linear predictors,
\begin{equation} \label{eq:intro_expected_loss} \min_{\btheta \in \R^d} f(\btheta) := \bE_{(\bzeta,y) \sim \mathcal{P}} \ell(\bzeta^T\btheta, y), \end{equation}
is a central task in machine learning. Here the loss function $\ell: \R \times \R \to \R$, the probability distribution $\mathcal{P}$ is unknown, and the data sample $(\bzeta, y) \in \R^d \times \R$ is a random vector distributed as $\mathcal{P}$. The most prevalent algorithm employed for solving \eqref{eq:intro_expected_loss} is \textit{stochastic gradient descent} (SGD). Whereas a significant amount of work has been devoted to the convergence analysis of SGD (see, {\em e.g.}, \cite{RM1951,Curtis_SGD,Bubeck_Convex_book, Pflug}), leading, in particular, to learning rate schedules, the question of how to terminate the algorithm when one is near an optimal classifier remains largely unaddressed.

 Yet, inexpensive stopping criteria are of utmost interest in machine learning. For instance, if one could produce a low cost test to determine near-optimality, then without sacrificing the quality of the solution or efficiency of the SGD algorithm, needless computational time would be eliminated. Secondly, early termination tests impose a degree of predictability on accuracy and running times-- a useful quality when SGD occurs as a subproblem of a larger computation. Several works show that early termination of SGD can prevent overfitting, speed up learning procedures, and/or improve generalization properties \citep{Prechelt2012, Hardt:2016:TFG:3045390.3045520,Yao2007}. Motivated by these facts, we sought to address from stochastic optimization the following question:
\begin{center}
    How to design a test to terminate SGD with a fixed learning rate that is inexpensive without sacrificing quality of the solution?
\end{center}
To do so, we simplified our setting to binary classification, one of the fundamental examples of supervised machine learning  \citep{ShalevShwartzBenDavid}. In binary classification, the learning algorithm is given a sequence of training examples $(\bzeta_1,y_1),(\bzeta_2,y_2),\ldots$, often noisy, where $\bzeta_i\in\R^d$ and $y_i\in\{0,1\}$ for each $i$.  The job of the algorithm is to develop a rule for distinguishing future, unseen $\bzeta$'s that are classified as $1$ from those classified as $0$. In this work, we limit attention to linear classifiers.  This means that the learning algorithm must determine a vector $\btheta$ such that the classification of $\bzeta$ is $1$ when $\bzeta^T\btheta>0$ else it is $0$. Note that any algorithm for linear classification can be extended to one for nonlinear classification via the construction of ``kernels''; see, {\em e.g.}, \cite{ShalevShwartzBenDavid}. This extension is not pursued; we leave it for later work. 

The usual technique for determining $\btheta$, which is also adopted herein, is to define a loss function that turns the discrete problem of computing a $1$ or $0$ for $\bzeta$ to a continuous quantity. Common choices of loss functions include logistic and hinge. For simplicity, we consider only the unregularized logistic and hinge loss in this work. 

 Our theoretical results assume that our data comes from a Gaussian mixture model (GMM). The GMM is attributed to \cite{article}.  The problem of identifying GMM parameters given random samples has attracted considerable attention in the literature; see, \textit{e.g}., the recent work of \cite{Ashtiani} and earlier references therein. Another common
use of GMMs in the literature, similar to our application here, is as test-cases for a learning algorithm intended to solve a more general problem. Examples include clustering; see, \textit{e.g.,} \cite{DBLP:journals/corr/abs-1902-07137} and \cite{pmlr-v70-panahi17a} and tensor factorization; see, \textit{e.g.,} \cite{sherman2019estimating}.

Ordinarily in deterministic first-order optimization methods, one terminates when the norm of the gradient falls below a predefined tolerance.  In the case of SGD for binary classification, this is unsuitable for two reasons. First, the true gradient is generally inaccessible to the algorithm or it is computationally expensive to generate even a sufficient approximation of the gradient.

 Second, even if the computations were possible, an `optimal' classifier $\btheta$ for the classification task is not necessarily the minimizer of the loss function since the loss function is merely a surrogate for correct classification of the data.  
 
\paragraph{Our contributions.} In this paper, we introduce a new and simple termination criterion for stochastic gradient descent (SGD) applied to binary classification using logistic regression and hinge loss with constant step-size $\alpha>0$. Notably, our proposed criterion adds no additional computational cost to the SGD algorithm. 

We analyze the behavior of the classifier at termination, where we sample from a normal distribution with unknown means $\bmu_0,\bmu_1\in \R^d$ and variances $\sigma^2I_d$. Here  $\sigma>0$ and $I_d$ is the $d \times d$ identity matrix. As such, we make no assumptions on the separability of the data set. 

When the variance is not too large, we have the following results:
\begin{enumerate}
    \item The test will be activated for any fixed positive step-size. In particular, we establish an upper bound for the expected number of iterations before the activation occurs. This upper bound tends to a numeric constant  when $\sigma$  converges to zero. In fact, we show that the expected time until termination decreases linearly as the data becomes more separable ({\em i.e.}, as the noise $\sigma \to 0$). 
    \item  We prove that the accuracy of the classifier at termination nearly matches the accuracy of an optimal classifier. Accuracy is the fraction of predictions that a classification model got right while an optimal classifier minimizes the probability of misclassification when the sample is drawn from the same distribution as the training data. 
\end{enumerate}
When the variance is large, we show that the test will be activated for a sufficiently small step-size. 

We empirically evaluate the performance of our stopping criterion versus a baseline competitor. We compare performances on both synthetic (Gaussian and heavy-tailed $t$-distribution) as well as real data sets (MNIST \citep{MNIST} and CIFAR-10 \citep{cifar10}). In our experiments, we observe that our test yields relatively accurate classifiers with small variation across multiple runs.

\paragraph{Related works.}  To the best of our knowledge, the earliest comprehensive numerical testing of a stopping termination test for SGD in neural networks was introduced by \cite{Prechelt2012}. 
His stopping criteria, which we denote as \textit{small validation set} (SVS), periodically checks the iterate on a validation set.  Theoretical guarantees for SVS were established in the works of \citep{Early_stopping_Lin,Yao2007}. \cite{Hardt:2016:TFG:3045390.3045520} shows that SGD is uniformly stable and thus solutions with low training error found quickly generalize well. These results support exploring new computationally inexpensive termination tests-- the spirit of this paper.

In a related topic,
the relationship between generalization and optimization is an active area of research in machine learning. Much of the pioneering work in this area focused on understanding how early termination of algorithms, such as conjugate gradient, gradient descent, and SGD, can act as an implicit regularizer and thus exhibit better generalization properties \citep{Prechelt2012,Early_stopping_Lin,Yao2007,CG_implicit_regularization,Multipass_Lin}.  The use of early stopping as a tool for improving generalization is not studied herein because our experiments indicate that for the problem under consideration, binary classification with a linear separator, the accuracy increases as SGD proceeds and ultimately reaches a steady value but does not decrease, meaning that there is no opportunity to improve generalization by stopping early.
See also \cite{Nemirovski_Robust_Stochastic_1}.

Instead of using a validation set to stop early, \cite{Variational_Duvenaud} employs an estimate of the marginal likelihood as a stopping criteria. Another termination test based upon a Wald-type statistic developed for solving least squares with reproducing kernels guarantees a minimax optimal testing \citep{Liu_stopping}. However it is unclear the practical benefits of such procedures over a validation set.

Several works have introduced validation procedures to check the accuracy of solutions generated from stochastic algorithms based upon finding a point $\btheta_\varepsilon$ that satisfies a high confidence bound $\bP(f(\btheta_\varepsilon) - \min f \le \varepsilon) \ge 1-p$, in essence, using this as a stopping criteria (\textit{e.g.}, see \cite{Dima_Robust_Stochastic,Ghadimi_Strongly_cvx_validation_2,Ghadimi_Strongly_cvx_validation_1, Juditsky_robust_stochastic,Nemirovski_Robust_Stochastic_1}). 
Yet, notably, all these procedures produce points with small function values. For binary classification, however, this could be quite expensive and a good classifier need not necessarily be the minimizer of the loss function. Ideally, one should terminate when the classifier's direction aligns with the optimal direction-- the approach we pursue herein.

\section{Background and preliminaries}\label{sec:defs}
 Throughout we consider a Euclidean space, denoted by $\R^d$, with an inner product and an induced norm $\norm{\cdot}$. The set of non-negative real numbers is denoted by $\R_{\geq 0}$. Bold-faced variables are vectors. Throughout, the matrix $I_d$ is the $d$ by $d$ identity matrix. All stochastic quantities defined hereafter live on a probability space denoted by $(\bP, \Omega, \mathcal{F})$, with probability measure $\bP$ and the $\sigma$-algebra $\mathcal{F}$ containing subsets of $\Omega$. Recall, a random variable (vector) is a measurable map from $\Omega$ to $\R$ ($\R^d$), respectively. An important example of a random variable is the \textit{indicator of the event $A \in \mathcal{F}$}:  
\[1_A(\omega) = \begin{cases}
1, & \omega \in A\\
0, & \omega \not \in A.
\end{cases}\] 
If $X$ is a measurable function and $t \in \R$, we often simplify the notation for the pull back of the function $X$, to simply $\{\omega \in \Omega \, :\, X(\omega) \le t\} =: \{X \le t\}$.
As is often in probability theory, we will not explicitly define the space $\Omega$, but implicitly define it through random variables. For any sequence of random vectors $(\X_1, \X_2, \hdots, \X_k)$, we denote the \textit{$\sigma$-algebra generated by random vectors $\X_1, \X_2, \hdots, \X_k$} by the notation $\sigma(\X_1, \X_2, \X_3, \hdots, \X_k)$ and the \textit{expected value of $\X$} by $\bE[\X] := \int_{\Omega} \X \, d\bP$.

Particularly, we are interested in random variables that are distributed from normal distributions. In the next section, we state some known results about normal distributions. 

\paragraph{Normal distributions}
The \textit{probability density function of a univariate Gaussian} with mean $\mu$ and variance $\sigma^2$ is described by:
\begin{equation*}
    \varphi(t) := \frac{1}{\sigma\sqrt{2\pi }}\exp\left(-\frac{(t-\mu)^2}{\sigma^2}\right).
\end{equation*}
In particular, we say a random variable $\xi$ is distributed as a Gaussian with mean $\mu$ and variance $\sigma^2$ by $\xi \sim N(\mu,\sigma^2)$ to mean $\bP(\xi \le t) = \int_{-\infty}^t \varphi(t) \, dt$. When the random variable $\xi \sim N(0,1)$, we denote its cumulative density function as
\[\Phi(t) := \bP(\xi \le t) = \frac{1}{\sqrt{2\pi}} \int_{-\infty}^t \exp\left (-\xi^2 \right )  \, d\xi,\] 
and its complement by $\Phi^c(t) = 1-\Phi(t)$. The symmetry of a normal around its mean yields the identity, $\Phi(t) = \Phi^c(-t)$.

One can, analogously, formulate a higher dimensional version of the univariate normal distribution called a \textit{multivariate normal distribution}. A random vector is a multivariate normal distribution if every linear combination of its component is a univariate normal distribution. We denote such multivariate normals by $\bxi \sim N(\bm{\mu},\Sigma)$ with $\bm{\mu}\in \R^d$ and $\Sigma$ is a symmetric positive semidefinite $d\times d$ matrix. 

Normal distributions have interesting properties which simplify our computations throughout the paper. We list those which we specifically rely on.  See \cite{Famoye_univardist_book} for proofs. Below, $\bm{v},\bm{v}'\in \R^d$, $r\in \R$, $\bxi \sim N(\bmu,\sigma^2I_d)$  and $\xi\sim N(\mu,\sigma^2)$. Also, $\psi \sim N(0,1)$.

Throughout our analysis, we encounter random variables of the form $\bm{v}^T\bxi+r$, i.e. affine transformations of a given normal distribution. A fundamental property of normal distributions is that they stay in the same class of distributions after any such transformation. In other words, it holds that 
\begin{equation}\label{eq:fact_affine}
    \bm{v}^T\bm{\xi}+r\sim N(\bm{v}^T\bm{\mu}+r, \sigma^2\Vert \bm{v} \Vert^2).
\end{equation}
Working with independent random variables makes the analysis significantly easier.  In particular, it is essential for us to know when the two random variables $\bm{v}^T\bxi$ and $\bm{v}'^T\bxi$ are independent. We will use the following simple fact below: The following is true
\begin{equation} \label{eqn:fact_independence}
\mbox{$\bm{v}^T\bxi$ and  $\bm{v}'^T\bxi$ are independent} \quad \text{ if and only if } \quad \bm{v}^T\bm{v}'=0.
\end{equation}
We will also use the following simple fact about truncated normal distributions:
\begin{equation} \label{eqn:fact_truncated}
\bE_{\xi}[\xi1_{\{\xi\leq b\}}]=0  \Longrightarrow \Phi\left(\frac{b-\mu}{\sigma}\right)\cdot \exp\left(\frac{1}{2}\cdot\left(\frac{b-\mu}{\sigma}\right)^2 \right)=\frac{\sigma}{\mu}.
\end{equation}
We conclude our remarks on normal distributions with the statement of two facts about the expected value of their norm. The following hold:
\begin{equation}\label{fact:norm_Gaussians}
    \bE\left[\Vert \bxi \Vert^2 \right]= \Vert \bmu \Vert^2+d\sigma^2, \quad \bE_{\xi}[\vert \xi \vert]\leq \sqrt{\frac{2}{\pi}}\cdot \sigma+\vert\mu\vert \quad \text{and}  \quad \bE\left[\vert \psi \vert \right] = \sqrt{\frac{2}{\pi}}.
\end{equation}

\paragraph{Martingales and stopping times} Here we state some relevant definitions and theorems used in analyzing our stopping criteria in Section~\ref{sec:Analysis}. We refer the reader to \cite{Durrett_probability_book} for further details. For any probability space, $\left(\bP,\Omega,\mathcal{F}\right)$, we call a sequence of $\sigma$-algebras, $\{\mathcal{F}_k\}_{k=0}^\infty$, a \textit{filtration} provided that $\mathcal{F}_i \subset \mathcal{F}$ and $\mathcal{F}_0 \subseteq \mathcal{F}_1 \subseteq \mathcal{F}_2 \subseteq \cdots$ holds. Given a filtration, it is natural to define a sequence of random variables $\{X_k\}_{k=0}^\infty$ with respect to the filtration, namely $X_k$ is a $\mathcal{F}_k$-measurable function. If, in addition, the sequence satisfies
\[\bE[|X_k|] < \infty \quad \text{and} \quad \bE[X_{k+1} | \mathcal{F}_k] \le X_k \quad \text{for all $k$},\]
we say $\{X_k\}_{k=0}^\infty$ is a  \textit{supermartingale}. In probability theory, we are often interested in the (random) time at which a given stochastic sequence exhibits a particular behavior. Such random variables are known as \textit{stopping times}. Precisely, a stopping time is a random variable $T: \Omega \to \mathbb{N} \cup \{0, \infty\}$ where the event $\{T=k\} \in \mathcal{F}_k$ for each $k$, i.e., the decision to stop at time $k$ must be measurable with respect to the information known at that time. Supermartingales and stopping times are closely tied together, as seen in the theorem below, which gives a bound on the expectation of a stopped supermartingale. 
\begin{theorem}[See \cite{Durrett_probability_book} Theorem 4.8.5]\label{thm:Durret_Martingale}
Suppose that $\left\{X_k\right\}_{k=0}^{\infty}$ is a supermartingale w.r.t to the filtration $\left\{\mathcal{F}_k\right\}_{k=0}^{\infty}$ and let $T$ be any stopping time satisfying $\bE[T]<\infty$. Moreover if $\bE\left[\vert X_{k+1}-X_k\vert|\mathcal{F}_k\right]\leq B$ a.s. for some constant $B>0$, then it holds that $\bE[X_T]\leq \bE[X_0]$.
\end{theorem}
As we illustrate in Section~\ref{sec:Analysis}, a connection between stopping criteria (i.e. the decision to stop an algorithm) and stopping times naturally exists.

\section{Stopping criterion for stochastic gradient descent}\label{sec:termtest} We analyze learning by minimizing an expected loss problem of homogeneous linear predictors ({\em i.e.}, without bias) of the form 
\begin{align*}
    \bE_{(\bzeta,y) \sim \mathcal{P}} [\ell(\bzeta^T\btheta, y)] 
\end{align*}
using logistic and hinge regression. Here the samples $(\bzeta, y) \in \R^d \times \{0,1\}$. We recall that in logistic regression the loss function is defined as follows
\begin{equation}\label{eq:log_loss_definition}
    \ell(x,y):=-yx+\log\left(1+\exp(x)\right).
\end{equation}
Also, the hinge loss is defined as the following
\begin{equation}\label{eq:hinge_loss_definition}
    \begin{aligned}
    \ell(x,y):=\begin{cases} \max(1-x,0)& \quad y=1, \\\max(1+x,0)& \quad y=0. \end{cases}
    \end{aligned}
\end{equation}


 The data comes from a mixture model, that is, flip a coin to determine whether an item is in the $y=0$ or $y=1$ class, then generate the sample $\bzeta$ from either the distribution $\mathcal{P}_0$ (if $y=0$ was selected) or $\mathcal{P}_1$ (if $y=1$ was selected). We denote the mean of the $\mathcal{P}_0$ (resp. $\mathcal{P}_1$) distribution by $\bmu_0$ (resp. $\bmu_1$). 
The homogeneity of the linear classifier is without loss of much generality because we can assume $\bmu_0 = - \bmu_1$. We enforce this assumption, with minimal loss in accuracy, by recentering the data using a preliminary round of sampling (see Sec.~\ref{sec:Num_Experiment}).

Because of the homogeneity, we can
simplify the notation by redefining our training examples to be $\bxi_k:=(2y_k-1)\bzeta_k$ and then assuming that for all $k \ge 0$, $y_k = 1$.  Then the new samples $\bxi$ can be drawn from a \textit{single}, mixed distribution $\mathcal{P}_*$ with mean $\bmu:=\bmu_1$ where sampling $\bxi\sim \mathcal{P}_1$ occurs with probability 0.5 and $-\bxi\sim\mathcal{P}_0$ occurs with probability 0.5.
We make this simplification and, from this point on, we analyze the following optimization problem:
\begin{equation}\label{optimization_problem}
\min_{\bm{\theta} \in \R^d} f(\bm{\theta}):= \bE_{\bm{\xi} \sim \mathcal{P}_*}[ \ell(\bxi^T\btheta,1)] 
\end{equation}
Let us remark that the right-hand side of \eqref{optimization_problem} is differentiable with respect to $\btheta$ in either cases of logistic and hinge loss functions. Indeed, in case of hinge loss, note that for any $\btheta_{k-1}$, the function $\bxi_k \mapsto \ell(\bxi_k^T\btheta_{k-1},1)$ is almost surely differentiable as $\bP_{\bxi_k}\left(\bxi_k^T\btheta_{k-1}=1\right)=0$. Hence, we consider the expectation in \eqref{optimization_problem} to be over $\mathbb{R}^d\backslash \left\{\bxi_k:\bxi_k^T\btheta_{k-1}=1 \right\}$ on which the argument is differentiable with respect to $\btheta_{k-1}$.

The most widely used method to solve \eqref{optimization_problem} is SGD. Unlike gradient descent which uses the entire data to compute the gradient of the objective function, the SGD algorithm, at each iteration, generates a sample from the probability distribution and updates the iterate based only on this sample,  
\begin{align}
 \label{eq:derivative_formula}
 \bm{\te}_{k}=\bm{\theta}_{k-1} -\alpha \nabla_{\btheta} \ell(\bxi_{k}^T\btheta_{k-1},1),
\end{align}
where $\bm{\xi}_k \sim \mathcal{P}_*$.  Our presentation of SGD assumes a constant step-size $\alpha>0$. Constant step-size is commonly used in machine learning implementations despite the decreasing step-size often assumed to prove convergence (see, \textit{e.g.}, \cite{RM1951}). \cite{Nemirovski_Robust_Stochastic_1} explain in more detail the theoretical basis for both constant and decreasing step-size and provide an explanation as well as workarounds for the poor practical performance of decreasing step-size.  However, in practice, constant step-size is still widely used.
With constant step-size, SGD is known to asymptotically converge to a neighborhood of the minimizer (see, {\em e.g.}, \cite{Pflug}). Yet, for binary classification, one does not require convergence to a minimizer in order to obtain good classifiers.  

For homogeneous linear classifiers applied to the hinge loss function, it has been shown (\cite{molitor2019bias}) that the homotopic sub-gradient method converges to a maximal margin solution on linearly separable data. In (\cite{Srebro_logistic}), SGD applied to the logistic loss on linearly separable data will produce a sequence of $\btheta_k$ that diverge to infinity, but when normalized also converge to the $L_2$-max margin solution. Little is known about the behavior of constant step-size SGD when the linear separability assumption on the data is removed (see, \textit{e.g.}, \citep{Telgarsky_logistic}). The assumption of zero-noise in our context would mean that $\mathcal{P}_0$, $\mathcal{P}_1$ each reduce to a single point, a trivial example of separable data. Since there is often noise in the sample procedure, the data \textit{may not necessarily be linearly separable}. Understanding the behavior of SGD in the presence of noise is, therefore, important.  

\subsection{Stopping criterion} A common stopping criterion from deterministic first-order optimization methods is to terminate at an iterate satisfying $\norm{\nabla f(\btheta)}^2 < \varepsilon$ for a predetermined $\varepsilon > 0$. Yet, in stochastic optimization, the full gradient is inaccessible or it is simply too expensive to compute. Several works \citep{Dima_Robust_Stochastic,Ghadimi_Strongly_cvx_validation_2,Ghadimi_Strongly_cvx_validation_1, Juditsky_robust_stochastic,Nemirovski_Robust_Stochastic_1} have suggested an alternative for the stochastic setting-- terminate when $\bP(f(\btheta)-\min~f \le \varepsilon) \ge 1-p$ for some chosen small $\varepsilon >0$ and probability $p$. 
However, for binary classification, the minimizer of the loss function and a perfect classifier may not be the same or one may find a suitable substitute, at a lower cost, without having to compute the exact minimizer.

\paragraph{Optimal classifiers.} In classification, we call a classifier, $\btheta^*$, \textit{optimal} if it has the property that 
\begin{equation} \label{eq:stop_criteria_observation} 
\btheta^*\in \Argmax_{\btheta} \bP\left(\bxi^T\btheta>0\,|\,\bxi\sim \mathcal{P}_*\right), \end{equation}
\textit{i.e.}, the classifier, $\btheta^*$, minimizes the probability of misclassifying. Note there exist many optimal classifiers, in fact, the condition \eqref{eq:stop_criteria_observation} is scale-invariant; hence, for any $\lambda > 0$, $\lambda \cdot \bxi^T \btheta^* > 0 \Longleftrightarrow \bxi^T \btheta^* > 0$. 
Even though the binary classifier is scale-free, the logistic and hinge regression loss is not.  It transitions from flat to unit-slope when $\bxi^T\btheta=O(1)$. This suggests that when  $\btheta$ reaches this region, a classification has been made. 

\paragraph{Termination test.} Motivated by the above property of optimal classifiers, we propose the following termination test: Sample $\hat{\bxi}_k \sim \mathcal{P}_*$ and
\begin{align} \label{eq: termination_test}
  \text{Terminate when \;$ \hat{\bxi}_k^T\btheta_k \ge 1$}.
\end{align}
A second motivation for this termination test comes from support vector machine (SVM) theory
\citep{ShalevShwartzBenDavid} in which the scaling of the optimizing classifier is constrained so that the margin between classes is $O(1)$. 
Therefore, our termination test blends an SVM notion with SGD. Algorithm~\ref{alg:SGD_termination} describes the termination criteria \eqref{eq: termination_test} as applied with the update rule governed by SGD.

The termination test \eqref{eq: termination_test} requires an additional sample and an additional inner product per iteration and, as such,  imposes a small additional cost. To reduce this cost, in all our numerical experiments (Sec.~\ref{sec:Num_Experiment}), we use the following termination test. 
\begin{equation}\label{eq: practical_termination_test}
\mbox{Terminate when \;$ \bxi_{k+1}^T\btheta_k\ge 1,$}
\end{equation}
which imposes no computational overhead as SGD already computes $\bxi_{k+1}^T\btheta_k$. Unfortunately, we could not perform a straightforward analysis of \eqref{eq: practical_termination_test} because it introduces additional dependencies in the sequences $\{\bxi_k\}_{k=1}^\infty$ and $\{\btheta_k\}_{k=0}^\infty$. 
After testing both \eqref{eq: termination_test} and \eqref{eq: practical_termination_test}, we found that up to the noise from the randomness, their behaviors in numerical experiments were identical. 

\begin{algorithm}[htp!]
\textbf{initialize:} $\bm{\theta}_0 \in \R^d$, $\alpha > 0$, $\hat{\bxi}_0 \sim \mathcal{P}_*$, $k = 0$\\
\textbf{while $\hat{\bxi}_k^T \btheta_k < 1$}\\
\qquad \qquad Pick data point $\bm{\xi}_{k+1} \sim \mathcal{P}_*$.\\
\qquad \qquad Compute $\nabla_{\bm{\theta}} \ell(\bm{\xi_{k+1}}^T\bm{\theta}_k,1)$ as in \eqref{eq:derivative_formula} \\
\qquad \qquad Update $\bm{\theta}$ by setting 
\begin{equation}\label{eq: SGD_update_expected_loss} \bm{\theta}_{k+1} \leftarrow \bm{\theta}_k - \alpha \nabla_{\bm{\theta}} \ell(\bm{\xi_{k+1}}^T\bm{\theta}_k,1) \end{equation}\\
\qquad \qquad Sample $\hat{\bxi}_{k+1} \sim \mathcal{P}_*$\\
\qquad \qquad $k \leftarrow k+1$\\
\textbf{end}
\caption{SGD with termination test} \label{alg:SGD_termination}
\end{algorithm}

\begin{assumption}\label{Gaussian_Assumption}[The distribution $\mathcal{P}_{*}$ is Gaussian] \rm{Our theoretical analysis makes a further assumption on the distribution $\mathcal{P}_*$. For the rest of this section and Sec.~\ref{sec:Analysis}, $\mathcal{P}_0=N(\bmu_0,\sigma^2I_d)$, $\mathcal{P}_1=N(\bmu_1,\sigma^2I_d)$, and therefore $\mathcal{P}_*=N(\bmu,\sigma^2I_d)$, a Gaussian with unknown mean $\bmu\; (=\bmu_1=-\bmu_0)$ and variance $\sigma^2I_d$. This assumption allows for non-separable data provided $\sigma > 0$.}
\end{assumption}


\paragraph{The minimizer of logistic and hinge regression}
In \eqref{eq:stop_criteria_observation} we defined $\btheta^*$ to be any member of the set of optimal classifiers.  For the remainder of this section, we provide an exact characterization of this set.  In the next lemma, we redefine $\btheta^*$ to the minimizer of the expected loss function for either hinge or logistic and show that it is a positive scalar multiple of $\bmu$.  We will continue to use $\btheta^*$ with this meaning for the remainder of the paper.  In the lemma after that, we show that the set of optimal classifiers are exactly positive scalar multiples of $\bmu$ (or of $\btheta^*$).

\begin{lemma}[Minimizer of the logistic and hinge loss] \label{lem:minimizers} The function $f$ defined in \eqref{optimization_problem} with $\ell$ defined in \eqref{eq:log_loss_definition} or \eqref{eq:hinge_loss_definition} has a unique minimizer at $\btheta^* = \rho^*\bmu$ for some $\rho^*\in (0,+\infty)$. Moreover, let $r=\rho^*\sigma^2$. Then in the case of logistic regression, it holds that $r=2$ and in the case of hinge loss, $w=\frac{\sigma}{r\Vert \bmu \Vert}-\frac{\Vert \bmu \Vert}{\sigma}$ satisfies 
\begin{equation}\label{eq:w_hinge}
\frac{1}{\sqrt{2\pi}}\cdot\frac{\sigma}{\Vert \bmu \Vert}=\Phi(w)\cdot \exp(\tfrac{1}{2}w^2).
\end{equation}
\end{lemma}

\begin{proof}
We consider the logistic and hinge loss case separately. 
\item \textbf{Logistic loss.} We have 
\[
f(\btheta)=\bE_{\bm{\xi} \sim N(\bmu,\sigma^2I_d)}[ -\btheta^T\bxi + \log(1+ \exp(\btheta^T\bxi))].
\]
Clearly, $f$ is a convex function. We next observe that for any $\bm{v},\btheta\in \R^d$ with $\bm{v}^T\btheta=0$, it holds that  
\begin{equation}\label{app.eq:min_eq101}
  \bm{v}^T\nabla f\left(\btheta\right)=\bE_{\bxi}\left[\frac{\bxi^T\bm{v}}{1+\exp(\bxi^T\btheta)}\right]=\bE_{\bxi}[\bxi^T\bm{v}]\bE_{\bxi}\left[\frac{1}{1+\exp(\bxi^T\btheta)}\right]=\bm{v}^T\bmu\cdot\bE_{\bxi}\left[\frac{1}{1+\exp(\bxi^T\btheta)}\right].
\end{equation}
Here we used that $\bxi^T\bm{v}$ and $\bxi^T\btheta$ are independent random variables and the expectation of the product of two uncorrelated random variables is the product of the expectations.  Now note that for any $\btheta$, the quantity  $\bE_{\bxi}\left[\frac{1}{1+\exp(\bxi^T\btheta)}\right]$ is strictly positive. Therefore, if  $\bm{v}^T\btheta=0$ and $\nabla f(\btheta)=\bm{0}$ then, using \eqref{app.eq:min_eq101}, we obtain that $\bm{v}^T\bmu=0$. Hence, we established that $\nabla f(\btheta)=\bm{0}$ implies $\btheta=\rho \bmu$ for some $\rho \in \R$. 
On the other hand, using \eqref{app.eq:min_eq101} again, we have that $\nabla f(\rho \bmu)=0$ if and only if $\bmu^T\nabla f(\rho \bmu)=0$. To see the only if direction, suppose $\bmu^T\nabla f(\rho \bmu) = 0$ and $\nabla f(\rho \bmu) \neq 0$. Then we have $\nabla f(\rho \bmu) = \bm{v}$ where the vector $\bm{v}$ is nonzero such that $\bm{v}^T\bmu = 0$. By \eqref{app.eq:min_eq101}, we deduce $\norm{\bm{v}}^2 = \bm{v}^T \nabla f(\rho \bmu) = 0$ yielding a contradiction. 

Next, we consider the function,
\begin{equation*}
  g(\rho):=-\bE_{\bxi}\left[\frac{\bmu^T\bxi}{1+\exp(\rho{\bmu}^T\bxi)} \right].
\end{equation*}
Observe that $g(\rho) = \bmu^T \nabla f(\rho \bmu)$. Therefore, if we can show $g(\rho)$ has a unique zero at $\rho = \tfrac{2}{\sigma^2} =: \rho^*$, we can conclude that $\bmu^T\nabla f(\rho^* \bmu) = 0$ which, in turn, gives us that $\rho^* \bmu$ is the unique solution to $\nabla f(\rho^* \bmu) = 0$. It remains to show that $\rho^*$ is the unique zero of $g$. By \eqref{eq:fact_affine}, $z: = \bmu^T\bxi \sim N(\Vert \bmu\Vert^2,\sigma^2\Vert \bmu \Vert^2)$. Therefore, this yields
\begin{equation*}
    g(\rho)=\frac{1}{\sigma\Vert \bmu \Vert \sqrt{2\pi}} \int_{-\infty}^{\infty} \frac{z}{1+\exp(\rho z)} \exp\left(-\frac{(z-\Vert \bmu \Vert^2)^2}{2\sigma^2\Vert \bmu \Vert^2} \right)dz.
\end{equation*}
Expanding out the term inside the integral, we conclude 
\begin{align}
    \frac{z}{1+\exp(\rho z)} \exp\left(-\frac{(z-\Vert \bmu \Vert^2)^2}{2\sigma^2\Vert \bmu \Vert^2} \right)&=\frac{z}{2\cosh\left(\frac{\rho z}{2}\right)}\exp\left(-\frac{\rho z}{2}-\frac{(z-\Vert \bmu \Vert^2)^2}{2\sigma^2\Vert \bmu \Vert^2} \right)\nonumber\\
    &=\frac{z}{2\cosh\left(\frac{\rho z}{2}\right)}\exp\left(-\frac{z^2+\left(\rho\sigma^2\Vert \bmu \Vert^2-2\Vert \bmu \Vert^2 \right)z+\Vert \bmu \Vert^4}{2\sigma^2\Vert \bmu \label{app.eq:integrho^*}
    \Vert^2}\right).
  \end{align}
When $\rho = \rho^*$, we observe that equation \eqref{app.eq:integrho^*} is an odd function of $z$. Therefore, the function $g(\rho^*) = 0$, i.e. the integral of \eqref{app.eq:integrho^*} is $0$. To see that $\rho^*$ is the only zero of $g$, we note that
\[
g'(\rho)=\bE_{\bxi}\left[\frac{\left(\bmu^T\bxi \right)^2\exp(\rho \bmu^T\bxi)}{(1+\exp(\rho \bmu^T\bxi))^2} \right] >0.
\]
Here, $g'(\rho)=0$ implies that $\bmu^T\bxi=0$ a.s. which is not true. As a result, the function $g(\rho)$ is strictly decreasing with a zero at $\rho^*$. The result follows.  
\item \textbf{Hinge loss.} We begin by noting that $f$ is differentiable and it holds that
\[
\nabla f(\btheta)= -\bE_{\bxi}[\bxi1_{\{\bxi^T\btheta\leq 1\}}].
\]
We next observe that for any $\bm{v},\btheta\in \R^d$ such that $\bm{v}^T\btheta=0$, it holds that 
\begin{equation}\label{eq:hinge_min_lemma}
    -\bm{v}^T\nabla f(\btheta)=\bE_{\bxi}[\bm{v}^T\bxi1_{\{\bxi^T\btheta\leq 1\}}]=\bE_{\bxi}[\bm{v}^T\bxi]\bE_{\bxi}[1_{\{\bxi^T\btheta\leq 1\}}]=\bm{v}^T\bmu\cdot\bE_{\bxi}[1_{\{\bxi^T\btheta\leq 1\}}].
\end{equation}
Here we used that $\bxi^T\bm{v}$ and $\bxi^T\btheta$ are independent random variables and the expectation of the product of two uncorrelated random variables is the product of the expectations.  Now note that for any $\btheta$, the quantity  $\bE_{\bxi}[1_{\{\bxi^T\btheta\leq 1\}}]$ is strictly positive. Therefore, if  $\bm{v}^T\btheta=0$ and $\nabla f(\btheta)=\bm{0}$ then, using \eqref{eq:hinge_min_lemma}, we obtain that $\bm{v}^T\bmu=0$. Hence, we established that $\nabla f(\btheta)=\bm{0}$ implies $\btheta=\rho \bmu$ for some $\rho \in \R$.  On the other hand, using \eqref{eq:hinge_min_lemma} again, we have that $\nabla f(\rho \bmu)=0$ if and only if $\bmu^T\nabla f(\rho \bmu)=0$. To see the only if direction, suppose $\bmu^T\nabla f(\rho \bmu) = 0$ and $\nabla f(\rho \bmu) \neq 0$. Then we have $\nabla f(\rho \bmu) = \bm{v}$ where the vector $\bm{v}$ is nonzero such that $\bm{v}^T\bmu = 0$. By \eqref{eq:hinge_min_lemma}, we deduce $\norm{\bm{v}}^2 = \bm{v}^T \nabla f(\rho \bmu) = 0$ yielding a contradiction. 

Next, consider the function 
\begin{equation}\label{eq:derivative_hinge}
   g(\rho)=\bE_{\bxi}[ \bmu^T\bxi 1_{\{\rho \bxi^T\bmu\leq 1 \}}]. 
\end{equation}

Observe that $g(\rho) = \bmu^T \nabla f(\rho \bmu)$. Dominated Convergence Theorem yields that 
\[
\lim_{\rho \to +\infty}g(\rho)=\bE_{\bxi}[\bmu^T\bxi1_{\{\bmu^T\bxi\leq 0 \}}], \quad \lim_{\rho \to -\infty}g(\rho)=\bE_{\bxi}[\bmu^T\bxi1_{\{\bmu^T\bxi\geq 0 \}}].
\]
It, therefore, holds that $\lim_{\rho \to +\infty}g(\rho)<0$ and $\lim_{\rho \to -\infty}g(\rho)>0$. Since $g(0)=\bE_{\bxi}[\bmu^T\bxi]>0$, it remains to show that $g$ is a strictly decreasing function. To this end, we note that for any fixed $\rho_1<\rho_2$, it holds that 
\begin{equation}\label{eq:hinge_g_decreasing}
\bmu^T\bxi\left(1_{\{\rho_1\bmu^T\bxi\leq 1 \}}-1_{\{\rho_2\bmu^T\bxi\leq 1 \}} \right) \geq 0 \quad \text{for any value of $\bxi$}.
\end{equation}
Indeed, if $\bmu^T\bxi\geq 0$, then $\rho_1\bmu^T\bxi\leq \rho_2\bmu^T\bxi$; thus ensuring $1_{\{\rho_1\bmu^T\bxi\leq 1 \}}\geq 1_{\{\rho_2\bmu^T\bxi\leq 1 \}}$. The case $\bmu^T\bxi\leq 0$ follows similarly. We, therefore, conclude that $g(\rho_1)\geq g(\rho_2)$. Finally, note that $g(\rho_1)=g(\rho_2)$, implies that \eqref{eq:hinge_g_decreasing} holds with equality, almost surely. Clearly, this yields a contradiction. It remains to show \eqref{eq:w_hinge}. By \eqref{eq:derivative_hinge}, we have that $g'(\rho^*)=\bE_{\bxi}[\bmu^T\bxi1_{\{\bmu^T\bxi\leq \frac{1}{\rho^*}\}}]$.  Using \eqref{eq:fact_affine} and \eqref{eqn:fact_truncated}, we obtain that 
\begin{equation}\label{myeq2}
    \Phi\left(\frac{1-\rho^*\Vert \bmu \Vert^2}{\rho^*\sigma\Vert\bmu\Vert} \right)\cdot \exp\left(\frac{1}{2}\cdot\left(\frac{1-\rho^*\Vert \bmu \Vert^2}{\rho^*\sigma\Vert\bmu\Vert} \right)^2\right)=\frac{1}{\sqrt{2\pi}}\cdot\frac{\sigma}{\Vert\mu\Vert}.
\end{equation}
The result immediately follows.

\end{proof}

The previous lemma has defined $\btheta^*$ to be the minimizer of the loss function and showed that it is a positive multiple of $\bmu$.  We now show that this $\btheta^*$ and its positive scalar multiples are exactly the set of optimal classifiers in the sense of \eqref{eq:stop_criteria_observation}, i.e., we give an exact characterization of that set.

\begin{lemma}[Characterization of the optimal classifier]
The following is true 
\begin{equation} 
\Argmax_{\btheta} \bP\left(\bxi^T\btheta>0\right) = \{ \lambda\cdot \btheta^*: \lambda>0\}.\end{equation}
\end{lemma}

\begin{proof}
Observe that the following simple fact holds.
\begin{equation}\label{app.fact:PhiProb}
    \bP_{\hat{\bxi}}\left(\hat{\bxi}^T\btheta\geq t\right)=\Phi^c\left( \frac{\bmu^T\btheta-t}{\sigma \Vert \btheta\Vert}\right), \quad \text{for all $\btheta\in \R^d$, $t\in \R$ and $\hat{\bxi} \sim N(\bmu,\sigma^2I_d)$}.
\end{equation}
Therefore we have that $\bP_{\bxi}(\bxi^T\btheta>0)=\Phi^c\left(\frac{\Vert \bmu \Vert}{\sigma}\cdot\cos(w_{\btheta}) \right)$ where $\bxi \sim N(\bmu,\sigma^2I_d)$ and $w_{\btheta}$ denotes the angle between the two vectors $\btheta$ and $\bmu$. On the other hand a classifier $\btheta$ is optimal if and only if $\btheta=\rho \bmu$ for some $\rho > 0$, i.e. $\cos(w_{\btheta})=0$. The proof is complete after noting that $\Phi$ is an increasing function. 
\end{proof}

\section{Analysis of stopping criterion} \label{sec:Analysis}
In this section, we present our analysis of the stopping criterion \eqref{eq: termination_test} proposed in Section~\ref{sec:termtest}. Here we introduce the first iteration at which the stopping criterion is satisfied, denoted by the random variable
 \begin{equation}\label{eq:stopping_criteria}
    T:=\inf\left\{k> 0: \hat{\bxi}_k^T\btheta_k \geq 1\right\}.
\end{equation}
By viewing the stopping criterion through the lens of stopping times, we are able to utilize probability theory to analyze the classifier at termination $\btheta_T$.  Throughout this section, we work with the following filtration.
\begin{equation}\label{eq:sigma}
  \mathcal{F}_0 = \sigma(\btheta_0) \quad \text{and} \quad  \mathcal{F}_k:=\sigma(\btheta_0, \hat{\bxi_1}, \bxi_1, \hat{\bxi_2}, \bxi_2, \hdots, \hat{\bxi_k}, \bxi_k), \quad \text{for all $k\geq 1$}
\end{equation}
 Clearly, the random variable $\btheta_k$ is $\mathcal{F}_k$-measurable.  Our theoretical results are structured as follows.

 First, we show that SGD with our proposed termination test indeed stops after a finite number of iterations. To do so, we provide a bound on $\bE[T]$, \textit{i.e.} the expected number of iterations before termination. Yet, despite this guarantee, the resulting classifier at termination need not be optimal. Hence, our second result establishes that both $\btheta_T$ and $\btheta^*$ point in approximately the same direction; thereby ensuring that the classifier at termination, $\btheta_T$, is nearly optimal. We remark the worst-case bounds established throughout these sections are conservative; we observe in our experiments that the termination test stops sooner while also yielding good classification properties for Gaussian and non-Gaussian data sets. 

To bound $\bE[T]$, we identify subsets of $\R^d$ for which when an iterate enters the set, termination (\textit{i.e.} \eqref{eq: termination_test}) is \textit{highly likely} to succeed. Such sets $C$, we call \textit{target sets}. Precisely, for any $\btheta \in C$ and $\hat{\bxi}\sim N(\bmu,\sigma^2I_d)$, the probability of terminating is at least $\delta>0$,
\begin{equation}\label{eq:probability_delta}
    \exists \; \delta>0 \text{ such that } \mathbb{P}_{\hat{\bxi}}\left( \hat{\bxi}^T\btheta\geq 1 \right) \geq \delta.
\end{equation}
We guarantee the iterates generated by SGD enter the target set by way of a \textit{drift function}, $V:\R^d\rightarrow [0,+\infty)$. A drift function, on average, decreases each time the iterate fails to live in the target set. In other words, conditioned on the past iterates the following holds 
\begin{equation}\label{eq:driftequation}
    (\bE[V(\btheta_k)|\mathcal{F}_{k-1}]-V(\btheta_{k-1})])1_{\{\btheta_{k-1}\not\in C\}} \leq -b1_{\{\btheta_{k-1}\not\in C\}}
\end{equation}
for the target set $C$ and some positive constant $b$. Loosely speaking, the iterates in expectation \textit{drift} towards the target set. Target sets and drift functions in the context of drift analysis are well-studied in stochastic processes, see Lemma \ref{lem:drift_from_meyn} below.

A natural choice for the target set is a neighborhood of the unique optimum solution of \eqref{optimization_problem}, $\btheta^*$, with the drift function $\Vert \btheta-\btheta^* \Vert^2$. Indeed, it is known the iterates of SGD converge to a neighborhood of $\btheta^*$ (\cite{Pflug}). However, an iterate may be nearly optimal well before it enters this neighborhood. In fact when $\sigma \ll \Vert \bmu \Vert$, we identify a target set where satisfying the stopping criterion occurs at least half the time and does not require the iterate to be near $\btheta^*$. We summarize below our target set and drift function. 

\begin{enumerate}
    \item Under the assumption $\sigma \leq c \Vert \bmu \Vert$ for some numerical constant $c$, which we call the \textit{Low Variance Regime}, we define the target set to be 
    \begin{equation}\label{eq:targetset_lownoise}
        C = \{\btheta: \bmu^T\btheta\geq 1 \},
    \end{equation}
    and the drift function by 
    \begin{equation}\label{eq:driftfunction_lownoise}
        V(\btheta)=\left(M-\bmu^T\btheta\right)^2,
    \end{equation}
    for some constant $M$, to be determined later.
    \item Under the assumption $c\Vert \bmu \Vert \leq \sigma$ where the constant $c$ is the same as in 1 above, which we call the \textit{High Variance Regime}, we define the target set to be 
    \begin{equation}\label{eq:targetset_highnoise}
        C=\{\btheta: \vert \rho\sigma^2-1\vert <1 \text{ and } \sigma \Vert \tilde{\btheta}\Vert\leq c'\},
    \end{equation}
    for some numerical constant $c'$. Here, we orthogonally decompose $\btheta=\rho \bmu+\tilde{\btheta}$ with $\bmu^T\tilde{\btheta}=0$. We use the following drift function
    \begin{equation}\label{eq:driftfunction_highnoise}
     V(\btheta)=\frac{1}{2\alpha}\Vert \btheta-\btheta^*\Vert^2.   
    \end{equation}
\end{enumerate}
In Section \ref{sec:proof_low} (resp. Section \ref{sec:proof_high}) we show that the pairs $(C,V)$ defined in \eqref{eq:targetset_lownoise} and \eqref{eq:driftfunction_lownoise} (resp. \eqref{eq:targetset_highnoise} and \eqref{eq:driftfunction_highnoise}) satisfies the drift equation \eqref{eq:driftequation} for any step-size $\alpha$ (resp. for any sufficiently small step-size $\alpha$).

As mentioned above, the target set $C$ attracts the iterates generated by SGD. Each time an iterate enters $C$, the stopping criterion holds with probability at least $\delta>0$. Provided the iterates enters the set $C$ an infinite number of times, then after waiting a geometrically distributed many iterations, we expect the following condition to hold:
\begin{equation}\label{eq:termination_inside_C}
    \hat{\bxi}_k^T\btheta_k\geq 1 \text{ and } \btheta_k \in C.
\end{equation}
 The SGD algorithm does not know the value of $\btheta^*$; therefore at each iteration, it cannot check whether the condition \eqref{eq:termination_inside_C} occurs. Nevertheless, we are able to compute a bound on the average waiting time until \eqref{eq:termination_inside_C} holds and the first time \eqref{eq:termination_inside_C} holds is always an upper bound on $T$, our stopping criterion. This is summarized in Lemma \ref{lem:ETleqET_C}. Precisely, if we denote by 
 \begin{equation}\label{eq:stoppingtime_TC}
     T_C:=\inf\{k>0:\hat{\bxi}_k^T\btheta_k\geq 1 \text{ and } \btheta_k \in C\},
 \end{equation}
 then $T\leq T_C$, thus yielding $\bE[T]\leq \bE[T_C]$. We bound $\bE[T_C]$ by way of stopping times $\tau_m$ defined as the $m^{th}$ time the iterates of SGD enters $C$. Formally for any sequence $\{\btheta_k\}_{k=0}^\infty$ generated by SGD starting at $\btheta_0 = \bm{0}$, we set  
 \begin{equation}\label{eq:tau1}
     \tau_1:=\inf\{k>0:\btheta_k\in C\}
 \end{equation}
and inductively, for $m\geq 2$, 
\begin{equation}\label{eq:taum}
    \tau_m:=\inf\{k>\tau_{m-1}:\btheta_k\in C\}.
\end{equation}
The following lemma formalizes the discussion above.  
\begin{lemma}\label{lem:ETleqET_C}
Let $\{\btheta_k\}_{k=0}^\infty$ be a sequence generated by SGD such that $\btheta_0 = \bm{0}$ and suppose that $\bE[\tau_m]<+\infty$ for all $m\geq 1$. Then the following holds 
\begin{equation}
    \bE[T]\leq \bE[T_C]\leq \sum_{m=1}^{\infty} \bE[\tau_m](1-\delta)^{m-1},
\end{equation}
where $\delta$ satisfies \eqref{eq:probability_delta}.
\end{lemma}
\begin{proof}
 We first show that \begin{equation} \label{eq:analysis_1} \bE\left[1_{\{T_C\geq \tau_m\}}\right]\leq (1-\delta)^{m-1}.\end{equation} Define the $\sigma$-algebra $\mathcal{F'}=\sigma(\btheta_0,\bxi_1,\bxi_2,\cdots)$. From the independence between $\sigma(\hat{\bxi}_k)$'s  and $\mathcal{F}'$ and also $\tau_i<+\infty$ a.s. for all $i\geq 1$, the following is obtained:
\begin{align*}
    \bE\left[1_{\{ T_C\geq \tau_m\}}|\mathcal{F}'\right]
    &= \bE\left[1_{\{ \hat{\bxi}_{\tau_1}^T\btheta_{\tau_1}<1\}}\cdots1_{\{ \hat{\bxi}_{\tau_{m-1}}^T\btheta_{\tau_{m-1}}<1\}}|\mathcal{F}'\right]\\&=\prod_{i=1}^{m-1}\bE\left[1_{\{ \hat{\bxi}_{\tau_i}^T\btheta_{\tau_i}<1\}}|\mathcal{F}'\right]\\&\leq (1-\delta)^{m-1}.
\end{align*}
By taking expectations, we conclude \eqref{eq:analysis_1} holds. Now since $\bE[1_{\{T_C=+\infty\}}]\leq \bE[1_{\{T_C\geq \tau_m\}}]$ for all $m\geq 1$, it follows from \eqref{eq:analysis_1}  that $T_C<\infty$ a.s. We next observe that
\begin{align*}
    \bE\left[T_C1_{\{ T_C= \tau_m\}}|\mathcal{F}'\right]&=\bE\left[\tau_m1_{\{ T_C= \tau_m\}}|\mathcal{F}'\right]\\
    &\le\tau_m\bE\left[1_{\{ \hat{\bxi}_{\tau_{1}}^T\btheta_{\tau_{1}}<1\}}\cdots1_{\{ \hat{\bxi}_{\tau_{m-1}}^T\btheta_{\tau_{m-1}}<1\}}|\mathcal{F}'\right]\\
    &=\tau_m\prod_{i=1}^{m-1}\bE\left[1_{\{ \hat{\bxi}_{\tau_{i}}^T\btheta_{\tau_{i}}<1\}}|\mathcal{F}'\right]\\
    &\leq \tau_m(1-\delta)^{m-1}.
\end{align*}
Taking expectations yields $\bE\left[T_C1_{\{T_C=\tau_m\}}\right]\leq \bE\left[\tau_m\right](1-\delta)^{m-1}$ for all $m\geq 1$. Now since $T_C<\infty$ a.s. we get $1=\sum_{m=1}^{+\infty}1_{\{T_C=\tau_m\}}$ a.s. This yields that 
\[
    \bE[T]\leq \bE[T_C]=\sum_{m=1}^{\infty} \bE[T_C1_{\{T_C=\tau_m\}}]\leq \sum_{m=1}^{\infty}\bE[\tau_m](1-\delta)^{m-1}.
\]
The proof is complete.
\end{proof}
 Now, in view of Lemma \ref{lem:ETleqET_C}, it suffices to bound $\bE[\tau_m]$ by a sequence which can not grow too fast in $m$. Indeed, we show that \eqref{eq:driftequation} implies the following 
 \begin{equation}
     \bE[\tau_m]=\mathcal{O}(m).
 \end{equation}
\begin{theorem}(Low Regime)\label{thm:low}
Let $\{\btheta_k\}_{k=0}^\infty$ be a sequence generated by Algorithm~\ref{alg:SGD_termination} such that $\btheta_0=\bm{0}$. There exists positive constants $c,b$ and $M$ such that provided $\sigma \leq c\Vert \bmu \Vert$ the following holds. 
\begin{equation}
    \bE[T]\leq 2+\frac{2M^2}{b}\cdot\left(\Phi^c\left(\frac{\Vert \bmu \Vert}{\sigma} \right)+\frac{\alpha\sigma^3}{ \Vert \bmu \Vert }\cdot\frac{1}{\sqrt{2\pi}}\exp\left(-\frac{\Vert \bmu \Vert^2}{2\sigma^2} \right)+1 \right).
\end{equation}
Here the constants $c,b$ and $M$ are defined as follows:
\begin{enumerate}
    \item For the logistic loss, 
    \begin{equation}\label{eq:para_log}
    c=0.33,\quad b=\alpha\Vert \bmu \Vert^2,\quad \text{ and } M=501+640\alpha \Vert \bmu \Vert^2.
\end{equation}
\item For the hinge loss, 
\begin{equation}\label{eq:para_hinge}
    c=1.25,\quad b=\alpha\Vert \bmu \Vert^2,\quad \text{ and } M=501+782\alpha \Vert \bmu \Vert^2.
\end{equation}
\end{enumerate}
\end{theorem}
Therefore, on relatively separable data (\textit{i.e.} in the low variance regime), the expected waiting time before termination exponentially decreases as the data becomes more separable (\textit{i.e.} $\sigma \to 0$). 
We prove Theorem \ref{thm:low} in Section \ref{sec:theorem_angle}. The next theorem shows that the expected value of the stopping time is finite provided that the $\sigma > c\Vert \bmu \Vert$ and the step-size is small enough. 
\begin{theorem}\label{thm:high}(High Regime) Suppose that $\sigma > c\Vert \bmu \Vert$ where $c$ is defined in \eqref{eq:para_log} and \eqref{eq:para_hinge}. Then there exists a universal positive constant $A$ such that if the step-size $\alpha$ satisfies 
\begin{equation}
    \alpha \leq A\cdot\frac{\Vert \bmu \Vert^2}{\sigma^2(\Vert \bmu \Vert^2+d\sigma^2)},
\end{equation}
then it holds that $\bE[T]<+\infty$. In particular, the termination criterion occurs almost surely.
\end{theorem}
It remains to determine whether the classifier at termination $\btheta_T$, has desirable accuracy. The scale-invariance of optimal classifiers means a classifier yields a lower probability of misclassification the closer its direction aligns with any optimal classifier. In view of this, it suffices to bound the absolute value of the inner product of any unit vector that is perpendicular to $\btheta^*$, $\bm{v}$ with $\btheta_T$. The following theorem establishes a bound on $\bE[\vert \bm{v}^T\btheta_T\vert]$. 
\begin{theorem}\label{thm:angle_bound}
Let $\btheta_0=\bm{0}$. Fix any unit vector $\bm{v}\in \R^d$ such that $\bm{v}^T\btheta^*=0$. Then the following estimate holds 
\begin{equation}
    \bE[\vert \bm{v}^T\btheta_T\vert]\leq \sigma \alpha \sqrt{\frac{2}{\pi}} \bE[T].
\end{equation}
\end{theorem}
In the low variance regime by combining Theorem \ref{thm:low} and \ref{thm:angle_bound} for a fixed step-size $\alpha$ it holds that $\bE[\vert \bm{v}^T\btheta\vert]\leq \mathcal{O}(\sigma)$. Thus, the more separable the data set is, the more accurate the classifier $\btheta_T$ is on average. In the high variance regime, Theorem \ref{thm:high} yields a very loose bound. Yet despite this, our numerical result in Section \ref{sec:Num_Experiment} show promising accuracy of \eqref{eq: termination_test} in this case as well. We conjecture that the inequality can be significantly strengthened. 
\subsection{Low regime, proof of Theorem \ref{thm:low}}\label{sec:proof_low}
In this section, we investigate the low variance regime. We consider the target set $C$ and function $V$ defined in \eqref{eq:targetset_lownoise} and \eqref{eq:driftfunction_lownoise} respectively, \textit{i.e.}
\begin{equation}\label{eq:CV_low_recall}
      C = \{\btheta: \bmu^T\btheta\geq 1 \}, \quad V(\btheta)=\left(M-\bmu^T\btheta\right)^2,
\end{equation}
where $M$ is a constant to be determined. 
Next lemma shows that the drift equation \eqref{eq:driftequation} holds for the pair $(C,V)$.  \begin{lemma}[Drift equation]\label{lem:driftequation} Consider the SGD algorithm and let the set $C$ and the function $V$ be as in \eqref{eq:CV_low_recall}. Define the constants $c,b,M$ as in \eqref{eq:para_log} and \eqref{eq:para_hinge}. Then provided that $\sigma \leq c\Vert \bmu \Vert$, the function $V$ is a drift function with respect to the set $C$ and it satisfies the drift equation \eqref{eq:driftequation} with the constant $b$.
\end{lemma}
\begin{proof}For simplicity we write $\mathcal{F}_{-1}:=\sigma\left(\{\btheta_0=\btheta\} \right)$. Fix $k\geq 1$ and write $\bxi_k=\bmu+\sigma\bm{\psi}_k$ with $\bm{\psi}_k \sim N(0,I_d)$. Denote $\psi_k:=\frac{\bmu^T\bm{\psi}_k}{\Vert\bmu \Vert}$, thus $\psi_k\sim N(0,1)$. In order to show that the function $V$ satisfies the drift equation \eqref{eq:driftequation}, it suffices to assume $\btheta_{k-1}\not\in C$; in particular, this means $\btheta_{k-1}^T\bmu<1$. 
\item 
\textbf{Logistic loss.} By expanding out the term using the update formula, we get the following
\begin{align} \label{eq: low_noise_blah_1}
    V(\btheta_{k})= V(\btheta_{k-1})-\frac{2\alpha \bmu^T\bxi_k(M-\bmu^T\btheta_{k-1})}{1+\exp(\bxi_k^T\btheta_{k-1})}+\frac{\alpha^2(\bmu^T\bxi_k)^2}{(1+\exp(\bxi_k^T\btheta_{k-1}))^2}.
\end{align}
We have
\begin{align*}
 &\bE_{\bxi_k}\left[ \frac{\bmu^T\bxi_k}{1+\exp(\bxi_k^T\btheta_{k-1})}|\mathcal{F}_{k-1}\right]\\
 & \qquad =\Vert\bmu\Vert^2\bE_{\bxi_k}\left[ \frac{1}{1+\exp(\bxi_k^T\btheta_{k-1})}|\mathcal{F}_{k-1}\right]+\sigma \Vert \bmu\Vert\bE_{\bxi_k,\psi_k}\left[ \frac{\psi_k}{1+\exp(\bxi_k^T\btheta_{k-1})}|\mathcal{F}_{k-1}\right]\\
 &  \qquad \geq \Vert \bmu \Vert^2 \bE_{\bxi_k}\left[ \frac{1}{1+\exp(\bxi_k^T\btheta_{k-1})}|\mathcal{F}_{k-1}\right]+\sigma \Vert \bmu\Vert\bE_{\psi_k}\left[\psi_k1_{\{\psi_k<0\}} \right]\\
  &  \qquad = \Vert \bmu \Vert^2 \bE_{\bxi_k}\left[ \frac{1}{1+\exp(\bxi_k^T\btheta_{k-1})} \left ( 1_{\{\bmu^T \btheta_{k-1} \geq  \bxi_k^T \btheta_{k-1}\}} + 1_{\{\bmu^T \btheta_{k-1} < \bxi_k^T \btheta_{k-1}\}} \right ) |\mathcal{F}_{k-1}\right] -\sigma \Vert \bmu\Vert \sqrt{\frac{1}{2\pi}}\\
  &  \qquad \geq \frac{\Vert \bmu\Vert^2}{1+\exp(\bmu^T\btheta_{k-1})}\bE_{\bxi_k}\left[ 1_{\{\bmu^T \btheta_{k-1} \geq  \bxi_k^T \btheta_{k-1}\}} |\mathcal{F}_{k-1} \right]-\sigma \Vert \bmu\Vert \sqrt{\frac{1}{2\pi}}\\
 &\qquad \geq \frac{\Vert \bmu\Vert^2}{2(1+e)}-\sigma \Vert \bmu\Vert \sqrt{\frac{1}{2\pi}}\\
 & \qquad \geq 0.001\Vert \bmu \Vert^2 .
\end{align*}
Here the first inequality follows from $\bE[X] \ge \bE[X1_{\{X<0\}}]$ and $1+ \exp(\bxi_k^T\btheta_{k-1}) \ge 1$, the second equation from \eqref{fact:norm_Gaussians}, and the second to last from the observation that for any $X$ normally distributed, $\bP(\bE[X] \ge X) = 1/2$ and $\bxi_k^T\btheta_{k-1} \sim N(\bmu^T\btheta_{k-1}, \sigma^2 \norm{\btheta_{k-1}}^2)$ and $\bmu^T\btheta_{k-1} < 1$. The last inequality uses the assumption $\sigma \le 0.33 \norm{\bmu}$. By taking the conditional expectations of \eqref{eq: low_noise_blah_1} combined with the above sequence of inequalities, we deduce the following bound
\begin{align*}
    &\bE\left[V(\btheta_{k})-V(\btheta_{k-1})|\mathcal{F}_{k-1}\right]\\
    & \quad \qquad = \bE_{\bxi_k} \left [-\frac{2\alpha \bmu^T\bxi_k(M-\bmu^T\btheta_{k-1})}{1+\exp(\bxi_k^T\btheta_{k-1})} | \mathcal{F}_{k-1} \right]  + \bE_{\bxi_k} \left [\frac{\alpha^2(\bmu^T\bxi_k)^2}{(1+\exp(\bxi_k^T\btheta_{k-1}))^2} | \mathcal{F}_{k-1} \right]\\
    & \quad \qquad \leq -0.002(M-1)\alpha \Vert \bmu \Vert^2+ \alpha^2\Vert \bmu \Vert^2\left(\Vert \bmu\Vert^2+\sigma^2 \right)\\
    & \quad \qquad = \alpha \Vert \bmu \Vert^2\left[ -0.002(M-1)+\alpha\left(\Vert \bmu \Vert^2+\sigma^2\right)\right].
\end{align*}
Here the first inequality follows from $\bmu^T\btheta_{k-1} < 1$ and by upper bounding $\frac{(\bmu^T\bxi_k)^2}{(1+\exp(\bxi_k^T\btheta_{k-1}))^2}$ with $(\bmu^T\bxi_k)^2$ and then applying \eqref{fact:norm_Gaussians}.
A quick computation after plugging in the value of $M$ and the bound $\sigma\leq 0.33\Vert \bmu \Vert$ from \eqref{eq:para_log} yields the drift equation \eqref{eq:driftequation} with $b=\alpha \Vert \bmu \Vert^2$.
\item \textbf{Hinge loss.} By expanding out the term using the update formula, we get the following
\begin{equation}\label{eq:Vequation_hinge}
   V(\btheta_k)=V(\btheta_{k-1})-2\alpha (M-\bmu^T\btheta_{k-1})\bmu^T\bxi_k1_{\{\bxi_k^T\btheta_{k-1}\leq 1\}}+\alpha^2(\bmu^T\bxi_k)^21_{\{\bxi_k^T\btheta_{k-1}\leq 1\}}.
\end{equation}
We have
\begin{align*}
    \bE_{\bxi_k}[1_{\{ \bxi_k^T\btheta_{k-1}\leq 1\}}\bmu^T\bxi_k|\mathcal{F}_{k-1}] &=\Vert \bmu\Vert^2\bE_{\bxi_k}[1_{\{ \bxi_k^T\btheta_{k-1}\leq 1\}}|\mathcal{F}_{k-1}]+\sigma\Vert \bmu \Vert\bE_{\bxi_k,\psi_k}[1_{\{ \bxi_k^T\btheta_{k-1}\leq 1\}}\psi_k|\mathcal{F}_{k-1}]\\ 
    & \geq \frac{1}{2}\Vert \bmu \Vert^2+\sigma\Vert \bmu \Vert\bE_{\psi_k}[\psi_k1_{\{\psi_k<0 \}}]\\
    & = \frac{1}{2}\Vert \bmu \Vert^2-\sigma\Vert \bmu \Vert\sqrt{\frac{1}{2\pi}} \\&
    \geq 
    0.001 \Vert \bmu \Vert^2.
\end{align*}
Here the first inequality follows from $1_{\{\bxi_k^T\btheta_{k-1}\leq \bmu^T\btheta_{k-1}\}}\leq 1_{\{\bxi_k^T\btheta_{k-1}\leq 1\}}$ and  $\bE_{\bxi_k}[1_{\{\bxi_k^T\btheta_{k-1}\leq \bmu^T\btheta_{k-1}\}}]=\tfrac{1}{2}$, and the second from \eqref{fact:norm_Gaussians}. The last inequality uses the assumption $\sigma \leq 1.25 \Vert \bmu \Vert$. By taking conditional expectations of \eqref{eq:Vequation_hinge} combined with the above sequence of inequalities, we deduce the bound
\begin{equation*}
\begin{aligned}
\bE[V(\btheta_k)-V(\btheta_{k-1})|\mathcal{F}_{k-1}]&=\bE_{\bxi_k}\left[-2\alpha(M-\bmu^T\btheta_{k-1})1_{\{\bxi_k^T\btheta_{k-1}\leq 1\}}|\mathcal{F}_{k-1}\right]+\bE_{\bxi_k}\left[\alpha^2(\bmu^T\bxi_k)^21_{\{\bxi_k^T\btheta_{k-1}\leq 1\}}|\mathcal{F}_{k-1}\right]
\\
&\leq \alpha \Vert \bmu \Vert^2\left[ -0.002(M-1)+\alpha\left(\Vert \bmu \Vert^2+\sigma^2\right)\right].
\end{aligned}
\end{equation*}
A quick computation after plugging in the value of $M$ and the bound $\sigma \leq 1.25 \Vert \bmu \Vert$ yields the desired result.
\end{proof}
 Recall, the stopping times $\tau_m$ denote the $m^{th}$ time that the SGD iterates enter the target set $C$. We show that $\bE[\tau_m]=\mathcal{O}(m)$. To do so, we begin by stating a lemma that gives a bound on the stopping time $\tilde{\tau}_1$ starting from any $\btheta_0$. In other words, for an arbitrary starting $\btheta_0$, we define 
 \[
 \tilde{\tau}_1:=\inf\{k>0:\btheta_k\in C\}.
 \]
\begin{lemma}[\cite{meyn2012markov}, Theorem 11.3.4]\label{lem:drift_from_meyn} Suppose that $V:\R^d\rightarrow [0,+\infty)$ is a drift function with respect to some target set $C$ \textit{i.e.} for some constant $b\in (0,+\infty)$ the drift equation \eqref{eq:driftequation} holds. The following is true
\begin{equation}
    \bE[\tilde{\tau}_1|\btheta_0=\btheta]\leq \tfrac{1}{b}V(\btheta).
\end{equation}
\end{lemma}
We establish upper bounds on $\bE[\tau_m]$ for $m\geq 1$ in the following proposition. 
\begin{proposition}(Bound on $\bE[\tau_m]$)\label{prp:Etaum_lownoise}
Let $\btheta_0=\bm{0}$ and assume the notation and assumptions of Lemma \ref{lem:driftequation} hold. The following is true for all $m\geq 1$
\begin{align}
    \bE[\tau_m]\leq (m-1) \left ( 1+\frac{M^2}{b}\cdot\Phi^c\left(\frac{\Vert \bmu \Vert}{\sigma} \right)+\frac{\alpha\sigma^3 M^2}{ \Vert \bmu \Vert b}\cdot\frac{1}{\sqrt{2\pi}}\exp\left(-\frac{\Vert \bmu \Vert^2}{2\sigma^2} \right) \right ) +\frac{M^2}{b}.
\end{align}
\end{proposition}
\begin{proof}First, the result for $m=1$ follows immediately by combining Lemmas \ref{lem:driftequation} and \ref{lem:drift_from_meyn} with $\btheta_0 = \bm{0}$. We now assume that $\tau_{m-1}<\infty$ a.s. for some $m\geq 2$. Fix an integer $n \ge 1$. We decompose the space to yield the following bounds

 \begin{equation} \begin{aligned} \label{eq: low_noise_bound}
    \bE\big[(\tau_m-\tau_{m-1})\wedge& n|\mathcal{F}_{\tau_{m-1}+1}\big]=\bE\left[((\tau_m-\tau_{m-1})\wedge n)|\mathcal{F}_{\tau_{m-1}+1}\right]1_{\{\bmu^T\btheta_{\tau_{m-1}+1}\geq 1\}}\\
    & \qquad \qquad \qquad \qquad  \qquad +\bE\left[((\tau_m-\tau_{m-1})\wedge n)|\mathcal{F}_{\tau_{m-1}+1}\right]1_{\{\bmu^T\btheta_{\tau_{m-1}+1}<1\}}\\
    &= 1_{\{\bmu^T\btheta_{\tau_{m-1}+1}\geq 1\}}+ \bE\left[((\tau_m-\tau_{m-1})\wedge n)|\mathcal{F}_{\tau_{m-1}+1}\right]1_{\{\bmu^T\btheta_{\tau_{m-1}+1}<1\}}\\
    &=1_{\{\bmu^T\btheta_{\tau_{m-1}+1}\geq 1\}}+\sum_{i=1}^{\infty}\bE\left[(\tau_m-\tau_{m-1})\wedge n|\mathcal{F}_{\tau_{m-1}+1}\right]1_{\{i-1< 1-\bmu^T\btheta_{\tau_{m-1}+1}\leq i\}}\\
&= 1+\sum_{i=1}^{\infty}\bE\left[\tilde{\tau}_1\wedge n|\btheta_0=\btheta_{\tau_{m-1}+1}\right]1_{\{i-1< 1-\bmu^T\btheta_{\tau_{m-1}+1}\leq i\}}.
\end{aligned} \end{equation}
Here the first equality follows because $((\tau_m-\tau_{m-1})\wedge n) 1_{\{\bmu^T\btheta_{\tau_{m-1}+1} \ge 1\}} = 1_{\{\bmu^T\btheta_{\tau_{m-1}+1}\geq 1\}}$ and the last equality by the strong Markov property.  We consider the logistic and hinge loss case separately to show that the following is true
\begin{equation}\label{eq:ineq:indicators}
\begin{aligned}
1_{\{i-1< 1-\bmu^T\btheta_{\tau_{m-1}+1}\leq i\}} &\leq 1_{\{\bmu^T\bxi_{\tau_{m-1}+1}<\frac{1-i}{\alpha}\}}.
\end{aligned}
\end{equation}
For clarity, in the next few inequalities, we write $1\{.\}$ instead of $1_{\{.\}}$. In case of logistic loss, 
for each $i \ge 1$, we observe the bound
\begin{equation}
\begin{aligned}
 1\{i-1< 1-\bmu^T\btheta_{\tau_{m-1}+1}\leq i\}&\leq 1\{i-1< 1-\bmu^T\btheta_{\tau_{m-1}+1}\}\nonumber\\&=1\left\{i-1< 1-\bmu^T\btheta_{\tau_{m-1}}-\frac{\alpha \bmu^T\bxi_{{\tau_{m-1}}+1}}{1+\exp(\bxi_{{\tau_{m-1}}+1}^T\btheta_{\tau_{m-1}})}\right\}\nonumber\\
    &\leq 1\left\{i-1<-\frac{\alpha \bmu^T\bxi_{\tau_{m-1}+1}}{1+\exp(\bxi_{\tau_{m-1}+1}^T\btheta_{\tau_{m-1}})}\right\} \label{eq:i-1_calculus}\nonumber\\
    & \leq 1\left\{i-1<-\alpha \bmu^T\bxi_{\tau_{m-1}+1}\right\},\nonumber
\end{aligned}
\end{equation}
where the second inequality follows because $\bmu^T\btheta_{\tau_{m-1}} \geq 1$ and the last inequality because $-\alpha \bmu^T \bxi_{\tau_{m-1}+1}$ is positive since $i-1\geq 0$.
 \item In case of hinge loss, for each $i\geq 1$, similar as above, we observe the bound
 \begin{equation}
     \begin{aligned}
     1\left\{i-1<1-\bmu^T\btheta_{\tau_{m-1}+1}\leq i\right\}&\leq 1\left\{i-1<1-\bmu^T\btheta_{\tau_{m-1}+1}\right\}\\&\leq 1\left\{i-1<1-\bmu^T\btheta_{\tau_{m-1}}-\alpha \bmu^T\bxi_{\tau_{m-1}+1}1_{\{\bxi_{\tau_{m-1}+1}^T\btheta_{\tau_{m-1}}\leq 1\}}\right\}\\&\leq 1\left\{i-1<-\alpha \bmu^T\bxi_{\tau_{m-1}+1}1_{\{\bxi_{\tau_{m-1}+1}^T\btheta_{\tau_{m-1}}\leq 1\}} \right\}\\&= 1\left\{i-1<-\alpha\bmu^T\bxi_{\tau_{m-1}+1}\right\}.
     \end{aligned}
 \end{equation}
 Therefore we have shown that \eqref{eq:ineq:indicators} holds. Setting $\btheta_0=\btheta_{\tau_{m-1}+1}$ , by Lemma \ref{lem:drift_from_meyn} for each $i\geq 1$, we deduce 
\begin{equation} \begin{aligned}  \label{eq: low_noise_blah_3}
\bE\left[\tilde{\tau}_1\wedge n|\btheta_0=\btheta_{\tau_{m-1}+1}\right]1_{\{i-1< 1-\bmu^T\btheta_{\tau_{m-1}+1}\leq i\}}&\leq  \frac{(M-\bmu^T\btheta_{\tau_{m-1}+1})^2}{b}1_{\{i-1< 1-\bmu^T\btheta_{\tau_{m-1}+1}\leq i\}}\\&\leq  \frac{(M+i-1)^2}{b}1_{\{ \bmu^T\bxi_{\tau_{m-1}+1}< \frac{1-i}{\alpha}\}}.
\end{aligned} \end{equation}
Finally we observe that 
\begin{equation} \begin{aligned} \label{eq: low_noise_blah_4}
    \bE\left[1_{\{ \bmu^T\bxi_{\tau_{m-1}+1}< \frac{1-i}{\alpha}\}} \right]&=\bE\left[\sum_{k=1}^{\infty}1_{\{ \bmu^T\bxi_{k+1}< \frac{1-i}{\alpha}\}}1_{\{\tau_{m-1}=k\}}\right]\\
&=\sum_{k=1}^{\infty}\bE\left[1_{\{ \bmu^T\bxi_{k+1}< \frac{1-i}{\alpha}\}}  \right]\bE\left[1_{\{\tau_{m-1}=k\}} \right]\\
    &= \Phi\left(\frac{\frac{1-i}{\alpha}-\Vert \bmu\Vert^2}{\sigma\Vert \bmu\Vert } \right)\sum_{k=1}^{\infty}\bE\left[1_{\{\tau_{m-1}=k\}} \right]\\
    &=\Phi\left(\frac{\frac{1-i}{\alpha}-\Vert \bmu\Vert^2}{\sigma\Vert \bmu\Vert } \right).
\end{aligned} \end{equation}
The second equality is by independence and the third equality because $\bmu^T \bxi_{k+1} \sim N(\norm{\bmu}^2, \sigma^2 \norm{\bmu}^2)$.  By combining  \eqref{eq: low_noise_bound}, \eqref{eq: low_noise_blah_3}, and \eqref{eq: low_noise_blah_4}, we obtain the following
\begin{equation} \begin{aligned} \label{eq: low_noise_blah_5}
    \bE\big[(\tau_m-&\tau_{m-1})\wedge n\big]\leq 1+\frac{M^2}{b}\cdot\Phi\left(-\frac{\Vert \bmu \Vert}{\sigma} \right)+\sum_{i=2}^{\infty} \frac{(M+i-1)^2}{b}\cdot\Phi\left(\frac{\frac{1-i}{\alpha}-\Vert \bmu\Vert^2}{\sigma\Vert \bmu\Vert } \right)\\
    &=1+\frac{M^2}{b}\cdot\Phi^c\left(\frac{\Vert \bmu \Vert}{\sigma} \right)+\sum_{i=2}^{\infty} \frac{(M+i-1)^2}{b}\cdot\Phi^{c}\left(\frac{\Vert \bmu\Vert^2+\frac{i-1}{\alpha}}{\sigma\Vert \bmu\Vert } \right)\\
    &\le1+\frac{M^2}{b}\cdot\Phi^c\left(\frac{\Vert \bmu \Vert}{\sigma} \right)+\frac{\alpha\sigma\Vert \bmu \Vert}{b\sqrt{2\pi}}\cdot\sum_{i=2}^{\infty} \frac{(M+i-1)^2}{\alpha\Vert \bmu\Vert^2+i-1}\cdot\exp\left(-\frac{1}{2}\left(\frac{\Vert \bmu\Vert^2+\frac{i-1}{\alpha}}{\sigma\Vert \bmu\Vert}\right)^2 \right),
\end{aligned} \end{equation}
where we used the inequality $\Phi^{c}(t)<\frac{1}{t\sqrt{2\pi}}\exp(-\frac{t^2}{2})$ for all $t>0$. Next, note that $\frac{M+i-1}{\alpha \Vert \bmu \Vert^2+i-1}\leq \frac{M}{\alpha \Vert\bmu\Vert^2}$ holds for all $i\geq 2$.  Using this we obtain the following bound 
\begin{equation} \begin{aligned}\label{eq:exp_low_noise}
       \sum_{i=2}^{\infty} \frac{(M+i-1)^2}{\alpha\Vert \bmu\Vert^2+i-1}\cdot\exp&\left(-\frac{1}{2}\left(\frac{\Vert \bmu\Vert^2+\frac{i-1}{\alpha}}{\sigma\Vert \bmu\Vert}\right)^2 \right)\\
       &\leq \frac{\sigma M^2}{\alpha \Vert \bmu \Vert^3}\cdot\sum_{i=2}^{\infty} \frac{\alpha\Vert \bmu\Vert^2+i-1}{\alpha \sigma \Vert \bmu \Vert}\cdot\exp\left(-\frac{1}{2}\left(\frac{\alpha\Vert \bmu\Vert^2+i-1}{\alpha\sigma\Vert \bmu\Vert}\right)^2 \right)\\&\leq \frac{\sigma M^2}{\alpha \Vert \bmu \Vert^3}\cdot \alpha\sigma \Vert \bmu \Vert\cdot\int_{\frac{\Vert \bmu \Vert}{\sigma}}^{+\infty} t\exp\left(-\frac{t^2}{2}\right)dt \\&=\frac{\sigma^2 M^2}{\Vert \bmu \Vert^2}\cdot\exp\left(-\frac{\Vert \bmu \Vert^2}{2\sigma^2} \right).
\end{aligned} \end{equation}
Here we have used that $t \mapsto t\exp(-\frac{t^2}{2})$ is decreasing over $[1,+\infty)$. Combining \eqref{eq: low_noise_blah_5} and \eqref{eq:exp_low_noise}, we obtain that 
\begin{equation}\begin{aligned}
 \bE\left[(\tau_m-\tau_{m-1})\wedge n\right]&\leq 1+\frac{M^2}{b}\cdot\Phi^c\left(\frac{\Vert \bmu \Vert}{\sigma} \right)+\frac{\alpha\sigma^3 M^2}{ \Vert \bmu \Vert b}\cdot\frac{1}{\sqrt{2\pi}}\exp\left(-\frac{\Vert \bmu \Vert^2}{2\sigma^2} \right).   
\end{aligned}
\end{equation}
Taking the limit as $n \to +\infty$, we observe that
\[\bE[\tau_m]\leq 1+\frac{M^2}{b}\cdot\Phi^c\left(\frac{\Vert \bmu \Vert}{\sigma} \right)+\frac{\alpha\sigma^3 M^2}{ \Vert \bmu \Vert b}\cdot\frac{1}{\sqrt{2\pi}}\exp\left(-\frac{\Vert \bmu \Vert^2}{2\sigma^2} \right)+\bE[\tau_{m-1}].
\]
We then iterate the above inequality yielding
\[ \bE[\tau_m] \le (m-1) \left ( 1+\frac{M^2}{b}\cdot\Phi^c\left(\frac{\Vert \bmu \Vert}{\sigma} \right)+\frac{\alpha\sigma^3 M^2}{ \Vert \bmu \Vert b}\cdot\frac{1}{\sqrt{2\pi}}\exp\left(-\frac{\Vert \bmu \Vert^2}{2\sigma^2} \right)\right ) + \bE[\tau_1]. \]
The result follows by plugging in the bound from Lemma \ref{lem:drift_from_meyn} for the base case $m=1$.
\end{proof}
We are now ready to prove Theorem~\ref{thm:low}.
\begin{proof}[Proof of Theorem
\ref{thm:low}]
In order to simplify the subsequent argument, we define the quantity, 
\[ M' := 1+\frac{M^2}{b}\cdot\Phi^c\left(\frac{\Vert \bmu \Vert}{\sigma} \right)+\frac{\alpha\sigma^3 M^2}{ \Vert \bmu \Vert b}\cdot\frac{1}{\sqrt{2\pi}}\exp\left(-\frac{\Vert \bmu \Vert^2}{2\sigma^2} \right).  \]
It is easy to see that $\mathbb{P}_{\hat{\bxi}\sim N(\bmu,\sigma^2I_d)}\left( \hat{\bxi}^T\btheta\geq 1 \right) \geq \frac{1}{2}$ for any $\btheta\in C$. Therefore $\delta=\frac{1}{2}$ satisfies \eqref{eq:probability_delta}. By Proposition~\ref{prp:Etaum_lownoise} with Lemma \ref{lem:ETleqET_C}, we conclude that 
\begin{align*}
\bE[T]&\leq \bE[T_C]=\sum_{m=1}^{\infty} \bE[T_C1_{\{T_C=\tau_m\}}]\leq \sum_{m=1}^{\infty}\frac{\bE[\tau_m]}{2^{m-1}} \le \sum_{m=1}^\infty \frac{(m-1)M' + \frac{M^2}{b} }{2^{m-1}} = 2M' + \frac{2M^2}{b}.
\end{align*}
\end{proof}

\subsection{High regime, proof of Theorem \ref{thm:high}}\label{sec:proof_high}
In this section, we consider the high variance regime. We consider the target set $C$ and the function $V$ defined in \eqref{eq:targetset_highnoise} and \eqref{eq:driftfunction_highnoise}, respectively, \textit{i.e.}
\begin{equation}\label{eq:CV_high_recall}
    C:=\left\{\btheta: \vert \rho -\rho^*\vert<\tfrac{1}{2}\rho^* \text{ and } \sigma \Vert \tilde{\btheta}\Vert \leq c'\right\} \quad \text{and}\quad V(\btheta):=\frac{1}{2\alpha}\Vert \btheta-\btheta^*\Vert^2,
\end{equation}
where the minimizer $\btheta^*=\rho^*\bmu$ is defined in Lemma \ref{lem:minimizers} and the constant $c'$ is to be determined. We first aim to show that $V$ is a drift function with respect to the set $C$ under the high variance regime assumption, meaning $\sigma \geq c\Vert \bmu \Vert$. We next state a standard SGD convergence result applied to the logistic and hinge loss functions. 
\begin{lemma}\label{lem:technical_convex_bound}
Consider the optimization problem \eqref{optimization_problem} where $\ell:\R\times \R \rightarrow \R$ is either the logistic or hinge loss function. Denote the vector $\btheta^*$ as the unique minimizer of $f$ in \eqref{optimization_problem}. Let $\btheta_0\in \R^d$. The sequence $\{\bm{\theta}_k \}_{k = 0}^\infty$ generated by SGD satisfies the following for all $k\geq 1$,
    \begin{equation} \label{eq:high_convex_tech_lemma}
       f(\bm{\theta}_{k-1})-f(\bm{\theta}^*)\leq \frac{1}{2\alpha}\left(\Vert  \bm{\theta}_{k-1}-\bm{\theta} ^*\Vert^2-\bE\left[\Vert \bm{\theta}_{k}-\bm{\theta}^* \Vert^2\, |\mathcal{F}_{k-1}\right] \right)+\frac{\alpha}{2}\left(\Vert \bmu \Vert^2+d\sigma^2\right).
    \end{equation}

\end{lemma}
\begin{proof}
 Define the quantity 
  \begin{equation*}
      \bm{g_k}:=\frac{1}{\alpha}\left( \bm{\theta}_{k-1}-\bm{\theta}_{k}\right) = \nabla_{\btheta}\ell\left(\bxi_k^T\btheta_{k-1},1 \right). \end{equation*}
Here, it is easy to check that the derivative with respect to $\btheta$ and the expectation over $\bxi_k$ are interchangeable, thus yielding      \[
      \bE_{\bxi_{k}}\left[\bm{g}_k|\mathcal{F}_{k-1} \right]=\nabla_{\btheta} f(\btheta_{k-1}).
      \]
By convexity of the function $f$, we have the following
\begin{align*}
\norm{\bm{\theta}_{k}-\bm{\theta}^*}^2 &= \norm{\bm{\theta}_{k-1}-\bm{\theta}^*}^2 -2 \alpha \bm{g}_k^T (\bm{\theta}_{k-1}-\bm{\theta}^*) + \alpha^2\norm{\bm{g}_k}^2\\
&= \norm{\bm{\theta}_{k-1}-\bm{\theta}^*}^2 - 2\alpha(\bm{g}_k-\bE_{\bxi_{k}}\left[\bm{g}_k|\mathcal{F}_{k-1} \right])^T(\bm{\theta}_{k-1}-\bm{\theta}^*) - 2 \alpha \bE_{\bxi_{k}}\left[\bm{g}_k|\mathcal{F}_{k-1} \right]^T (\bm{\theta}_{k-1}-\bm{\theta}^*)+\alpha^2 \norm{\bm{g}_k}^2\\
&\le \norm{\bm{\theta}_{k-1}-\bm{\theta}^*}^2- 2\alpha(\bm{g}_k-\bE_{\bxi_{k}}\left[\bm{g}_k|\mathcal{F}_{k-1} \right])^T(\bm{\theta}_{k-1}-\bm{\theta}^*)-2\alpha (f(\bm{\theta}_{k-1})-f(\bm{\theta}^*)) + \alpha^2 \norm{\bm{g}_k}^2.
  \end{align*}
  By taking conditional expectations with respect to $\mathcal{F}_{k-1}$ and rearranging the above inequality, we obtain that 
    \begin{equation}\label{eq:label_ineq_blah1}
       f(\bm{\theta}_{k-1})-f(\bm{\theta}^*)\leq \frac{1}{2\alpha}\left(\Vert  \bm{\theta}_{k-1}-\bm{\theta} ^*\Vert^2-\bE\left[\Vert \bm{\theta}_{k}-\bm{\theta}^* \Vert^2\, |\mathcal{F}_{k-1}\right] \right)+\frac{\alpha}{2}\bE_{\bxi_k}\left[\Vert \nabla_{\btheta} \ell(\bxi_k^T\btheta_{k-1},1)\Vert^2 \right].
    \end{equation}
    We next observe the following bound 
    \begin{equation}\label{eq:label_ineq_blah2}
        \bE_{\bxi_k}[\Vert\nabla_{\btheta}\ell\left( \bxi_k^T\btheta_{k-1},1\right)\Vert^2 \, | \, \mathcal{F}_{k-1}]\leq \bE_{\bxi_k}[\Vert \bxi_k \Vert^2|\mathcal{F}_{k-1}]=\Vert \bmu \Vert^2+d\sigma^2.
\end{equation}
Combining \eqref{eq:label_ineq_blah1} and \eqref{eq:label_ineq_blah2}, the result follows. 
\end{proof}
By Lemma \ref{lem:technical_convex_bound} for each $k\geq 1$, we deduce 
\begin{equation}\label{eq:driftequation_high}
\bE[V(\btheta_k)|\mathcal{F}_{k-1}]-V(\btheta_{k-1})\leq -\left(f(\btheta_{k-1})-f(\btheta^*)\right)+\frac{\alpha}{2}(\Vert\bmu\Vert^2+d\sigma^2).
\end{equation}
Therefore, in order to show that the pair $(C,V)$ in \eqref{eq:CV_high_recall} satisfies the drift equation \eqref{eq:driftequation}, it suffices to lower bound the quantity $f(\btheta_{k-1})-f(\btheta^*)$ whenever $\btheta_{k-1}\not\in C$. To do so, we orthogonally decompose $\btheta_{k-1}=\rho_{k-1}\bmu+\tilde{\btheta}_{k-1}$, \textit{i.e.} $\bmu^T\tilde{\btheta}_{k-1}=0$ and $\rho_{k-1}\in \R$ and write 
\begin{equation}\label{eq:quantity_split}
\begin{aligned}
f(\btheta_{k-1})-f(\btheta^*)= & \underbrace{f(\btheta_{k-1})-f(\rho_{k-1}\bmu)}_{(a)}+\underbrace{f(\rho_{k-1}\bmu)-f(\btheta^*)}_{(b)}.
\end{aligned}
\end{equation}
The assumption $\btheta_{k-1}\not \in C$ yields that either $\sigma \Vert \tilde{\btheta}_{k-1}\Vert \geq c'$ or $\vert \rho_{k-1}-\rho^*\vert\geq \tfrac{1}{2}\rho^*$. In Lemma \ref{lem:lower_bound_1} (resp. \ref{lem:lower_bound_2}), we show that (a) (resp. (b)) in \eqref{eq:quantity_split} are both non-negative and they are lower bounded by some positive constant provided that $\sigma \Vert \tilde{\btheta}_{k-1}\Vert \geq c'$ and $\vert \rho_{k-1}-\rho^*\vert\leq \tfrac{1}{2}\rho^*$ (resp. $\vert \rho_{k-1}-\rho^*\vert\geq \tfrac{1}{2}\rho^*$).
\begin{lemma}(Lower bound for (a) in \eqref{eq:quantity_split})\label{lem:lower_bound_1}
Fix $\btheta\in \R^d$ and orthogonally decompose $\btheta=\rho\bmu+\tilde{\btheta}$ where $\bmu^T\tilde{\btheta}=0$ and $\rho \in \R$.  Then the following are true
\begin{enumerate}
    \item  $f(\btheta)-f(\rho\bmu)\geq 0$.
    \item  $f(\btheta)-f(\rho\bmu)\geq 1$ provided that $\vert \rho-\rho^*\vert\leq \tfrac{1}{2}\rho^*$, $\sigma \Vert \tilde{\btheta}\Vert \geq c'$ and $\sigma \geq c\Vert \bmu \Vert$ where $c$ is defined in \eqref{eq:para_log} and \eqref{eq:para_hinge}. Here $\rho^*$ is defined in Lemma \ref{lem:minimizers} and the constant $c'$ is defined by $436$ and $8+10\rho^*\sigma^2$ for the logistic and hinge loss respectively. 
\end{enumerate}
\end{lemma}

\begin{proof} We consider the logistic and hinge loss separately. 
\item
\textbf{Logistic loss.} The two normal random variables, $\bm{\tilde{\theta}}^T\bm{\xi} \sim N(0,\sigma^2\Vert \bm{\tilde{\theta}} \Vert^2)$ and $\bmu^T\bxi \sim N(\Vert \bmu \Vert^2,\sigma^2\Vert \bmu \Vert^2)$, are independent by \eqref{eqn:fact_independence}. Since we have $\bE_{\bxi}[\log(\exp(-\bm{\tilde{\theta}}^T\bxi))]=\bE_{\bxi}[\log(\exp(\bm{\tilde{\theta}}^T\bxi))]=0$, it holds
\begin{align*}
    f(\btheta)=\bE_{\bxi} \left[ \log\left (1+\exp(-\btheta^T\bxi)\right)\right] &= \bE_{\bxi} \left[\log \left(1 + \exp(-\tilde{\btheta}^T\bxi) \exp(-\rho \bmu^T\bxi) \right ) \right ]\\
    &=\bE_{\bxi} \left[ \log\left (\exp(\bm{\tilde{\theta}}^T\bxi)+\exp(-\rho \bmu^T\bxi)\right)\right]\\
    &=\bE_{\bxi} \left[ \log\left (\exp(-\bm{\tilde{\theta}}^T\bxi)+\exp(-\rho \bmu^T\bxi)\right)\right],
\end{align*}
where the last equality is true because $\bm{\tilde{\theta}}^T\bxi \sim -\bm{\tilde{\theta}}^T\bxi$. Therefore we obtain
\begin{align*}
  \bE_{\bxi} &\left[ \log\left (1+\exp(-\btheta^T\bxi)\right)\right]\\
  & \qquad \qquad= \frac{1}{2} \bE_{\bxi} \left[ \log\left (\exp(\bm{\tilde{\theta}}^T\bxi)+\exp(-\rho \bmu^T\bxi)\right)\right] + \frac{1}{2}\bE_{\bxi} \left[ \log\left (\exp(-\bm{\tilde{\theta}}^T\bxi)+\exp(-\rho \bmu^T\bxi)\right)\right] \\
  & \qquad \qquad= \frac{1}{2} \bE_{\bxi} \left[ \log\left ( (\exp(\bm{\tilde{\theta}}^T\bxi)+\exp(-\rho \bmu^T\bxi))(\exp(-\bm{\tilde{\theta}}^T\bxi)+\exp(-\rho \bmu^T\bxi)) \right)\right]\\ 
  &\qquad \qquad=\frac{1}{2}\bE_{\bxi} \left[\log\left(1+\exp(-\bm{\tilde{\theta}}^T\bxi-\rho \bmu^T\bxi)+\exp(\bm{\tilde{\theta}}^T\bxi-\rho \bmu^T\bxi)+\exp(-2\rho \bmu^T\bxi)\right) \right].
  \end{align*}
By the equality $\exp(\bm{\tilde{\theta}}^T\bxi)+\exp(-\bm{\tilde{\theta}}^T\bxi)= 2+4\sinh^2(\tfrac{\tilde{\btheta}^T\bxi}{2})$, we have
\begin{align*}
   \bE_{\bxi} &\left[ \log\left (1+\exp(-\btheta^T\bxi)\right)\right]\\
    & \qquad \qquad \qquad =\frac{1}{2}\bE_{\bxi} \left[\log\left(1+2\exp(-\rho \bmu^T\bxi)+\exp(-2\rho \bmu^T\bxi)+4\sinh^2(\tfrac{\tilde{\btheta}^T\bxi}{2})\exp(-\rho \bmu^T\bxi)\right) \right].
\end{align*}
Therefore, we have
\begin{equation} \begin{aligned} \label{eqn: high_noise_blah_20}
     2\bE_{\bxi}\left[\log\left( \frac{1+\exp(-\btheta^T\bxi)}{1+\exp(-\rho \bmu^T\bxi)}\right)\right]&=2\bE_{\bxi}\left[\log(1+\exp(-\btheta^T\bxi))\right]-\bE_{\bxi}\left[\log\left(1+\exp(-\rho \bmu^T\bxi)\right)^2\right]\\
     \qquad \qquad &= \bE_{\bxi}\left[\log \left ( 1 +  \frac{4\sinh^2(\frac{\tilde{\btheta}^T\bxi}{2})\exp(-\rho \bmu^T\bxi)}{(1+ \exp(-\rho \bmu^T\bxi))^2} \right ) \right] \ge 0.
     \end{aligned} \end{equation}
     Thereby, we showed that $f(\btheta)-f(\rho \bmu)\geq 0$. Now we establish the positive lower bound. First, we note the following $1+ \exp(-\rho \bmu^T\bxi) = 2 \exp( -\tfrac{\rho \bmu^T\bxi}{2}) \cosh(\tfrac{\rho \bmu^T\bxi}{2})$. Fix a  constant $r > 0$ and consider the set $\{\bxi: |\btheta^T\bxi| > r\}$. Applying the inequality $x^2+y^2\geq 2\vert xy \vert$ and \eqref{eqn: high_noise_blah_20}, we obtain that 
     \begin{align}
   2\bE_{\bxi}\left[\log\left( \frac{1+\exp(-\btheta^T\bxi)}{1+\exp(-\rho \bmu^T\bxi)}\right)\right]  & = \bE_{\bxi}\left[\log\left(1 + \frac{4\sinh^2(\frac{\tilde{\btheta}^T\bxi}{2})\exp(-\rho \bmu^T\bxi)}{(1+ \exp(-\rho \bmu^T\bxi))^2}\right) \right] \nonumber \\ 
     &= \bE_{\bxi}\left[\log\left(1+\frac{\sinh^2(\frac{\tilde{\btheta}^T\bxi}{2})}{\cosh^2(\frac{\rho}{2}\bmu^T\bxi)} \right) \right] \nonumber \\ 
     &\geq \bE_{\bxi}\left[\log\left(1+\frac{\sinh^2(\frac{\tilde{\btheta}^T\bxi}{2})}{\cosh^2(\frac{\rho}{2}\bmu^T\bxi)} \right)\cdot1_{\{\bxi:\vert \tilde{\btheta}^T\bxi\vert \geq r\}}\right] \label{eq:high_noise_blah_21}
     \\
     &\geq \bE_{\bxi}\left[\left(\log2+\log\left( \frac{\vert\sinh(\frac{\tilde{\btheta}^T\bxi}{2})\vert}{\cosh(\frac{\rho}{2}\bmu^T\bxi)}\right) \right)\cdot1_{\{\bxi:\vert \tilde{\btheta}^T\bxi\vert \geq r\}}\right] \nonumber.
    \end{align}
Here  \eqref{eq:high_noise_blah_21} follows from $\log\left(1+\frac{\sinh^2(\frac{\tilde{\btheta}^T\bxi}{2})}{\cosh^2(\frac{\rho}{2}\bmu^T\bxi)} \right)$ is always positive. From \eqref{eq:fact_affine}, we have $\bmu^T\bxi \sim N(\Vert \bmu \Vert^2,\sigma^2\Vert \bmu \Vert^2)$ and $\tilde{\btheta}^T\bxi \sim N(0,\sigma^2\Vert \tilde{\btheta} \Vert^2)$, so $\tilde{\btheta}^T\bxi = \sigma \| \tilde{\btheta} \| \psi$ where $\psi \sim N(0,1)$. Moreover, a simple computation shows that $-\log\left(\cosh(\frac{\rho}{2}\bmu^T\bxi)\right)1_{\{\vert \tilde{\btheta}^T\bxi\vert \geq r\}} \geq -\log\left(\cosh(\frac{\rho}{2}\bmu^T\bxi)\right) $ since $\cosh(\frac{\rho}{2}\bmu^T\bxi)\geq 1$ always holds. Using the inequality $\log\cosh(x)\leq \vert x\vert$ for $x$, the following bound holds
     \begin{equation} \begin{aligned} \label{eq: high_noise_22}
       &\bE_{\bxi}\left[\log\left( \frac{1+\exp(-\btheta^T\bxi)}{1+\exp(-\rho \bmu^T\bxi)}\right)\right]\\
       & \quad \ge \tfrac{1}{2} \log(2) \cdot \bE_{\psi}\left[1_{\{\vert \psi\vert \geq \frac{r}{\sigma \Vert\tilde{\btheta}\Vert}\}}\right] + \tfrac{1}{2} \bE_{\psi}\left[\log\left\vert\sinh(\tfrac{\sigma\Vert \tilde{\btheta}\Vert\psi}{2})\right\vert1_{\{\vert \psi\vert \geq \frac{r}{\sigma \Vert\tilde{\btheta}\Vert}\}}\right] - \tfrac{1}{2}\bE_{\bxi} \left [ \log( \cosh(\tfrac{\rho}{2} \bmu^T\bxi)) \right ]\\
       & \quad \ge \tfrac{1}{2} \log(2) \cdot \bE_{\psi}\left[1_{\{\vert \psi\vert \geq \frac{r}{\sigma \Vert\tilde{\btheta}\Vert}\}}\right] + \tfrac{1}{2} \bE_{\psi}\left[\log\left\vert\sinh(\tfrac{\sigma\Vert \tilde{\btheta}\Vert\psi}{2})\right\vert1_{\{\vert \psi\vert \geq \frac{r}{\sigma \Vert\tilde{\btheta}\Vert}\}}\right] - \tfrac{1}{2} \bE_{\bxi} \left [ | \tfrac{\rho}{2} \bmu^T\bxi | \right ]\\
       &\quad\geq \tfrac{1}{2} \log(2) \cdot \bE_{\psi}\left[1_{\{\vert \psi\vert \geq \frac{r}{\sigma \Vert\tilde{\btheta}\Vert}\}}\right] + \tfrac{1}{2} \bE_{\psi}\left[\log\left\vert\sinh(\tfrac{\sigma\Vert \tilde{\btheta}\Vert\psi}{2})\right\vert1_{\{\vert \psi\vert \geq \frac{r}{\sigma \Vert\tilde{\btheta}\Vert}\}}\right]-\frac{3}{4}\left(\frac{\Vert \bmu\Vert^2}{\sigma^2}+ \sqrt{\frac{2}{\pi}}\cdot\frac{\Vert \bmu \Vert}{\sigma} \right),
    \end{aligned} \end{equation}
where the last inequality uses \eqref{fact:norm_Gaussians} and $\rho\leq \frac{3}{\sigma^2}$. Using the inequality $\vert \sinh(x) \vert \geq \exp(\frac{\vert x\vert}{2})$ for all $\vert x \vert \geq 2\log(\sqrt{2}+1)$ and letting $r=4\log(\sqrt{2}+1)$, we obtain 
\begin{align}
   \tfrac{1}{2} \log(2) \cdot \bE_{\psi} \big [ 1_{\{|\psi| \ge \frac{4 \log( \sqrt{2}+ 1)\}}{\sigma \norm{\tilde{\btheta}} }} \big ] &+ \tfrac{1}{2}\bE_\psi\left[\log\left\vert\sinh(\tfrac{\sigma\Vert \tilde{\btheta}\Vert\psi}{2})\right\vert1_{\{|\psi| \geq \frac{4\log(\sqrt{2}+1)}{\sigma \Vert \tilde{\btheta}\Vert}\}}\right]\nonumber\\
   &\geq \tfrac{1}{2} \log(2) \cdot \bE_{\psi} \big [ 1_{\{|\psi| \ge \frac{4 \log( \sqrt{2}+ 1)}{\sigma \norm{\tilde{\btheta}} } \}} \big ] +  \tfrac{1}{2}\bE_\psi\left[\left\vert\tfrac{\sigma \Vert \tilde{\btheta}\Vert \psi}{4} \right\vert1_{\{|\psi| \geq \frac{4\log(\sqrt{2}+1)}{\sigma \Vert \tilde{\btheta}\Vert}\}}\right]\nonumber\\
    &\geq \tfrac{1}{2} \log(2) \cdot \bE_{\psi}[1_{\{\vert\psi\vert \ge 1\}}] + \tfrac{1}{2} \bE_\psi\left[\left\vert\tfrac{\sigma \Vert \tilde{\btheta}\Vert \psi}{4} \right\vert1_{\vert\psi\vert \geq 1\}}\right] \label{eq:high_noise_blah_111} \\\nonumber
    &\geq \left ( \tfrac{1}{2} \log(2) + \frac{\sigma \Vert \tilde{\btheta}\Vert}{8} \right )\cdot\Phi^c(1).
\end{align}
Here \eqref{eq:high_noise_blah_111} follows from the assumption that $\sigma \Vert \tilde{\btheta}\Vert\geq 436
$. Combining  \eqref{eq: high_noise_22}, \eqref{eq:high_noise_blah_111} and the bounds $\sigma \geq 0.33\Vert \bmu \Vert$ and $\sigma\Vert \tilde{\btheta}\Vert \geq 436$ the result follows. \item
\textbf{Hinge loss.} We begin by denoting
$\xi_1:=\bxi^T\tilde{\btheta}$ and $\xi_2:=\bxi^T\bmu$. Notice that $\xi_1$ and $\xi_2$ are independent random variables. Recall that $\ell(t):=\ell(t,1)=\max(0,1-t)$. We have that 
\begin{align*}
    f(\btheta)-f(\rho \bmu)&=\bE_{\bxi}\left[\ell(\bxi^T\btheta)-\ell(\rho\bxi^T\bmu) \right]\\&=\bE_{\xi_1,\xi_2}\left[\ell(\xi_1+\rho\xi_2)-\ell(\rho\xi_2) \right]\\&=\bE_{\xi_1,\xi_2}\left[\ell(-\xi_1+\rho\xi_2)-\ell(\rho\xi_2) \right].
\end{align*}
The second equality follows since $\xi_1 \sim -\xi_1$. We define the function 
\[
\kappa(\xi_1,\xi_2):=\ell(\xi_1+\rho\xi_2)+\ell(-\xi_1+\rho\xi_2)-2\ell(\rho\xi_2).
\] 
We therefore obtain that 
\[
2\left(f(\btheta)-f(\rho\bmu)\right)=\bE_{\xi_1,\xi_2}\left[\kappa(\xi_1,\xi_2)\right].
\]
Next we claim that
\begin{equation}\label{eq:hinge_lemma_claim1}
    \kappa(\xi_1,\xi_2)=0 \text{ whenever } \vert \xi_1\vert\leq \vert 1-\rho\xi_2\vert.
\end{equation}
To see this, suppose that $\vert \xi_1\vert\leq \vert 1-\rho\xi_2\vert$ holds. We consider two cases.  First, assume that $0\leq 1-\rho\xi_2$ which yields that $\rho\xi_2-\xi_1\leq 1$ and $\rho\xi_2+\xi_1\leq 1$. We therefore have $\kappa(\xi_1,\xi_2)=1-\xi_1-\rho\xi_2+1+\xi_1-\rho\xi_2-2(1-\rho\xi_2)=0$. Second, assume that $1-\rho\xi_2\leq 0$. It thus holds that $1\leq \rho\xi_2-\xi_1$ and $1\leq \rho\xi_2+\xi_1$. Now it immediately follows that $\kappa(\xi_1,\xi_2)=0$ and equation \eqref{eq:hinge_lemma_claim1} is established. We claim the following 
\begin{equation}\label{eq:hinge_lemma_claim2}
    \kappa(\xi_1,\xi_2)=\vert \xi_1\vert-\vert 1-\rho\xi_2\vert \text{ whenever } \vert \xi_1\vert\geq \vert 1-\rho\xi_2\vert.
\end{equation}
To this end, we again consider two cases. First, assume that $\xi_1\leq -\vert 1-\rho\xi_2\vert$. This yields that $1\leq -\xi_1+\rho\xi_2$ and $\xi_1+\rho\xi_2\leq 1$, so it holds that $ \kappa(\xi_1,\xi_2)= 1-\xi_1-\rho\xi_2-2\ell(\rho\xi_2)$. The claim \eqref{eq:hinge_lemma_claim2} follows from the following simple identity 
\begin{equation}\label{eq:hinge_identity}
   2\ell(t)= 1-t+\vert 1 - t \vert, \quad \forall t \in \R.
\end{equation}
Second, assume that $\xi_1\geq \vert 1-\rho\xi_2\vert$. It then holds that $\xi_1+\rho\xi_2\geq 1$ and $-\xi_1+\rho\xi_2\leq 1$ and therefore $
\kappa(\xi_1,\xi_2)=1+\xi_1-\rho\xi_2-2\ell(\rho \xi_2).
$ The claim \eqref{eq:hinge_lemma_claim2} follows from the identity \eqref{eq:hinge_identity}.  We therefore obtain 
\begin{align}
    \bE_{\xi_1,\xi_2}[\kappa(\xi_1,\xi_2)]&=2\bE_{\xi_1,\xi_2}[(\ell(\xi_1+\rho\xi_2)+\ell(-\xi_1+\rho\xi_2)-2\ell(\rho\xi_2))1_{\{ \xi_1>0\}}]\label{eq:hinge_lemma_it1}\\&=2\bE_{\xi_1,\xi_2}[(\ell(\xi_1+\rho\xi_2)+\ell(-\xi_1+\rho\xi_2)-2\ell(\rho\xi_2))1_{\{ \xi_1\geq\vert 1-\rho \xi_2\vert\}}]\label{eq:hinge_lemma_it2}\\&=\bE_{\xi_1,\xi_2}[(\xi_1-\vert 1-\rho \xi_2\vert)1_{\{ \xi_1\geq \vert 1-\rho\xi_2\vert\}} ]. \label{eq:hinge_lemma_it3}
\end{align}
Here equation \eqref{eq:hinge_lemma_it1} holds because $\xi_1\sim -\xi_1$ and $\kappa(\xi_1,\xi_2)=\kappa(-\xi_1,\xi_2)$. Equation \eqref{eq:hinge_lemma_it2} is true because of claim \eqref{eq:hinge_lemma_claim1} and \eqref{eq:hinge_lemma_it3} follows from claim \eqref{eq:hinge_lemma_claim2}. From \eqref{eq:hinge_lemma_it3}, we conclude that $\bE_{\xi_1,\xi_2}[\kappa(\xi_1,\xi_2)]\geq 0$.  We then observe the bound
\begin{equation}
    \begin{aligned}
    \bE_{\xi_1,\xi_2}[(\xi_1-\vert 1-\rho \xi_2\vert)1_{\{ \xi_1\geq \vert 1-\rho\xi_2\vert\}} ] &= \tfrac{1}{2}\bE_{\xi_1,\xi_2}[\xi_1-\vert 1-\rho \xi_2\vert+\vert \xi_1-\vert 1-\rho \xi_2\vert \vert]\\&\geq -\tfrac{1}{2}\bE_{\xi_2}[\vert 1-\rho \xi_2\vert]+\tfrac{1}{2}\bE_{\xi_1,\xi_2}[\vert \xi_1\vert-\vert 1-\rho \xi_2\vert]\\&=\tfrac{1}{2}\bE_{\xi_1}[\vert \xi_1\vert]-\bE_{\xi_2}[\vert1-\rho\xi_2\vert]. \label{eq:thrid_last_equation}
\end{aligned}
\end{equation}
The second inequality follows from $\bE_{\xi_1}[\xi_1]=0$ and the triangle inequality $\vert x\vert-\vert y\vert \leq \vert \vert x\vert-y\vert$. On the other hand, it holds that
\begin{equation}\label{first_last_equation}
\bE_{\xi_1}[\vert \xi_1\vert]=\sqrt{\frac{2}{\pi}}\cdot\sigma \Vert \tilde{\btheta}\Vert,
\end{equation}
and 
\begin{equation}\label{second_last_equation}
\bE_{\bxi}[\vert 1-\rho \bmu^T\bxi\vert]\leq 1+\rho \bE_{\bxi}[\vert \bmu^T\bxi\vert]\leq 1+\rho\Vert \bmu \Vert \left(\sqrt{\frac{2}{\pi}}\cdot\sigma +\Vert \bmu \Vert \right).
\end{equation}
Combing equations \eqref{eq:hinge_lemma_claim1}, \eqref{eq:hinge_lemma_claim2},  \eqref{eq:thrid_last_equation}, \eqref{first_last_equation}, and \eqref{second_last_equation}, we deduce
\begin{equation}\label{eq:last_hinge_high}
    f(\btheta)-f(\rho \bmu) \geq \frac{1}{2}\left(\sqrt{\frac{1}{2\pi}}\cdot\sigma\Vert\tilde{\btheta}\Vert-1-\rho\Vert \bmu \Vert \left(\sqrt{\frac{2}{\pi}}\cdot\sigma +\Vert \bmu \Vert \right) \right).
\end{equation}
Using the bounds $\sigma\Vert \tilde{\btheta}\Vert \geq 8+10\rho^*\sigma^2$, $\sigma \geq 0.62\Vert \bmu \Vert$ and $\rho\leq \tfrac{3}{2}\rho^*$, the result follows from \eqref{eq:last_hinge_high}.
\end{proof}
We next derive a lower bound \eqref{eq:quantity_split}, Part (b). But, first we need a basic lemma from convex analysis.
\begin{lemma}\label{lem:convex_analysis_lemma}
Suppose that $g:\R_{\geq 0}\to \R$ is a convex function with a minimizer at $\rho^*>0$. Assume that $g$ is twice differentiable on the interval $[\tfrac{3}{4}\rho^*,\tfrac{5}{4}\rho^*]$ and there exists a constant $B>0$ such that $g''(\rho)\geq B$ for all $\rho \in [\tfrac{3}{4}\rho^*,\tfrac{5}{4}\rho^*]$. Then it holds that
\begin{equation}\label{eq0:cvx_lemma}
    g(\rho)-g(\rho^*)\geq \frac{\rho^*B}{8}\vert \rho-\rho^*\vert \quad \text{for all $\rho \not\in [\tfrac{1}{2}\rho^*,\tfrac{3}{2}\rho^*]$}.
\end{equation}
\end{lemma}
\begin{proof}
 The proof follows by considering the second order Taylor series expansion of the function $g$.
\end{proof}

\begin{lemma}(Lower bound for (b) in \eqref{eq:quantity_split})\label{lem:lower_bound_2} Fix $\btheta\in \R^d$ and orthogonally decompose $\btheta=\rho\bmu+\tilde{\btheta}$. Suppose that $\vert \rho-\rho^*\vert \geq \tfrac{1}{2}\rho^*$. Then provided that 
$\sigma \geq c\Vert \bmu \Vert$ where the constant $c$ is defined in \eqref{eq:para_log} and \eqref{eq:para_hinge}, there exists a positive constant $A$ such that the following is true
\begin{equation}\label{eq:lemma9}
    f(\rho \bmu)-f(\bm{\theta^*}) \geq A\cdot\frac{\Vert \bmu \Vert^2}{\sigma^2}.
\end{equation}
\begin{proof} We consider the logistic and hinge loss separately. 
\item \textbf{Logistic loss.} Define the function
\[
    g(\rho):=\bE_{\bxi}\left[\log\left(1+\exp(-\rho \bmu^T\bxi)\right) \right], \quad \bxi \sim N(\bmu,\sigma^2I_d).
\]
By Lemma~\ref{lem:minimizers}, we know that $g$ is a convex function with a unique minimizer at $\rho^*:=\frac{2}{\sigma^2}$. Observe that $f(\rho \bmu)-f(\btheta^*) = g(\rho)-g(\rho^*)$; hence in order to prove \eqref{eq:lemma9}, we instead aim to bound this difference in the function $g$. From \eqref{eq:fact_affine}, we have $\bmu^T\bxi \sim N(\Vert \bmu \Vert^2,\sigma^2\Vert \bmu \Vert^2)$. It thus holds
\[
 4g''(\rho)=\bE\left(\frac{(\bmu^T\bxi)^2}{\cosh(\frac{\rho}{2} \bmu^T\bxi)^2} \right)=\frac{1}{\sigma\Vert \bmu\Vert\sqrt{2\pi }}\int_{-\infty}^{\infty}\frac{z^2}{\cosh^2(\frac{\rho z}{2})}\exp\left(-\frac{(z-\Vert \bmu \Vert^2)^2}{2\sigma^2\Vert \bmu \Vert^2}\right)dz.
\]
Upper bounding $\cosh^2(\frac{\rho z}{2})$ by $\exp(\vert \rho z\vert)$, we next obtain
\begin{align*}
    4g''(\rho)&\geq \frac{1}{\sigma\Vert \bmu\Vert\sqrt{2\pi}} \int_{-\infty}^{\infty} z^2\exp\left(-\vert \rho z \vert \right)\exp\left(-\frac{(z-\Vert \bmu \Vert^2)^2}{2\sigma^2\Vert \bmu\Vert^2}\right)dz & \\ &=\frac{1}{\sigma\Vert \bmu\Vert\sqrt{2\pi }} \int_{-\infty}^{\infty} z^2\exp\left(-\frac{(z-\Vert \bmu \Vert^2)^2+2\sigma^2\Vert \bmu\Vert^2\vert\rho z\vert}{2\sigma^2\Vert \bmu\Vert^2}\right)dz  ,&\\
    &=\frac{1}{\sigma\Vert \bmu\Vert\sqrt{2\pi }}\cdot\exp\left(-\frac{\Vert \bmu \Vert^2}{2\sigma^2} \right) \int_{-\infty}^{\infty} z^2\exp\left(-\frac{z^2-2\Vert \bmu \Vert^2 z+2\sigma^2\Vert \bmu\Vert^2\vert\rho z\vert}{2\sigma^2\Vert \bmu\Vert^2}\right)dz
    \\&= \frac{\sigma^2\Vert \bmu\Vert^2}{\sqrt{2\pi }}\cdot\exp\left(-\frac{\Vert \bmu \Vert^2}{2\sigma^2} \right)\int_{-\infty}^{+\infty} z^2\exp\left(-\frac{z^2-2\frac{\Vert \bmu \Vert}{\sigma}z+2\vert \rho z \vert \sigma\Vert\bmu\Vert }{2} \right) dz\\&\geq \frac{\sigma^2\Vert \bmu\Vert^2}{\sqrt{2\pi }}\cdot\exp\left(-\frac{\Vert \bmu \Vert^2}{2\sigma^2} \right)\int_{0}^{+\infty} z^2\exp\left(-\frac{z^2}{2}\right)\exp\left(z\left(\frac{\Vert\bmu\Vert}{\sigma}-\rho \sigma \Vert \bmu \Vert\right) \right)dz\\&\geq \frac{\sigma^2\Vert \bmu\Vert^2}{\sqrt{2\pi }}\cdot\exp\left(-\frac{\Vert \bmu \Vert^2}{2\sigma^2}-\frac{1}{2}-\left \vert \frac{\Vert \bmu \Vert}{\sigma}-\rho \sigma \Vert \bmu \Vert \right\vert  \right)\int_{0}^1 z^2 dz.
    \end{align*}
Here the second to last inequality follows from the change of variables $z\rightarrow z\sigma\Vert \bmu\Vert$. The last inequality follows from restricting the integral's domain to $[0,1]$ and also lower bounding $-\frac{z^2}{2}$ and $z\left(\frac{\Vert\bmu\Vert}{\sigma}-\rho \sigma \Vert \bmu \Vert\right)$ by $\frac{-1}{2}$ and $-\left\vert\frac{\Vert\bmu\Vert}{\sigma}-\rho \sigma \Vert \bmu \Vert\right\vert$ respectively. We see that, for $\rho \in [\tfrac{3}{4}\rho^*,\tfrac{5}{4}\rho^*]$,  the term $\exp\left(-\frac{\Vert \bmu \Vert^2}{2\sigma^2}-\frac{1}{2}-\left \vert \frac{\Vert \bmu \Vert}{\sigma}-\rho \sigma \Vert \bmu \Vert \right\vert  \right)$ is lower bounded by $\exp\left(-\frac{1}{2c^2}-\frac{1}{4c}-\frac{1}{2} \right)$. By Lemma \ref{lem:convex_analysis_lemma}, the result follows with the constant $A$ computed as follows
\begin{equation*}\label{eq:constant_A_log}
    A=\frac{1}{12\sqrt{2\pi}}\cdot\exp\left(-\frac{1}{2c^2}-\frac{1}{4c}-\frac{1}{2} \right).
\end{equation*}
\item \textbf{Hinge loss.} We begin by defining the function $h(\rho)=f(\rho\bmu)$. Therefore
\[
f(\rho \bmu)=\bE_{\bxi}[\ell(\rho\bxi^T\bmu)]=\bE_{\bxi}[(1-\rho \bxi^T\bmu)1_{\{\rho\bxi^T\bmu\leq 1 \}}].
\]
Hence, it holds that 
\[
h'(\rho)=\bmu^T\nabla f(\rho \bmu)=-\bE_{\bxi}[\bxi^T\bmu1_{\{ \rho\bxi^T\bmu\leq 1\}}].
\]
From \eqref{eq:fact_affine}, we obtain that $\bmu^T\bxi \sim N(\Vert \bmu \Vert^2,\sigma^2\Vert \bmu \Vert^2)$.  For $\rho>0$, therefore, it holds that
\begin{equation}\label{der_phi_pos}
    h'(\rho)=\frac{-1}{\sigma\Vert \bmu \Vert\sqrt{2\pi}}\int_{-\infty}^{\frac{1}{\rho}}z\exp\left(-\frac{1}{2}\cdot\left(\frac{z}{\sigma\Vert \bmu \Vert}-\frac{\Vert \bmu \Vert}{\sigma} \right)^2\right)dz.
\end{equation}
Applying chain rule thus yields 
\[
h''(\rho)=\frac{1}{\rho^3\sigma\Vert \bmu \Vert\sqrt{2\pi}}\exp\left(-\frac{1}{2}\cdot \left(\frac{1}{\rho\sigma\Vert \bmu \Vert}-\frac{\Vert \bmu \Vert}{\sigma} \right)^2 \right)\quad \text{for all $\rho>0$}.
\]
Hence, for all $\rho \in [\tfrac{3}{4}\rho^*,\tfrac{5}{4}\rho^*]$ it holds that
\[
h''(\rho)\geq \frac{64}{125{\rho^*}^3\sigma\Vert \bmu \Vert\sqrt{2\pi}}\exp\left(-\frac{1}{2}\cdot \Gamma^2 \right),
\]
where $\Gamma:=\max\left\{\left|\frac{4}{3\rho^*\sigma\Vert \bmu \Vert}-\frac{\Vert \bmu \Vert}{\sigma} \right|,\left|\frac{4}{5\rho^*\sigma\Vert \bmu \Vert}-\frac{\Vert \bmu \Vert}{\sigma} \right| \right\}$. Therefore, by Lemma \ref{lem:convex_analysis_lemma} and $\vert \rho-\rho^*\vert \geq \tfrac{1}{2}\rho^*$,  it holds that 
\begin{equation}\label{eq:constant_r_hinge_loss}
\begin{aligned}
    f(\rho\bmu)-f(\btheta^*)&\geq \frac{4}{125\sqrt{2\pi}}\cdot\frac{\sigma}{r\Vert \bmu \Vert}\cdot\exp\left(-\frac{1}{2}\cdot \Gamma^2 \right).
\end{aligned}
\end{equation}
Here $r=\rho^*\sigma^2$. Note that $r>0$ by Lemma \ref{lem:minimizers}. We aim to lower bound the right-hand side of \eqref{eq:constant_r_hinge_loss}. We denote by $w=\frac{\sigma}{r\Vert \bmu \Vert}-\frac{\Vert \bmu \Vert}{\sigma}$ the quantity defined in Lemma \ref{lem:minimizers}. In particular, by Lemma \ref{lem:minimizers}, the following holds
\begin{equation}\label{eq:w_hinge_later}
\frac{1}{\sqrt{2\pi}}\cdot\frac{\sigma}{\Vert \bmu \Vert}=\Phi(w)\cdot \exp(\tfrac{1}{2}w^2).
\end{equation}

We consider two cases.  First suppose that $w\geq \frac{1}{(3\sqrt{2}-4)c}$. Along with the assumption $\frac{\sigma}{\Vert \bmu \Vert}\geq c$ this implies that $w\geq \frac{1}{3\sqrt{2}-4}\cdot \frac{\Vert \bmu \Vert}{\sigma}$. A simple computation shows that $w^2\geq \frac{1}{2}\cdot \Gamma^2$ for all $w\geq \frac{1}{3\sqrt{2}-4}\cdot \frac{\Vert \bmu \Vert}{\sigma}$. On the other hand, by \eqref{eq:w_hinge_later} for $w\geq 0$, we obtain that $\frac{2}{\pi}\cdot \frac{\sigma^2}{\Vert \bmu \Vert^2}\geq \exp(w^2)$. Plugging in the bounds  $w^2\geq \frac{1}{2}\cdot \Gamma^2 $,  $\exp(-w^2)\geq \frac{\pi}{2}\cdot\frac{\Vert \bmu \Vert^2}{\sigma^2}$, and $\frac{\sigma}{r\Vert \bmu \Vert}\geq w\geq \frac{1}{(3\sqrt{2}-4)c}$ into the right-hand-side of \eqref{eq:constant_r_hinge_loss}, we obtain that 
\begin{equation*}
    f(\rho\bmu)-f(\btheta^*)\geq \frac{\sqrt{2\pi}}{125(3\sqrt{2}-4)c}\cdot\frac{\Vert \bmu \Vert^2}{\sigma^2}.
\end{equation*}
Next, suppose that $w<\frac{1}{(3\sqrt{2}-4)c}$. In this case, the two factors $\frac{\sigma}{r\Vert \bmu \Vert}$ and $\exp\left(-\frac{1}{2}\cdot\Gamma^2 \right)$ in \eqref{eq:constant_r_hinge_loss} are lower bounded separately. Note that it always holds that $w\geq -\frac{\Vert \bmu \Vert}{\sigma}$ as $r>0$. Therefore, it is easy to see that the latter factor is lower bounded by $\exp\left(-\frac{1}{2} \left(\frac{4}{3(3\sqrt{2}-4)c}+\frac{1}{3c} \right)^2\right)$. Hence, it remains to bound the factor $\frac{\sigma}{r\Vert \bmu \Vert}$ in \eqref{eq:constant_r_hinge_loss}. To this end, we show that $w\geq -\frac{\Vert \bmu \Vert}{2\sigma}$ for all $\frac{\sigma}{\Vert \bmu \Vert}\geq c$.  Note that a chain of change of variables gives
  \begin{equation*}\label{eq:myeq_blah}
  \begin{aligned}
      \Phi(w)\cdot\exp\left(\frac{w^2}{2}\right)&=\frac{1}{\sqrt{2\pi}}\cdot\int_{0}^{+\infty}\exp(-\frac{1}{2}t^2)\cdot\exp(wt) \, dt.
 \end{aligned}\end{equation*}
The right-hand side of \eqref{eq:w_hinge_later} is an increasing function with respect to $w$. Therefore it suffices to show that the following holds
\begin{equation}\label{eq:sigma/bmu>c}
    \frac{1}{\sqrt{2\pi}}\cdot \frac{\sigma}{\Vert \bmu \Vert} \geq \Phi\left(-\frac{\Vert \bmu \Vert}{2\sigma} \right)\cdot \exp\left(\frac{\Vert \bmu \Vert^2}{8\sigma^2} \right) \quad \text{whenever} \quad \frac{\sigma}{\Vert \bmu \Vert}\geq c. 
\end{equation}
However, it can be verified by a plot that $\tfrac{1}{\sqrt{2\pi}}\geq t\cdot \Phi\left(-\tfrac{t}{2} \right)\cdot \exp\left(\frac{t^2}{8} \right)$ holds for all $t\in (0,\frac{1}{c})$. Therefore, we have shown that $w\geq -\frac{\Vert \bmu \Vert}{2\sigma}$ which implies that $\frac{\sigma}{r\Vert \bmu \Vert}\geq \frac{\Vert \bmu \Vert}{2\sigma}$. Finally we lower bound the quantity $\frac{\sigma}{r\Vert\bmu\Vert}$ by $c\cdot\frac{\Vert \bmu \Vert^2}{2\sigma^2}$. We have concluded \eqref{eq:lemma9} in case of hinge loss function where the constant $A$ can be computed as follows
\begin{equation*}\label{eq:constant_A_hinge}
    A=\min\left\{\frac{c}{2}\cdot\exp\left(-\frac{1}{2} \left(\frac{4}{3(3\sqrt{2}-4)c}+\frac{1}{3c} \right)^2\right), \frac{\sqrt{2\pi}}{125(3\sqrt{2}-4)c} \right\}.
\end{equation*}
\end{proof}
\end{lemma}
We now have the ingredients to prove Theorem \ref{thm:high}.
\begin{proof}[Proof of Theorem \ref{thm:high}]
Consider the set $C$ and function $V$ defined in \eqref{eq:CV_high_recall}:
\begin{equation}\label{eq:CV_high_proofthm3}
    C:=\left\{\btheta: \vert \rho -\rho^*\vert<\tfrac{1}{2}\rho^* \text{ and } \sigma \Vert \tilde{\btheta}\Vert \leq c'\right\} \quad \text{ and }\quad  V(\btheta)=\frac{1}{2\alpha}\Vert \btheta-\btheta^*\Vert^2.
\end{equation}
We let $c'$ to be defined as in Lemma \ref{lem:lower_bound_1}. This means that $c'$ equals to $436$ and $8+10\rho^*\sigma^2$ in case of logistic and hinge loss respectively. We next show that there exists a positive constant $\delta$ such that the following is true
\begin{equation}
    \bP_{\bxi}\left(\bxi^T\btheta\geq 1\right)\geq \delta \quad \text{for all}\quad \btheta \in C.
    \end{equation}
Let $\btheta\in C$ and orthogonally decompose it into $\btheta=\rho\bmu+\tilde{\btheta}$. We have that $\bxi^T\btheta=\rho\bxi^T\bmu+\bxi^T\tilde{\btheta}$. Note that $\rho>0$ as $\btheta\in C$. By \eqref{eqn:fact_independence}, we see that $\bxi^T\btheta$ and $\bxi^T\tilde{\btheta}$ are independent normal random variables. It thus holds that 
\begin{equation}
    \bP_{\bxi}\left(\bxi^T\btheta\geq 1\right)\geq \bP_{\bxi}\left(\rho\bxi^T\bmu\geq 1\right)\cdot \bP_{\bxi}\left(\bxi^T\tilde{\btheta}\geq 0\right) =\frac{1}{2}\cdot\bP_{\bxi}\left(\bxi^T\bmu\geq \frac{1}{\rho}\right).
\end{equation}
Rewrite the inequality $\bxi^T\bmu\geq \frac{1}{\rho}$ by $z:=\frac{\bxi^T\bmu-\Vert \bmu \Vert^2}{\sigma\Vert \bmu \Vert}\geq \frac{\frac{1}{\rho}-\Vert \bmu \Vert^2}{\sigma\Vert 
\bmu \Vert}$. Noting that $z\sim N(0,1)$ and using the inequality $\frac{2}{\rho^*}\geq\frac{1}{\rho}$, we obtain that 
\begin{equation}\label{eq:delta_proof}
    \bP_{\bxi}\left(\bxi^T\btheta\geq 1\right)\geq \delta:= \frac{1}{2}\cdot\Phi^c\left(\frac{\frac{2}{\rho^*}-\Vert\bmu \Vert^2}{\sigma\Vert \bmu \Vert}\right).
\end{equation} 
We next show that the pair $(C,V)$ satisfies the drift equation \eqref{eq:driftequation}.  Let us rewrite \eqref{eq:quantity_split}:
\begin{equation}\label{eq:quantity_split2}
\begin{aligned}
f(\btheta_{k-1})-f(\btheta^*)= & \underbrace{f(\btheta_{k-1})-f(\rho_{k-1}\bmu)}_{(a)}+\underbrace{f(\rho_{k-1}\bmu)-f(\btheta^*)}_{(b)}.
\end{aligned}
\end{equation}
By Lemmas \ref{lem:lower_bound_1} and \ref{lem:lower_bound_2}, both terms in $(a)$ and $(b)$ in \eqref{eq:quantity_split2} are non-negative . Assume that $\btheta_{k-1}\not\in C$. Therefore, either $\sigma \Vert \tilde{\btheta}_{k-1}\Vert\geq c'$ or $\vert\rho_{k-1}-\rho^*\vert\geq \tfrac{1}{2}\rho^*$; this implies that the quantity $(a)$ is at least 1 or the quantity $(b)$ is at least $A\cdot\frac{\Vert \bmu \Vert^2}{\sigma^2}$ respectively. The constant $A$ in Lemma \ref{lem:lower_bound_2} satisfies $1\geq A\cdot \frac{\Vert \bmu \Vert^2}{\sigma^2}$ for all $\frac{\sigma}{\Vert \bmu \Vert}\geq c$. Hence it holds that 
\begin{equation}
    A\cdot\frac{\Vert \bmu \Vert^2}{\sigma^2}\leq f(\btheta_{k-1})-f(\btheta^*) \quad \text{for all } \btheta_{k-1}\not\in C.
\end{equation}
We use \eqref{eq:high_convex_tech_lemma} next to establish the drift equation \eqref{eq:driftequation}. Recall that the following holds 
    \begin{equation} 
       f(\bm{\theta}_{k-1})-f(\bm{\theta}^*)\leq \frac{1}{2\alpha}\left(\Vert  \bm{\theta}_{k-1}-\bm{\theta} ^*\Vert^2-\bE\left[\Vert \bm{\theta}_{k}-\bm{\theta}^* \Vert^2\, |\mathcal{F}_{k-1}\right] \right)+\frac{\alpha}{2}\left(\Vert \bmu \Vert^2+d\sigma^2\right).
    \end{equation}
Combining the last two displayed inequalities and using the definition of function $V$, we obtain that 
\begin{equation}
     \left(\bE\left[V(\btheta_k)|\mathcal{F}_{k-1}\right]-V(\btheta_{k-1})\right)\cdot 1_{\{\btheta_{k-1}\not\in C\}} \leq \left(\frac{\alpha}{2}(\Vert \bmu \Vert^2+d\sigma^2)-A\cdot\frac{\Vert \bmu\Vert^2}{\sigma^2}\right)\cdot 1_{\{\btheta_{k-1}\not\in C\}}.
\end{equation}
Therefore, by choosing $\alpha<A\cdot\frac{\Vert \bmu \Vert^2}{\sigma^2(\Vert \bmu \Vert^2+d\sigma^2)}$, we obtain the drift equation \eqref{eq:driftequation} holds with $b:=\frac{A}{2}\cdot\frac{\Vert \bmu \Vert^2}{\sigma^2}$. Next, we obtain bounds on $\bE[\tau_m]$ for $m\geq 1$. By Lemma \ref{lem:drift_from_meyn} and a simple induction, we obtain that 
 \begin{equation}
      \bE[\tau_m]\leq \tfrac{1}{b}V(0)+\tfrac{1}{b}(m-1)\sup_{\btheta\in C} V(\btheta).
 \end{equation}
 Compactness of set $C$ yields that, $\sup_{\btheta\in C} V(\btheta)<+\infty$. Therefore, for some constant $\gamma$, the following is true
 \begin{equation}\label{eq:bound_on_tau_m_high}
      \bE[\tau_m] \leq \gamma\cdot m.
 \end{equation}
Combining \eqref{eq:bound_on_tau_m_high}, \eqref{eq:delta_proof} and Lemma \ref{lem:ETleqET_C}, the proof immediately follows. 
\end{proof}
\subsection{Angle bound, proof of Theorem \ref{thm:angle_bound}}\label{sec:theorem_angle}

\begin{proof}[Proof of Theorem \ref{thm:angle_bound}] Recall the SGD algorithm for logistic regression uses the update
\[
\btheta_k=\btheta_{k-1}+\frac{\alpha\bxi_k}{1+\exp(\bxi_k^T\btheta_{k-1})}
\]
and for hinge regression
\[
\btheta_k=\btheta_{k-1}+\alpha 1_{\{\bxi_k^T\btheta_{k-1}\leq 1\}}\bxi_{k-1}
\]
where $\btheta_0=\bm{0}$ and $\bxi_1,\bxi_2,\cdots \stackrel{i.i.d}{\sim} N(\bmu,\sigma^2I_d)$.  It clearly holds in both cases that
\begin{equation}
    \left\vert\vert \bm{v}^T\btheta_{k}\vert-\vert \bm{v}^T\btheta_{k-1}\vert \right\vert\leq \alpha \vert \bm{v}^T\bxi_{k-1}\vert.
\end{equation}
We define a new random variable $X_k:=\vert \bm{v}^T\btheta_k\vert-k \sigma\alpha \sqrt{\frac{2}{\pi}} $. Observe that $\bE\left[\left|X_0\right|\right] = 0$ and for all $k \ge 1$, it holds that
\begin{align*}
    \bE\left[\left\vert X_k\right\vert\right]&\leq \alpha\sum_{i=1}^k \bE\left[\left\vert \bm{v}^T\bxi_k \right\vert\right]+k\sigma\alpha \sqrt{\frac{2}{\pi}} <\infty,
\end{align*}
\textit{i.e.}, $X_k \in \mathcal{L}^1$ for all $k \geq 1$. Next, we have for any $k \ge 1$
\begin{align*}
  \bE\left[\left\vert X_{k}-X_{k-1}\right\vert\,|\,\mathcal{F}_{k-1} \right] \leq  \bE\left[\left\vert \vert \bm{v}^T\btheta_k\vert -\vert \bm{v}^T\btheta_{k-1}\vert\right\vert\,|\,\mathcal{F}_{k-1} \right]+\sigma\alpha  \sqrt{\frac{2}{\pi}}  &\leq 2\sigma\alpha  \sqrt{\frac{2}{\pi}} .
\end{align*}
Here we used that $\bm{v}^T\bxi_k \sim N(0,\sigma^2)$ along with \eqref{fact:norm_Gaussians}. We also see that \begin{align*}
    \bE\left[\vert \bm{v}^T\btheta_{k}\vert\,|\,\mathcal{F}_{k-1} \right]&\leq \vert \bm{v}^T\btheta_{k-1}\vert+ \sigma \alpha\sqrt{\frac{2}{\pi}} \quad \Rightarrow \quad  \bE\left[X_{k}\,|\,\mathcal{F}_{k-1} \right]\leq X_{k-1}.
\end{align*}
Therefore, we have shown that $X_0,X_1,\cdots$ is a super-martingale. By Theorem \ref{thm:Durret_Martingale}, we have $\bE\left[X_{T}\right]\leq 0$. The result follows.

\end{proof}

\section{Numerical Experiments} \label{sec:Num_Experiment}

We investigate the performance of our termination test on two popular data sets, MNIST \citep{MNIST} and CIFAR-10 \citep{cifar10}, as well as synthetic data generated from Gaussians and heavy-tailed student t-distributions. All tests were performed using our zero overhead stopping criteria outlined in \eqref{eq: practical_termination_test}; experiments using our test which required an extra sample \eqref{eq: termination_test} are not presented since the behaviors of the two criteria were indistinguishable on all data sets.

\paragraph{Comparison with a popular stopping criterion.} We include as a baseline a popular termination test, the small validation set (SVS) \citep{Prechelt2012}. The SVS termination test is as follows. One fixes a validation set of $p$ instances $(\bzeta^{\rm V}_1,y^{\rm V}_1)$, \ldots, $(\bzeta^{\rm V}_p,y^{\rm V}_p)$ drawn from the same distribution as the training data. Then for $m = 1, 2, \ldots$, one checks the fraction correct of the current classifier $\btheta_{ml}$, where $ml$ is the iteration index, on the $p$ instances.  In other words, the SVS test is run once every $l$ iterations.   If the fraction correct fails to increase compared to the last run of the SVS, then the SGD iterations are terminated. 

Note the computational overhead of running the small validation set is about $p$ times the cost of one SGD iteration.  Therefore, in order to make the overhead only a constant factor, we choose $l=2p$, meaning an approximately 50\% overhead for SVS. In contrast, the overhead for \eqref{eq: practical_termination_test} is
0. The value of $p$ is a tuning parameter for SVS; we exhibit results for three different $p$ values (see Figs.~\ref{fig:normal_scatter}, \ref{fig:ht_scatter}, \ref{fig:mnist_scatter}, \ref{fig:cifar_scatter} ).

\begin{figure*}[htp!]
\begin{center}
\subfigure[]{\includegraphics[height=3cm]{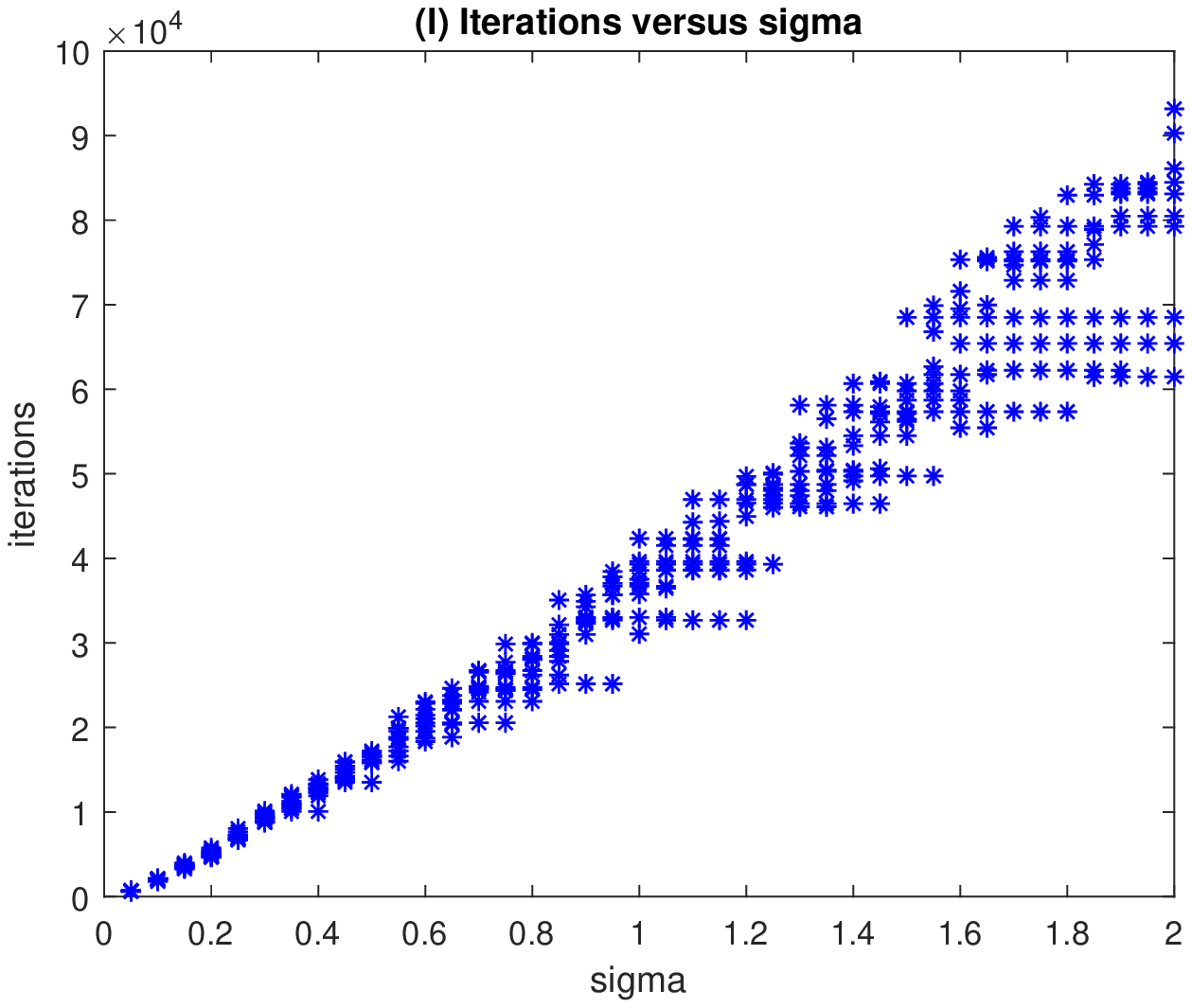}} \quad
\subfigure[\label{fig:black}]{\includegraphics[height=3cm]{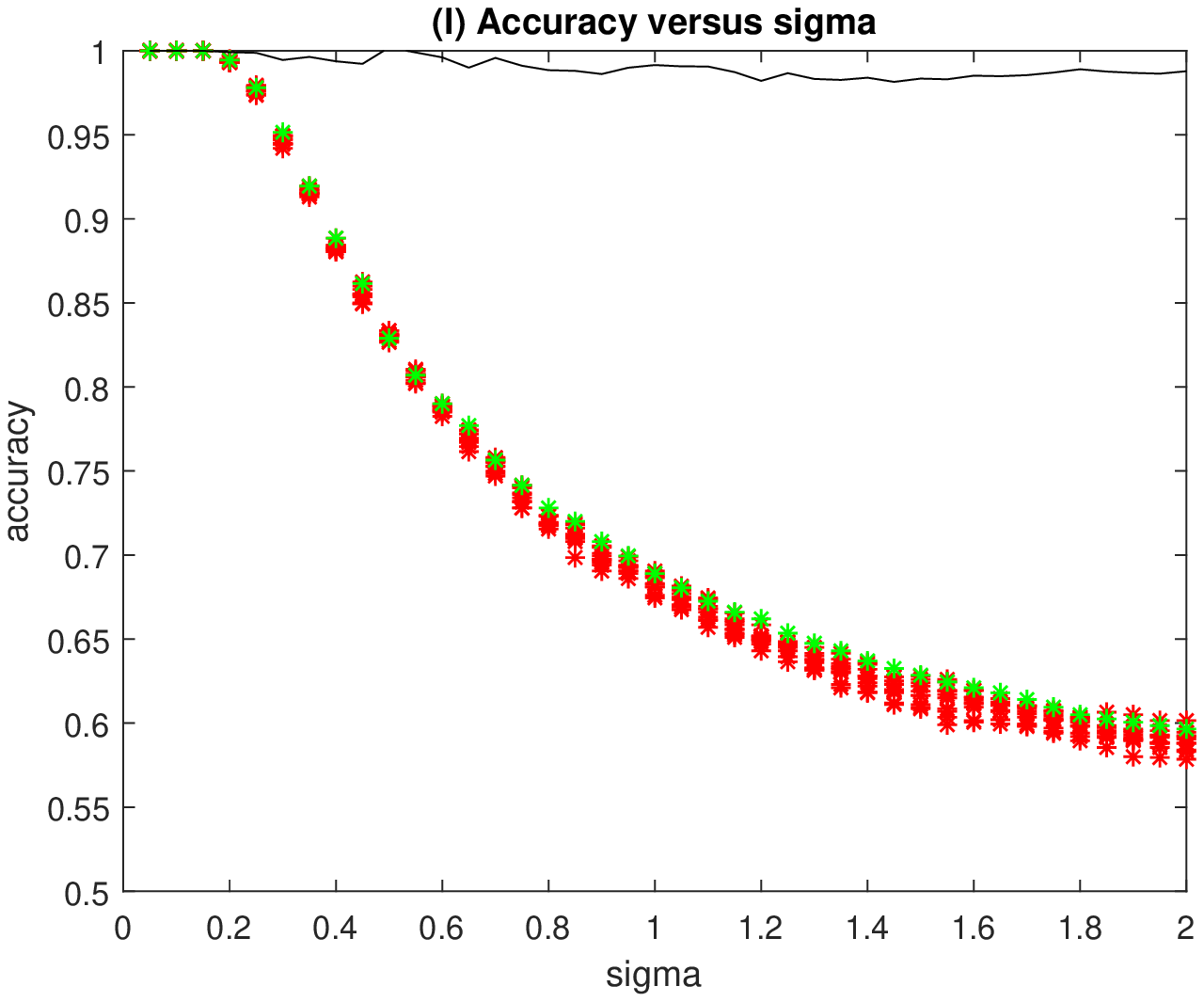}} \quad 
\subfigure[]{\includegraphics[height=3cm]{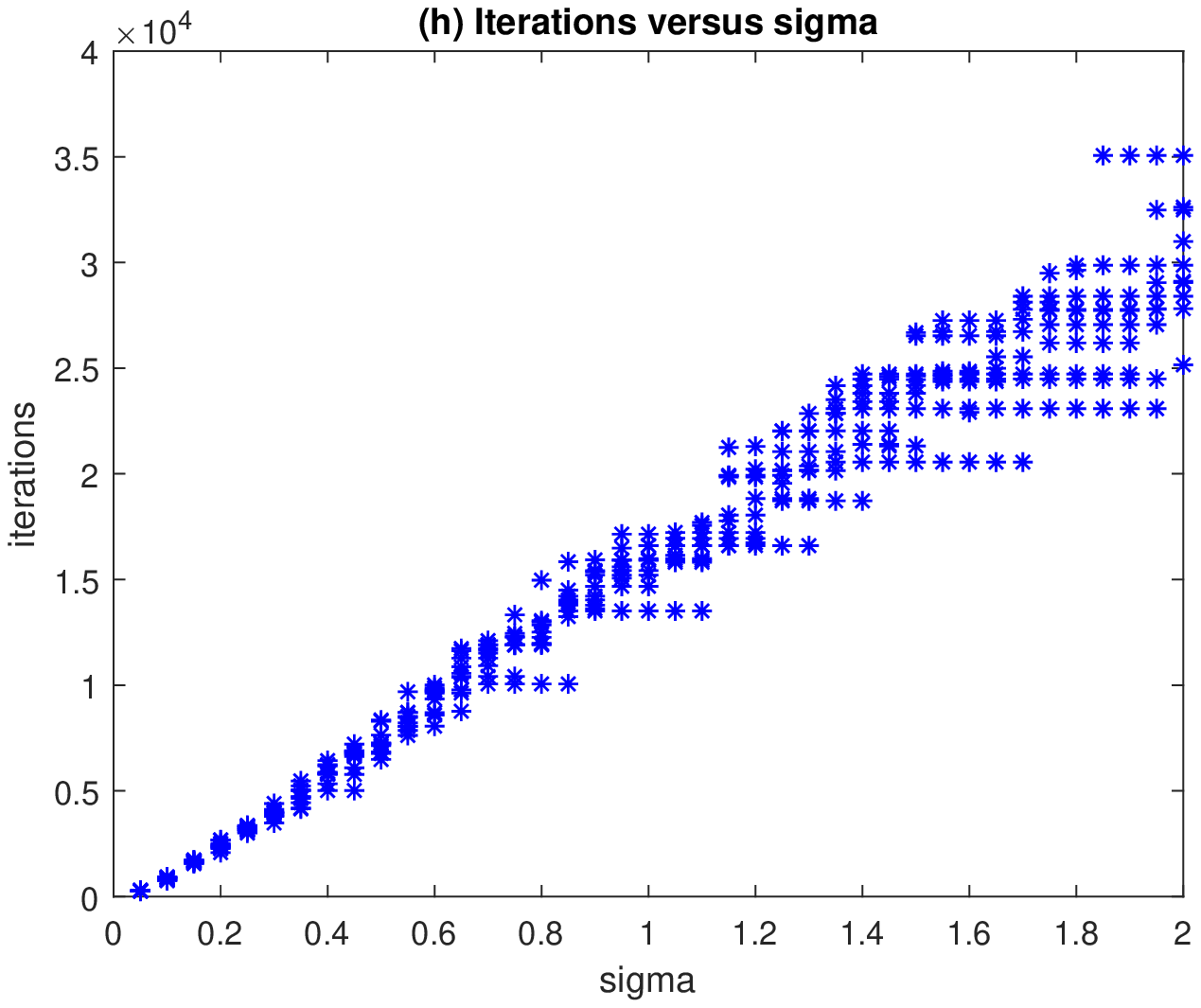}} \quad \subfigure[]{\includegraphics[height=3cm]{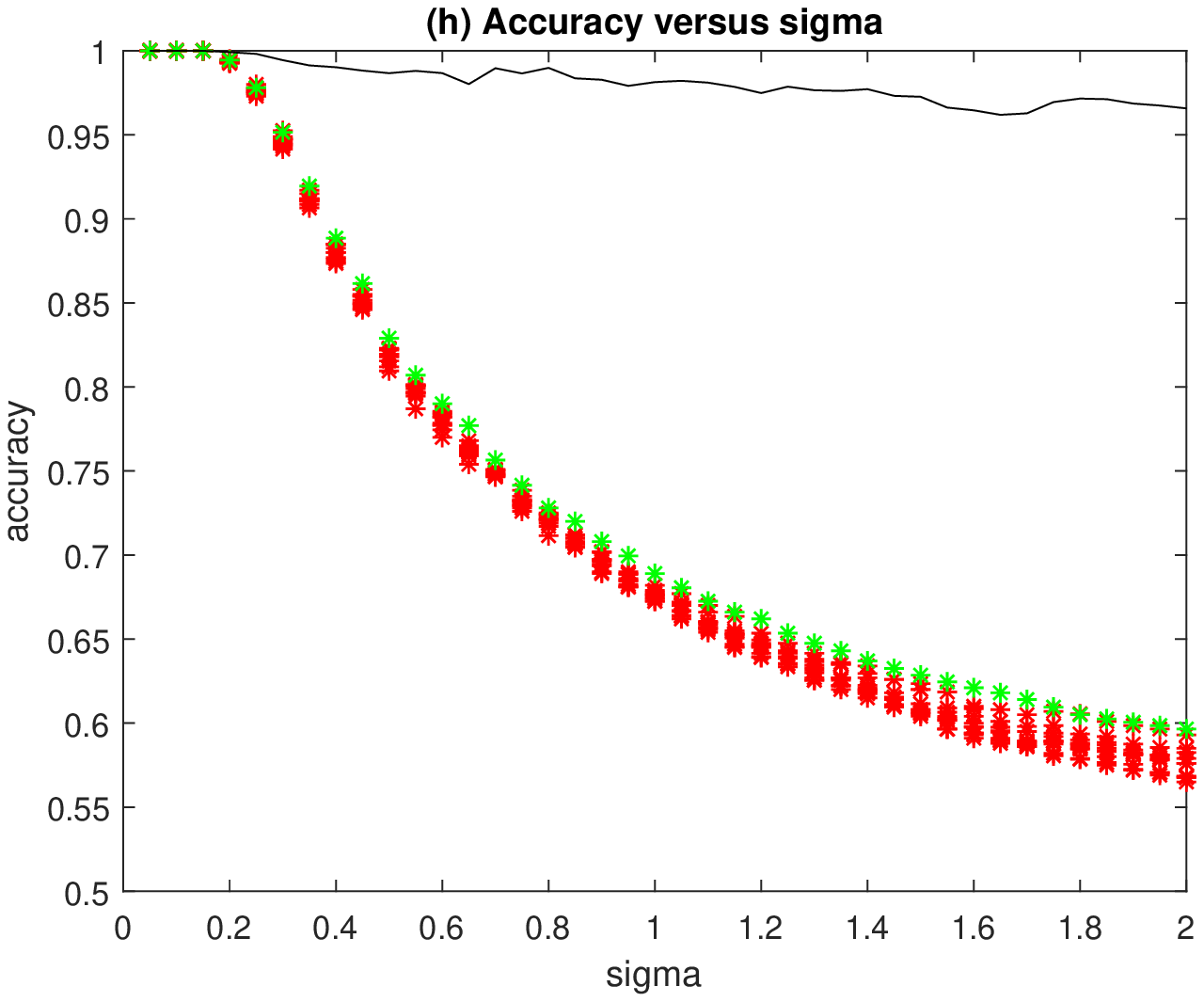}}
\end{center}
\caption{Performance of stopping criterion \eqref{eq: practical_termination_test} on a mixture of Gaussians as $\sigma$ is varied.  Plots $(a),(b)$ are logistic and $(c),(d)$ are hinge.  All plots show tests for values of $\sigma$ equally spaced from 0.05 to 2.0.  For each value of $\sigma$, ten trials were run. Plots $(a),(c)$ show the relationship between $\sigma$ and $k$, the iteration number when \eqref{eq: practical_termination_test} first holds. Plots $(b),(d)$ show the accuracy as red asterisks.  The green asterisks show the accuracy of the optimal classifier.  The black curve on the right is the ratio of the average accuracy (over 10 trials) of the classifier 
when \eqref{eq: practical_termination_test} holds to the accuracy of the optimal classifier.}
\label{fig:acctime}
\end{figure*}

\paragraph{Measuring the accuracy.} In all the experiments, we measure the performance of a method with a score, generally known as ``accuracy," that is the fraction correct on a large validation set drawn from the same distribution as the training data.  Thus, 1.0 is perfect accuracy, while 0.5 means that $\btheta_k$ is no better at classifying than random guessing.  It is important to note that even on data for which the means $\bmu_0,\bmu_1$ are known a priori ({\em e.g.}, synthetic data), the score of the optimal $\btheta^*$
will not be 1.0 because the large validation set itself is noisy.  

We center the data so that the linear classifier is homogeneous. In a preliminary phase, 100 samples are drawn from the training set. From this, $\bmu_0$ and $\bmu_1$ are estimated, and then the average of these estimates is used to offset training instances during SGD.

\paragraph{Parameter settings.} After centering, the vectors $\btheta$ and $\bxi$ scale inversely, so the step-size parameter $\alpha$ should scale as $1/\sigma^2$.  Therefore, we take the step-size to be 
$\tilde\alpha/\tilde\sigma^2$.  Here, $\tilde\sigma^2$ is the average of $\norm{\bzeta_j-\tilde\bmu_{y_j}}^2$, and $\tilde\bmu_i$ ($i=0$ or $i=1$)  is the estimate of $\bmu_i$, averaged over the two classes. We compute the quantities $\tilde\sigma^2$ and $\tilde\bmu_i$ using the 100 samples described in the preceding paragraph. Note that for the Gaussian mixture model, the expected value of $\tilde\sigma^2$ is $\sigma^2d$.  
For the synthetic data, the means and variances are known exactly a priori, so the estimation procedures described in the previous two paragraphs are unnecessary.  However, we used them anyway in order to be consistent with the tests on the realistic data.

The parameter $\tilde\alpha$ described in the last paragraph is a scale-free tuning parameter.
It is known (see, e.g., \cite{Nemirovski_Robust_Stochastic_1}) that a smaller $\tilde\alpha$ corresponds to more iterations but greater ultimate accuracy under a reasonable model of the data.  Our termination test is obviously sensitive to the choice of $\tilde\alpha$: the condition $\bxi_{k+1}^T\btheta_k\ge 1$ cannot hold unless $\norm{\btheta_k}\ge 1/\norm{\bxi_{k+1}}$, but $\bE\left[\norm{\btheta_k}\right]\le O(\alpha k)$. See also Theorems~\ref{thm:low} and \ref{thm:high}.
On the other hand, SVS is only mildly sensitive to $\tilde\alpha$, according to our testing.  Indeed, there is an upper bound of $pl$ on the total number of iterations possible before termination using the SVS condition, independent of $\tilde\alpha$ and of all other aspects of the problem.  The dependence of the termination test on $\tilde\alpha$ is evidently desirable because the user is presumably seeking greater accuracy when a smaller value of $\tilde\alpha$ is selected.

\subsection{Experiments with synthetic data}
\label{sec:comp-sim}

\begin{figure*}[htp!]
\centering
  \includegraphics[height=2.5cm]{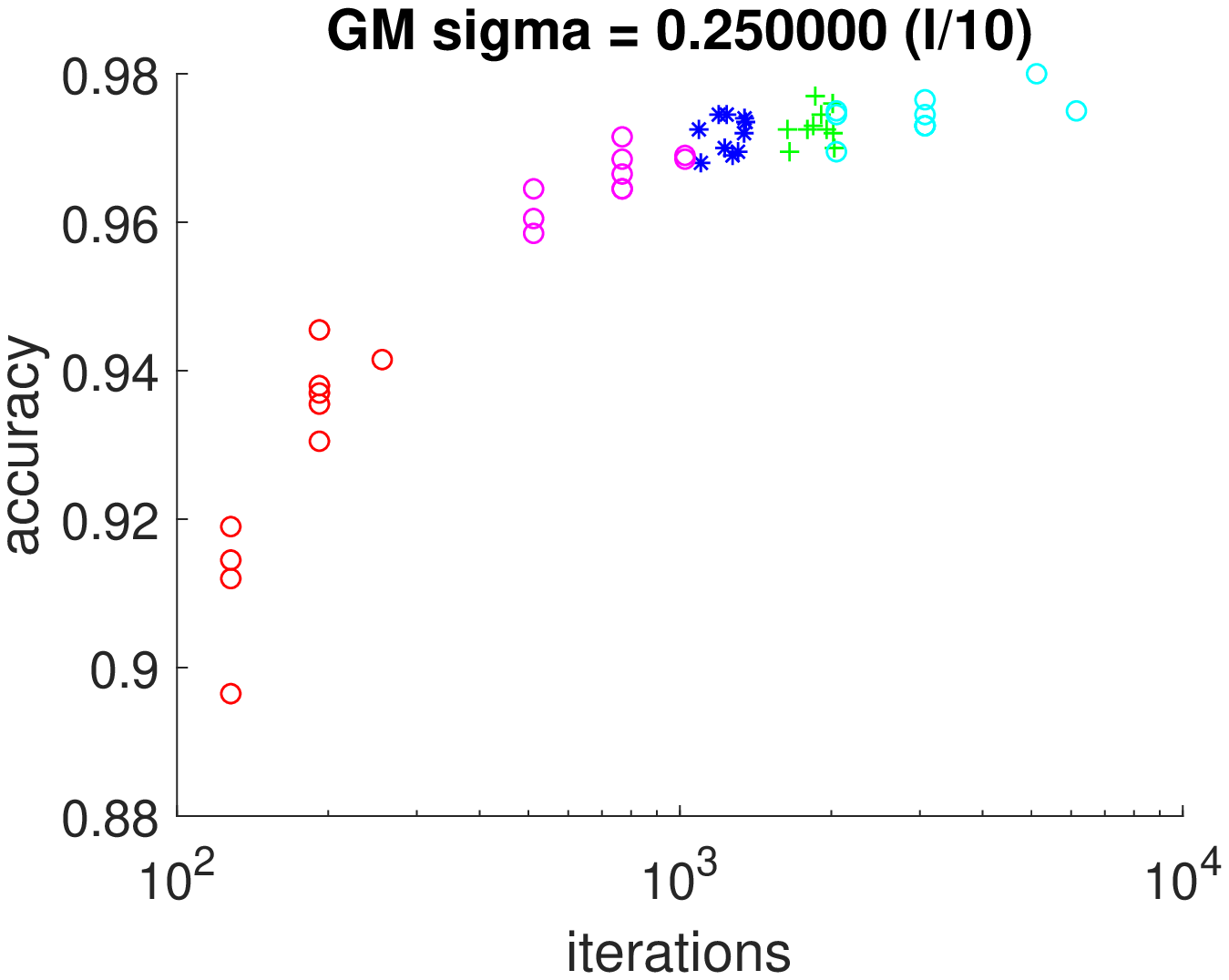} \qquad \includegraphics[height=2.5cm]{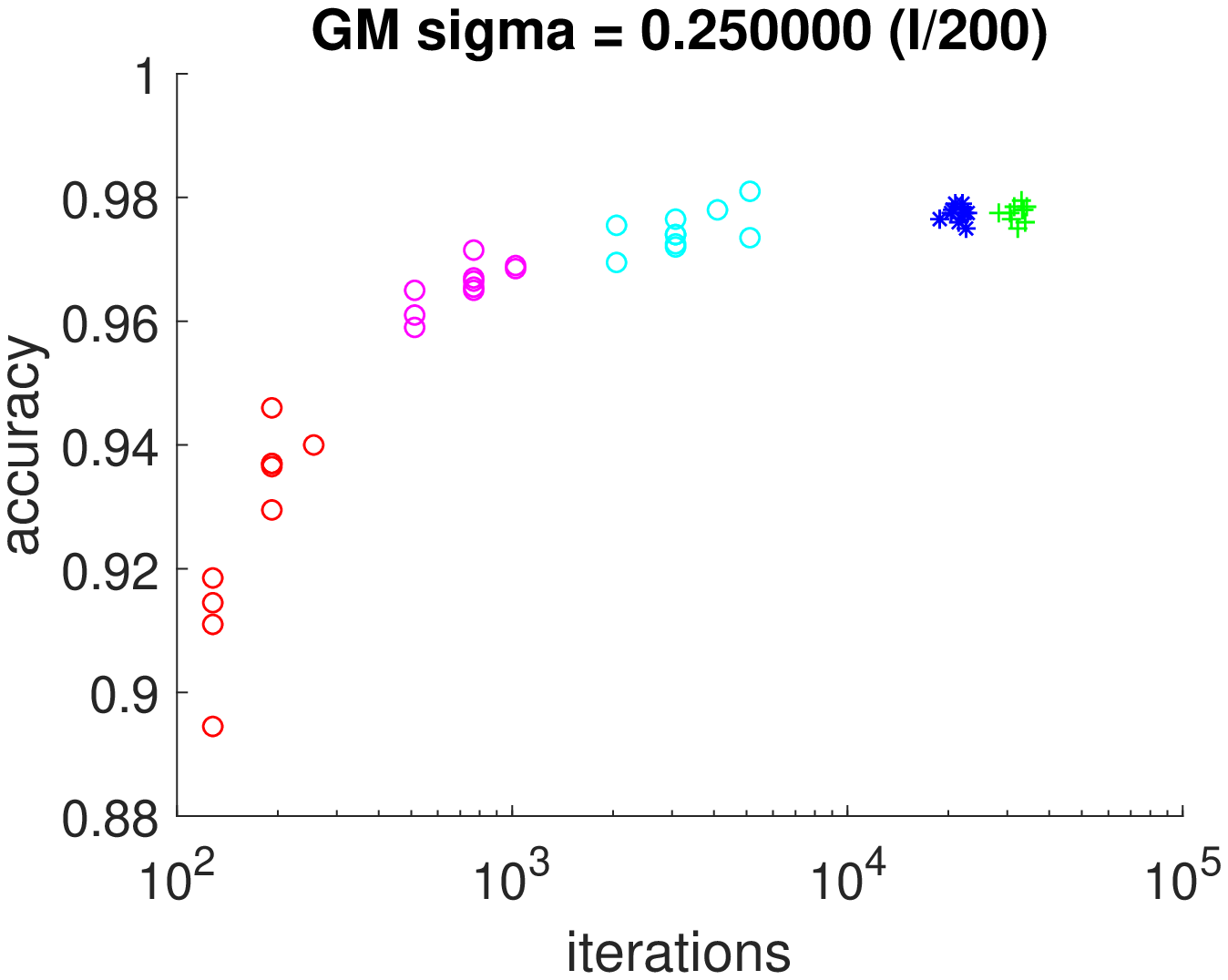} \qquad
\includegraphics[height=2.5cm]{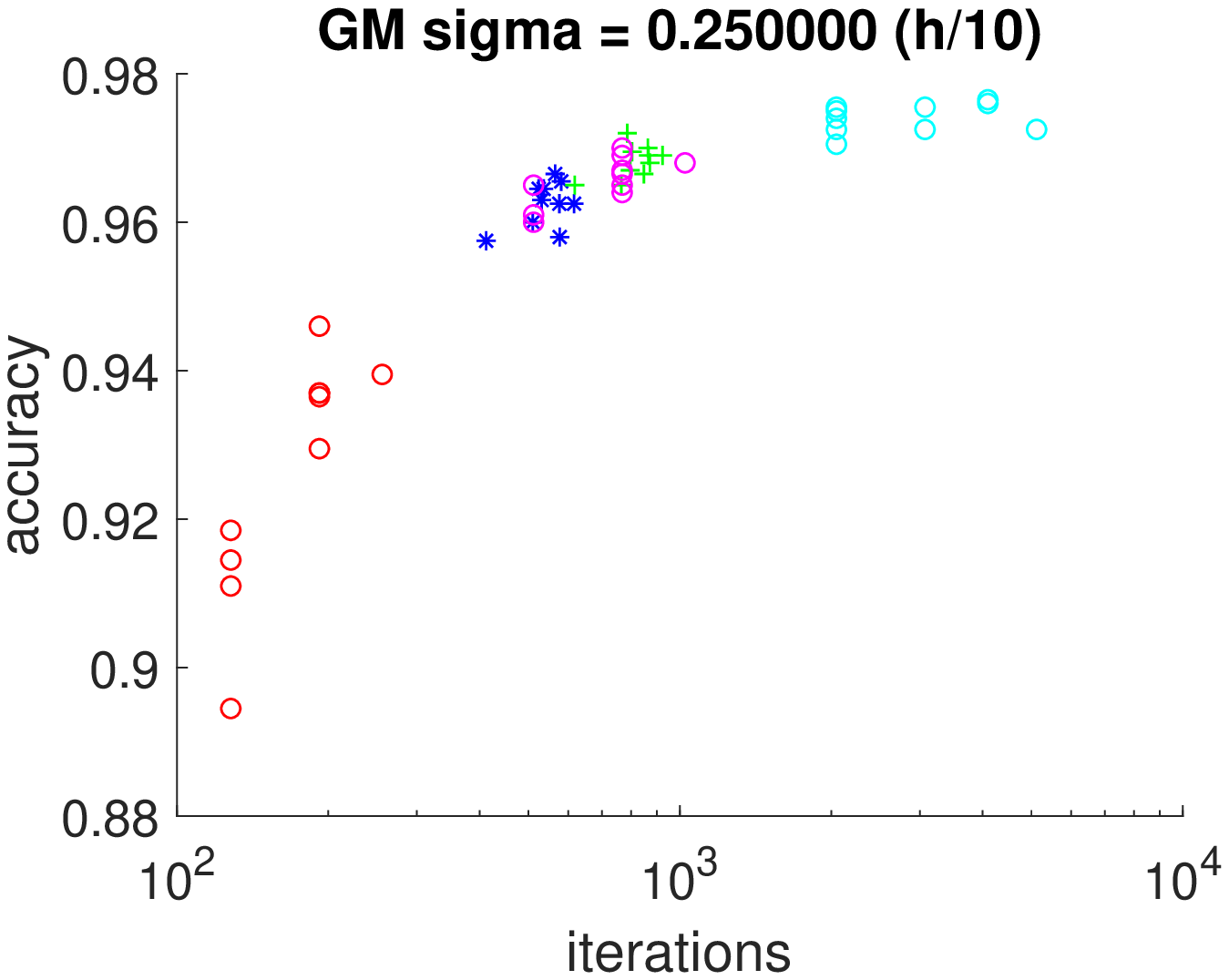}\qquad \includegraphics[height=2.5cm]{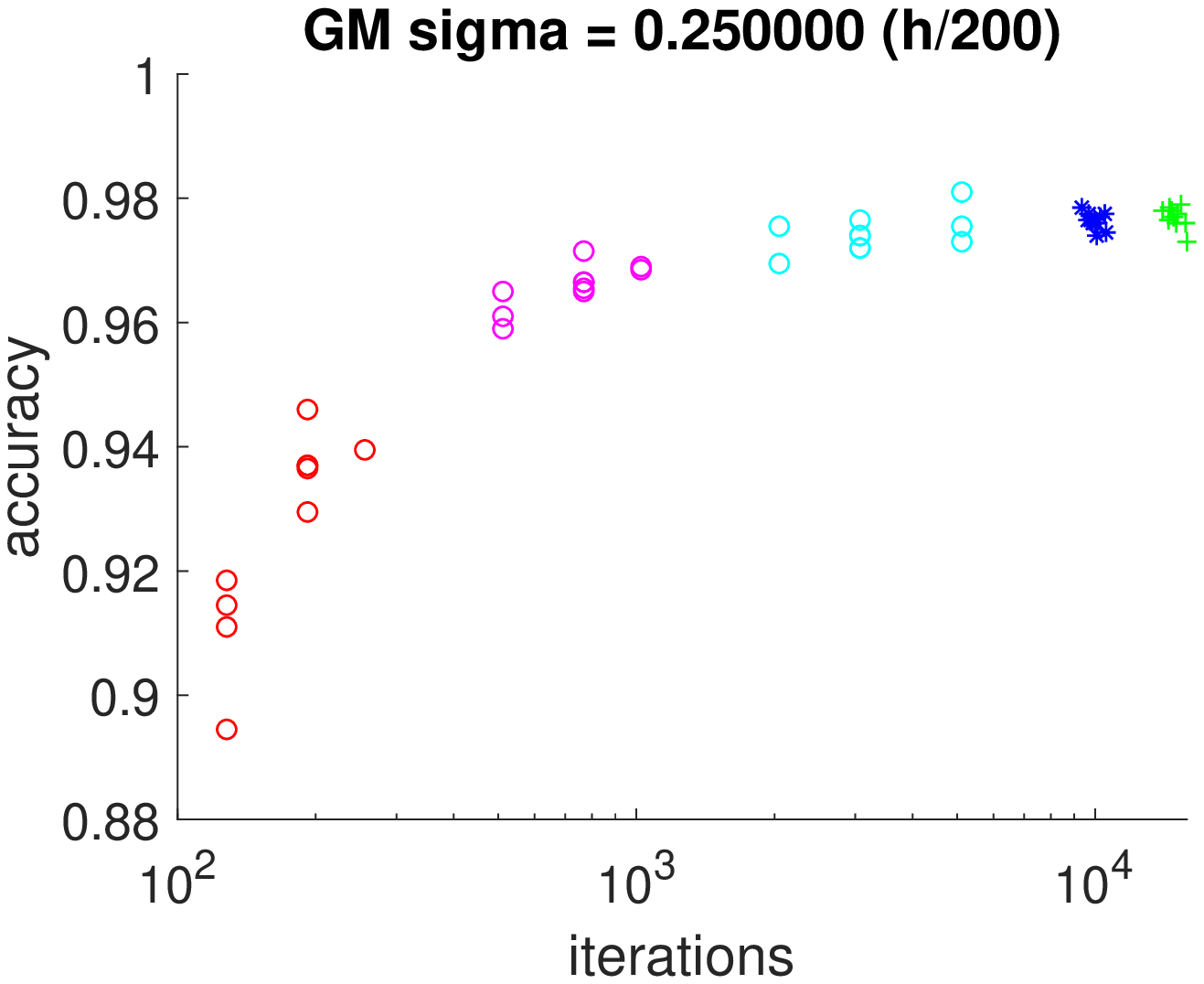} \\ \vspace{2 mm}
\includegraphics[height=2.5cm]{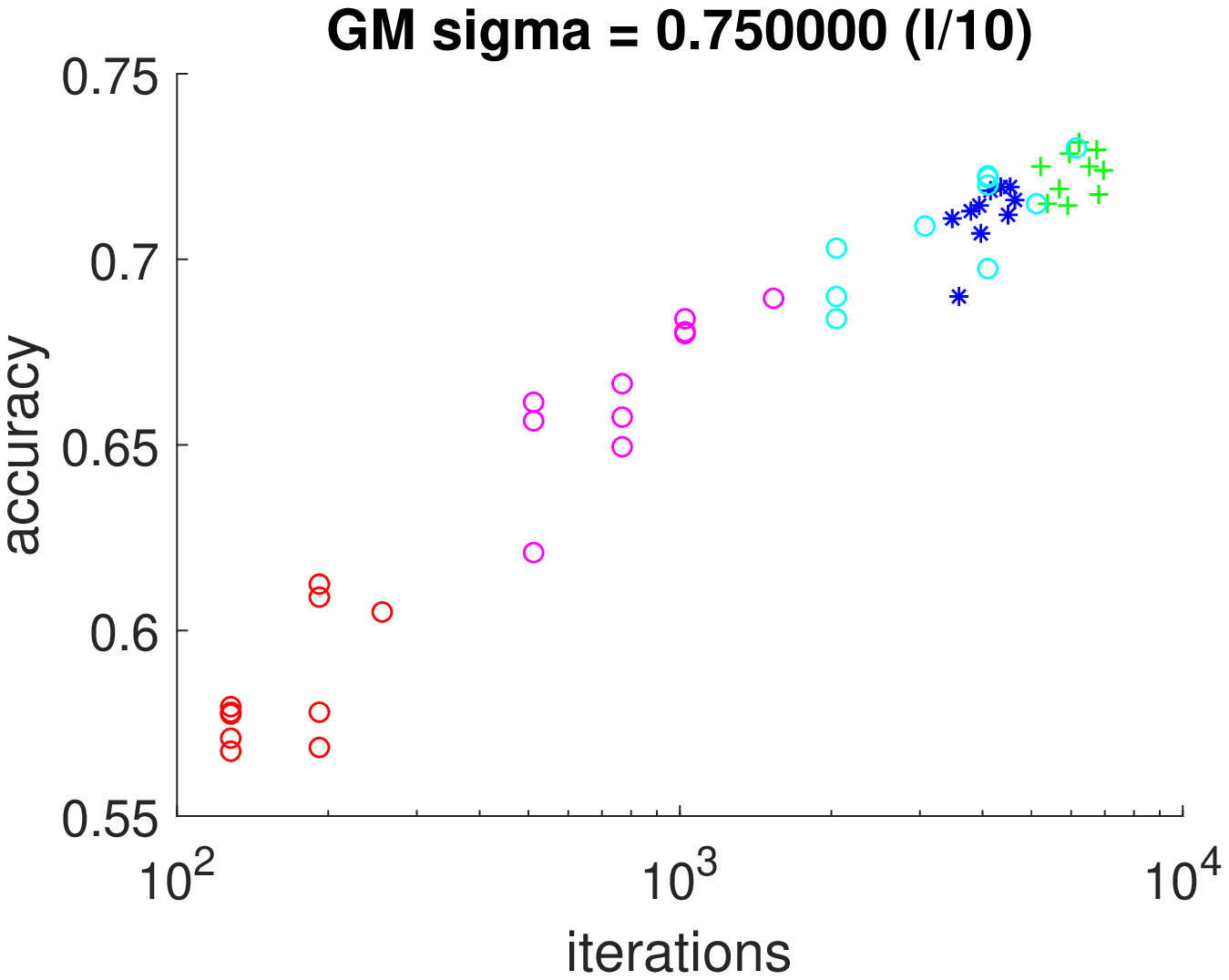}\qquad \includegraphics[height=2.5cm]{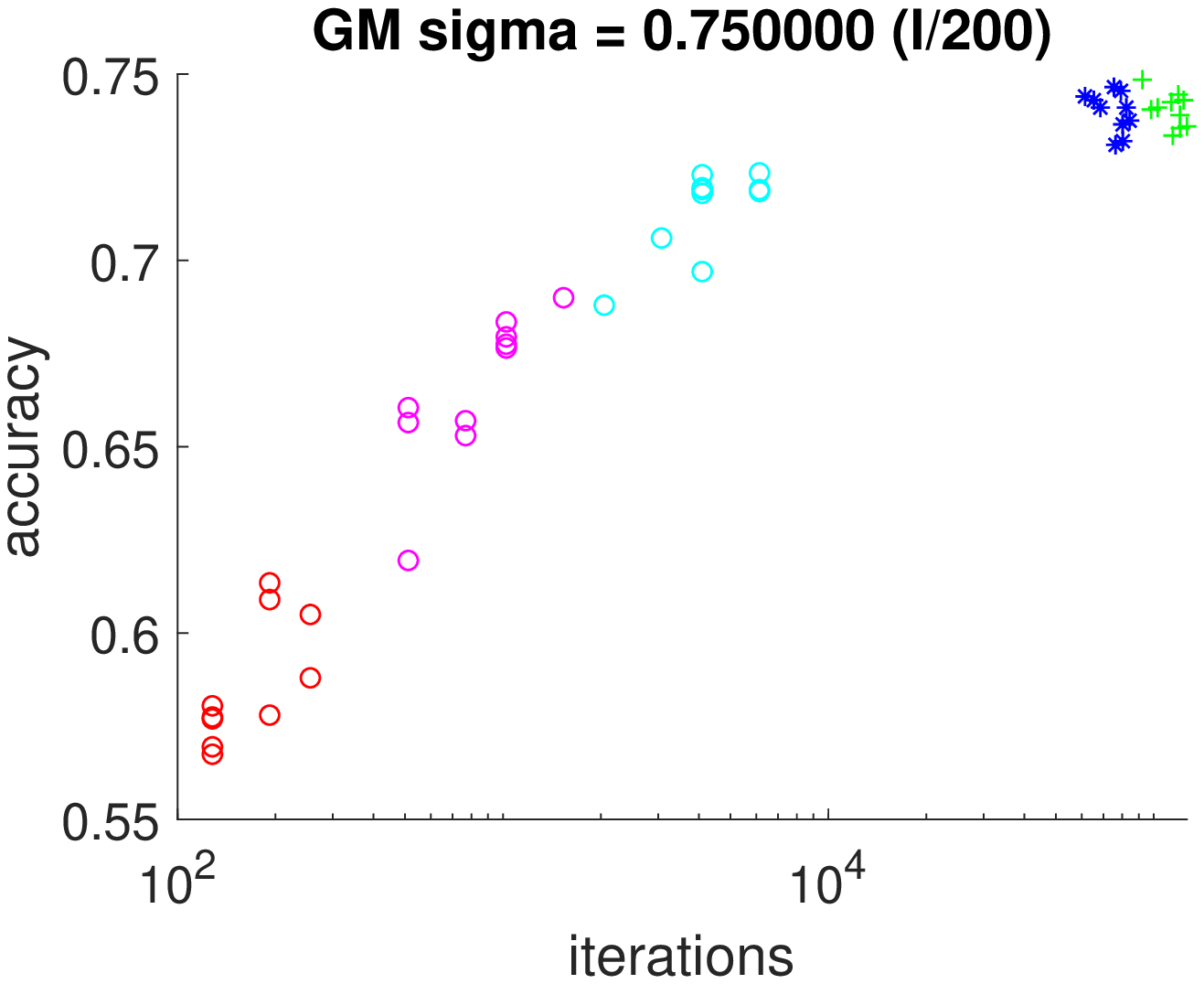} \qquad
\includegraphics[height=2.5cm]{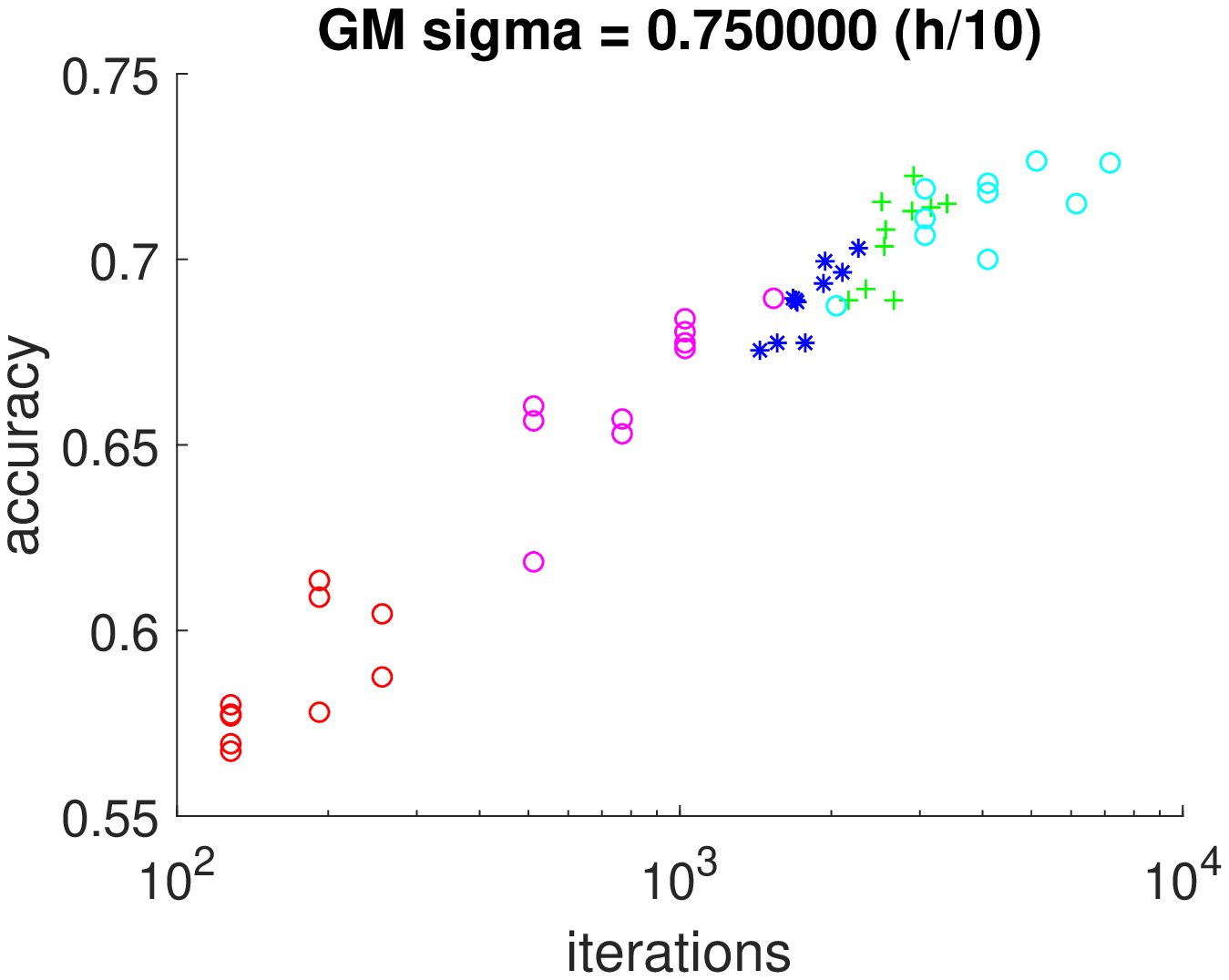}\qquad \includegraphics[height=2.5cm]{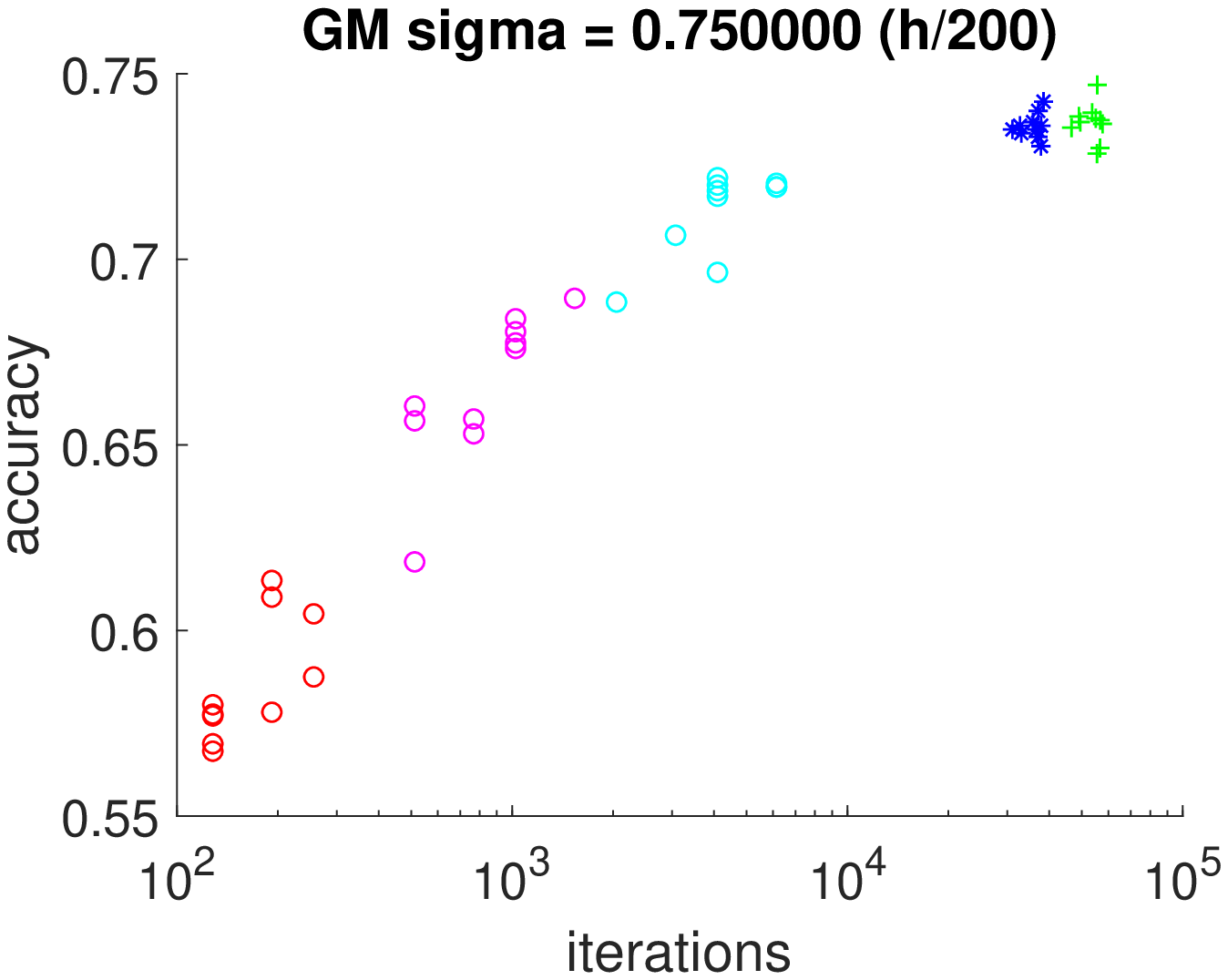} \\ \vspace{2 mm}
 \includegraphics[height=2.5cm]{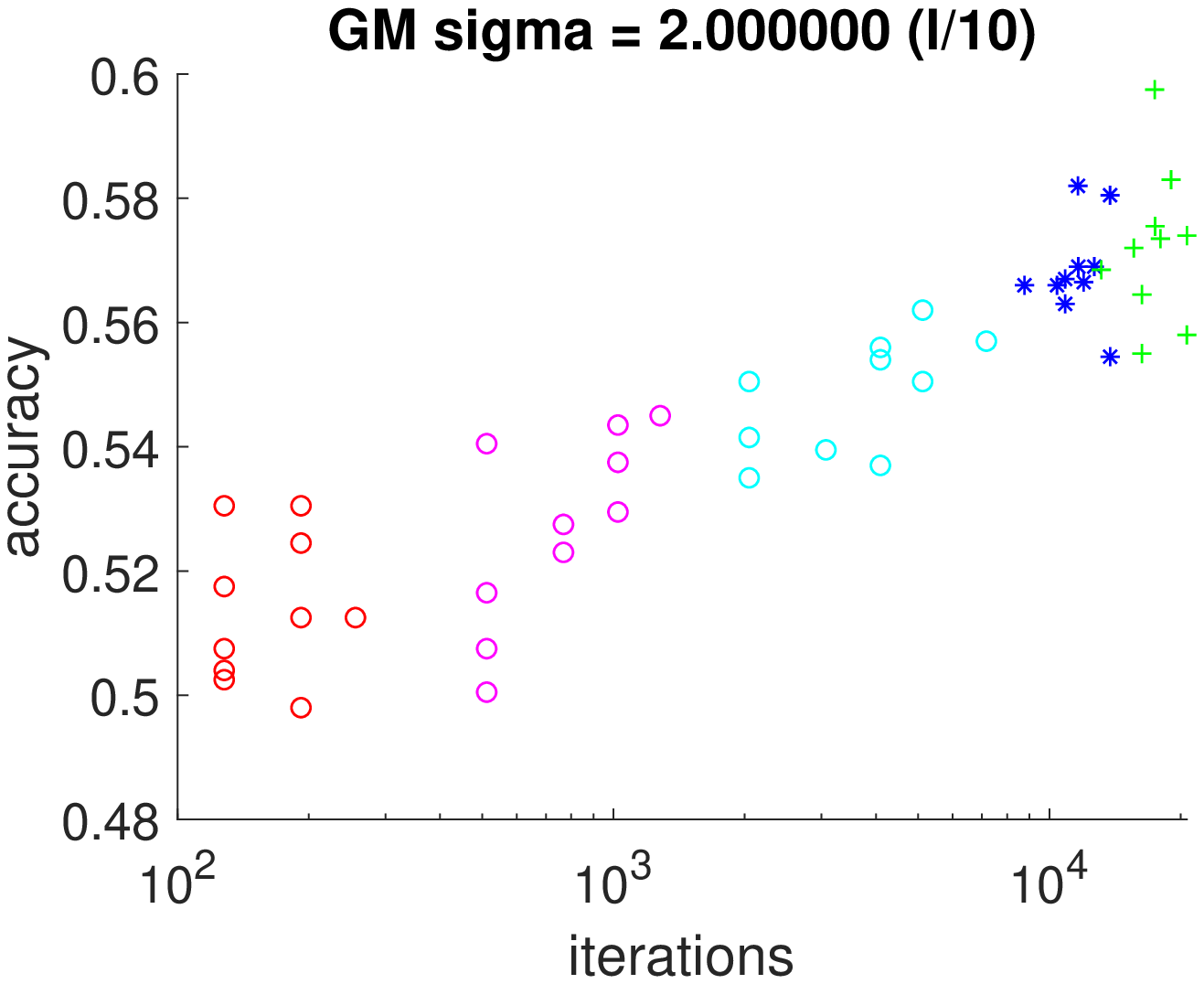}\qquad \includegraphics[height=2.5cm]{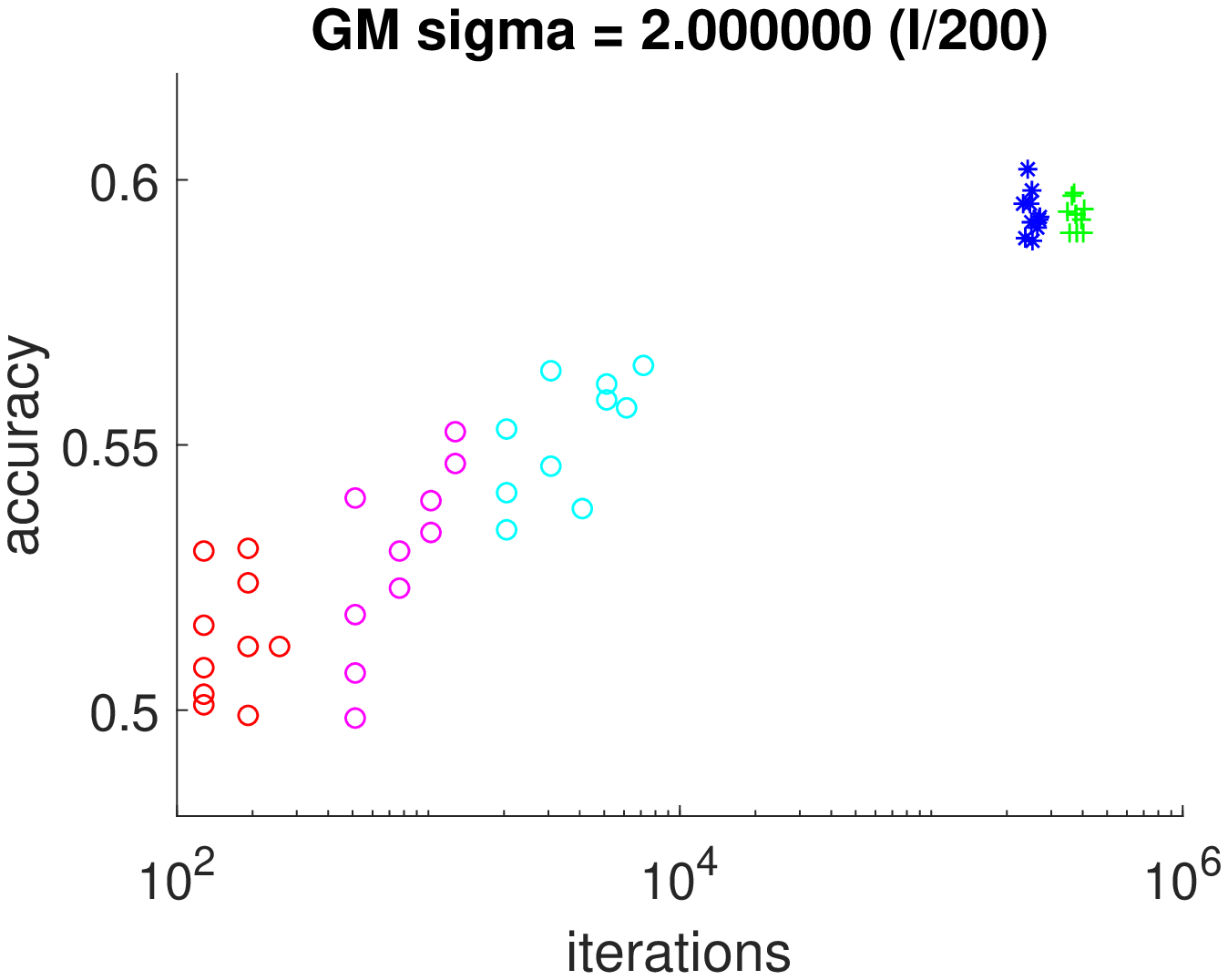} \qquad    \includegraphics[height=2.5cm]{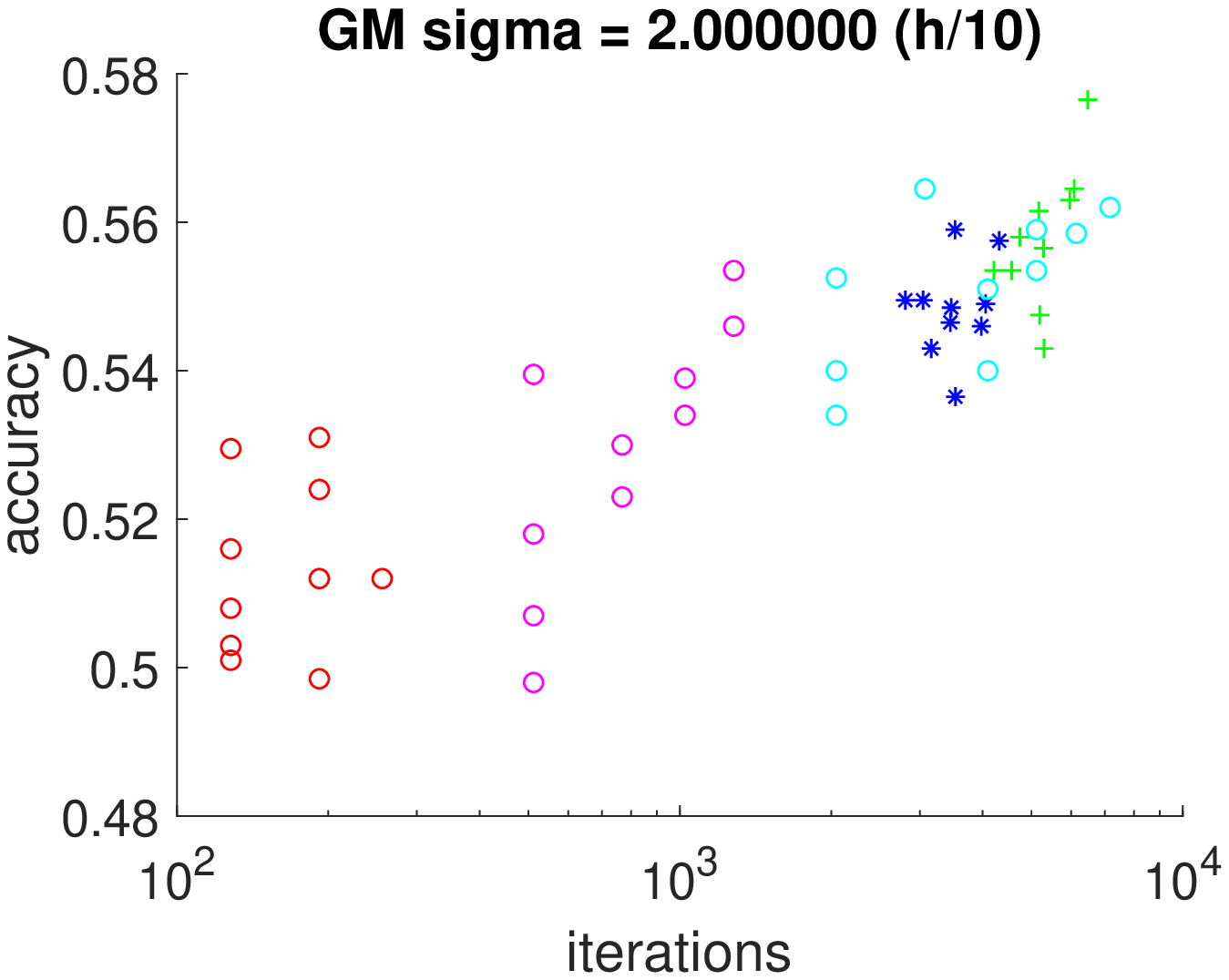}\qquad \includegraphics[height=2.5cm]{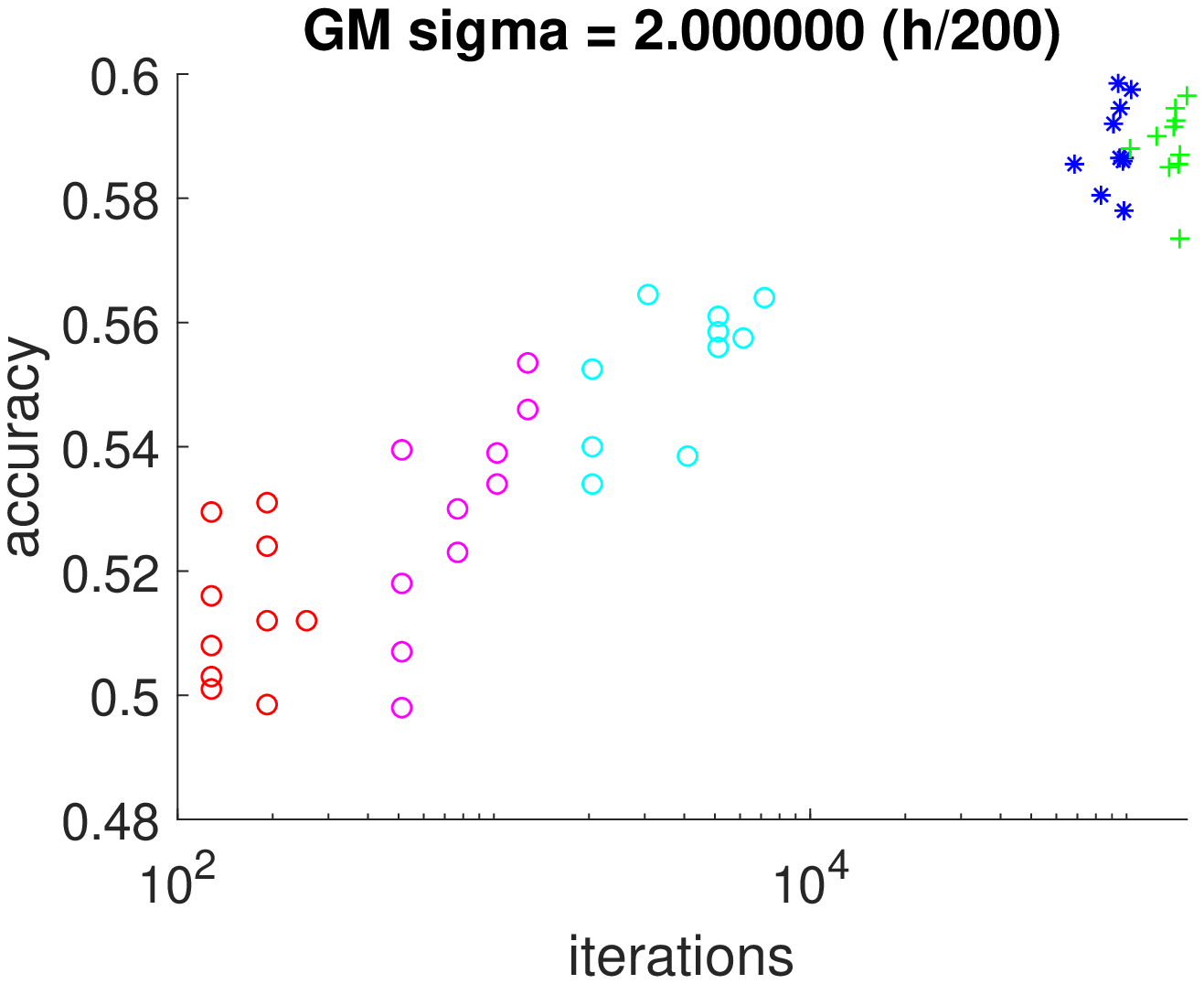}
 \caption{Each plot shows 10 random runs of SGD applied to normally distributed data with indicated values of $\sigma$ and for a fixed dimension $d=500$. For each of the ten runs, five termination tests corresponding to five colors were applied.  SVS was tried with $p=32, 128, 512$, depicted as red, magenta and cyan circles respectively.  Test \eqref{eq: practical_termination_test} is indicated with a blue asterisk. A green `+' corresponds to termination after $1.5k$ iterations, where $k$ is the iteration index that \eqref{eq: practical_termination_test} first holds.  The notation $(l/200)$ means logistic loss with $\tilde\alpha=1/200$; simillarly $(h/10)$ means hinge loss with $\tilde\alpha = 1/10$, and so on.}
\label{fig:normal_scatter}
\end{figure*}

\paragraph{Normal distribution.} We generated test and training data using a mixture of Gaussians given by $N(\bz,\sigma^2I)$ for the 0-class and $N(\bm{e}_1,\sigma^2I)$ for the 1-class, where $\bm{e}_1 = (1,0,\hdots, 0)^T \in \R^d$.  

In Fig.~\ref{fig:acctime}, we present the running time and accuracy (fraction correct) of our termination test for a fixed dimension $d=500$ and $\sigma$ ranging from $0.05$ to $2$. We record 10 runs for each value of $\sigma$. The performance of the classifier when our termination test \eqref{eq: practical_termination_test} holds almost matches the optimal classifier; in particular, the averaged accuracy of our classifier/accuracy of the optimal classifier over the 10 runs, black curve in Fig.~\ref{fig:black}, never dips below $0.95$.

In Fig.~\ref{fig:normal_scatter}, we compare performance of \eqref{eq: practical_termination_test} against SVS termination. One axis shows accuracy while the other shows iteration count. We continued to run SGD for an additional $1.5k$ iterations where $k$ is the first iteration at which \eqref{eq: practical_termination_test} holds (green '+') to test whether accuracy improves after termination.
The tests (for several values of $\sigma$, both hinge and logistic, and two values of $\tilde\alpha$) in Fig.~\ref{fig:normal_scatter} indicate that \eqref{eq: practical_termination_test} is more accurate than SVS, more predictable (i.e., there is less spread in the scatter plot), and that running until $1.5k$ iterations does not significantly improve the solution.  As expected,
for a large $\tilde\alpha$, \eqref{eq: practical_termination_test} requires fewer iterations than SVS with $p=512$, while the opposite relationship holds for a small $\tilde\alpha.$

\begin{figure}
    \begin{center}
    \includegraphics[height=2.5cm]{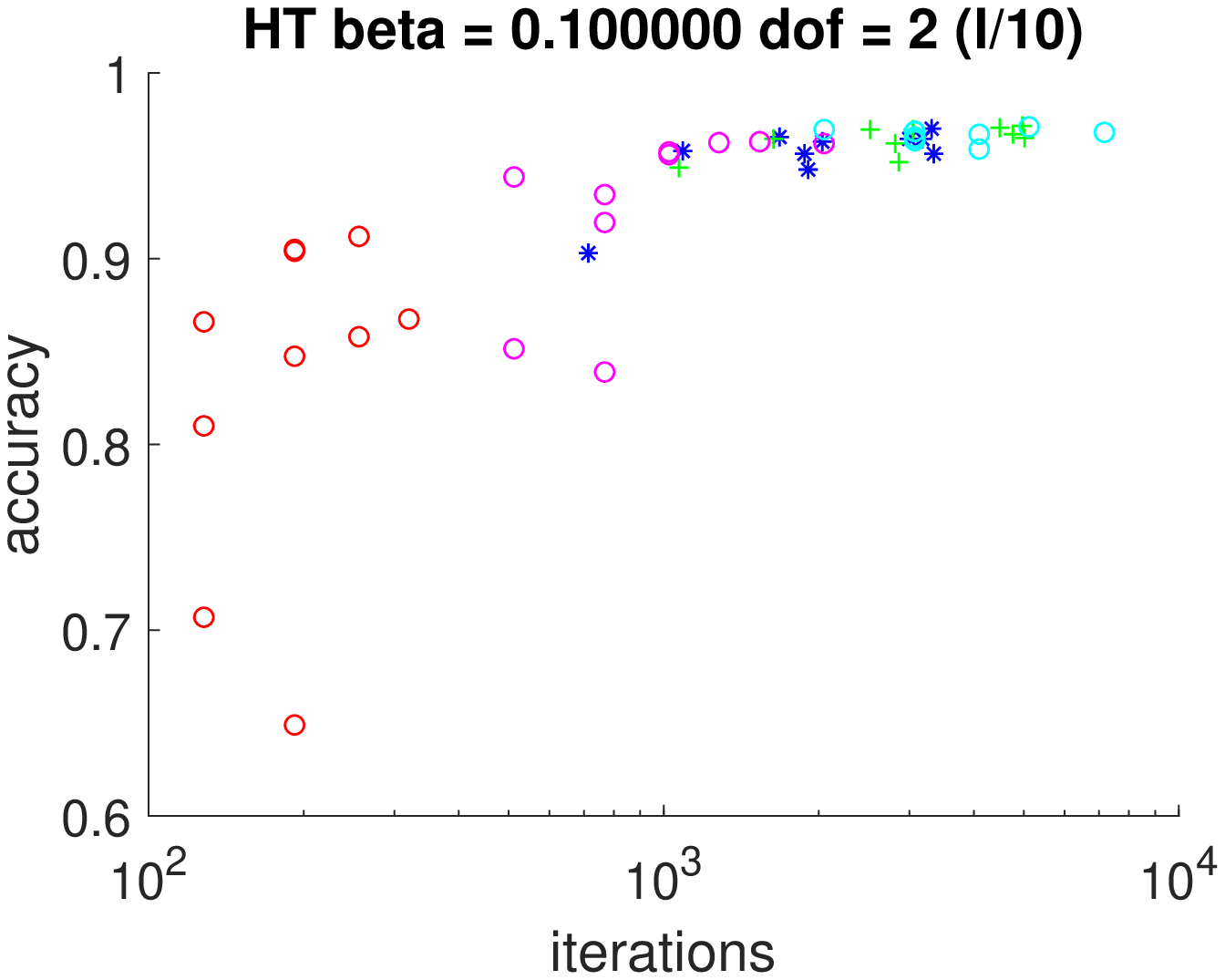} \qquad \includegraphics[height=2.5cm]{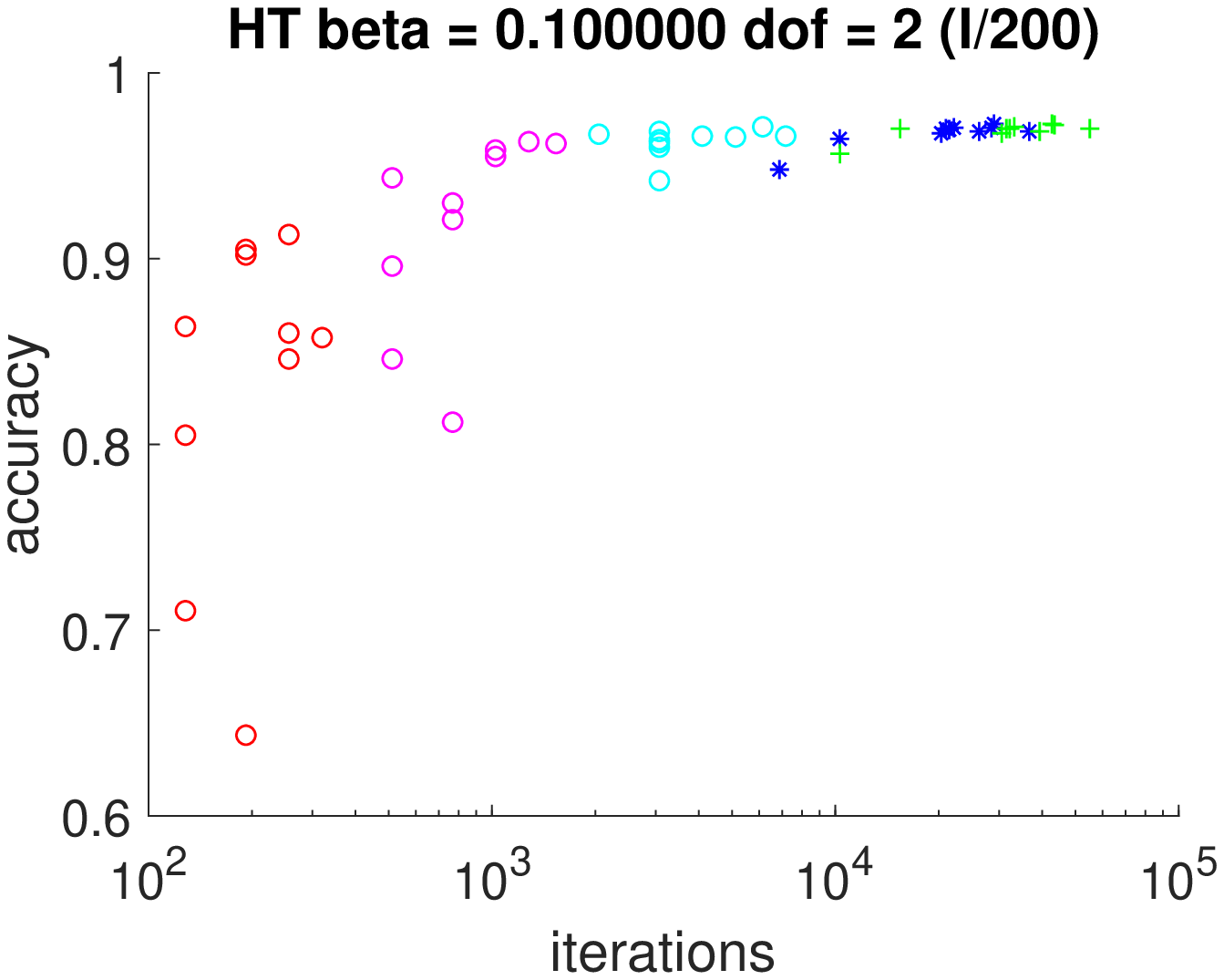} \qquad \includegraphics[height=2.5cm]{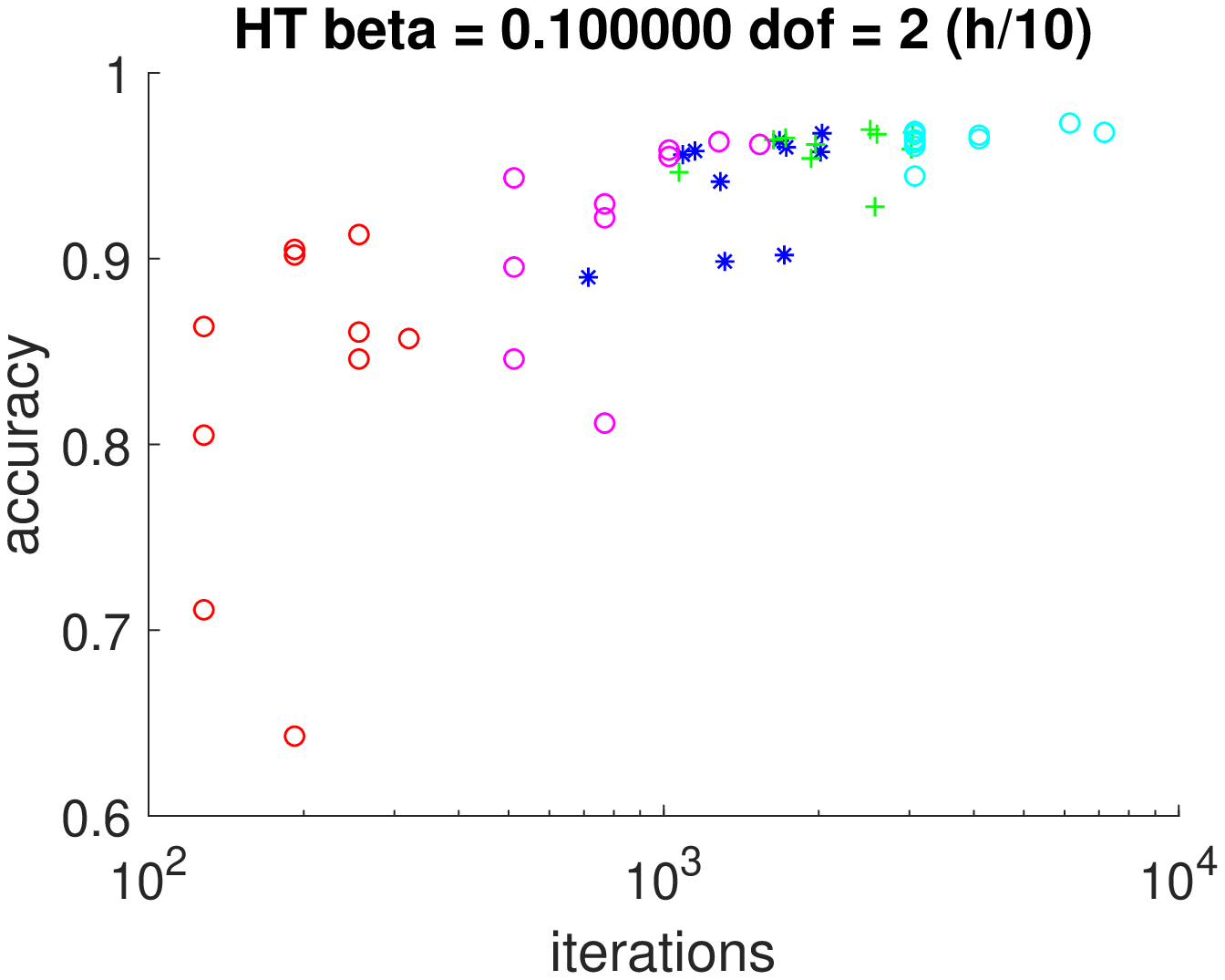} \qquad  
 \includegraphics[height=2.5cm]{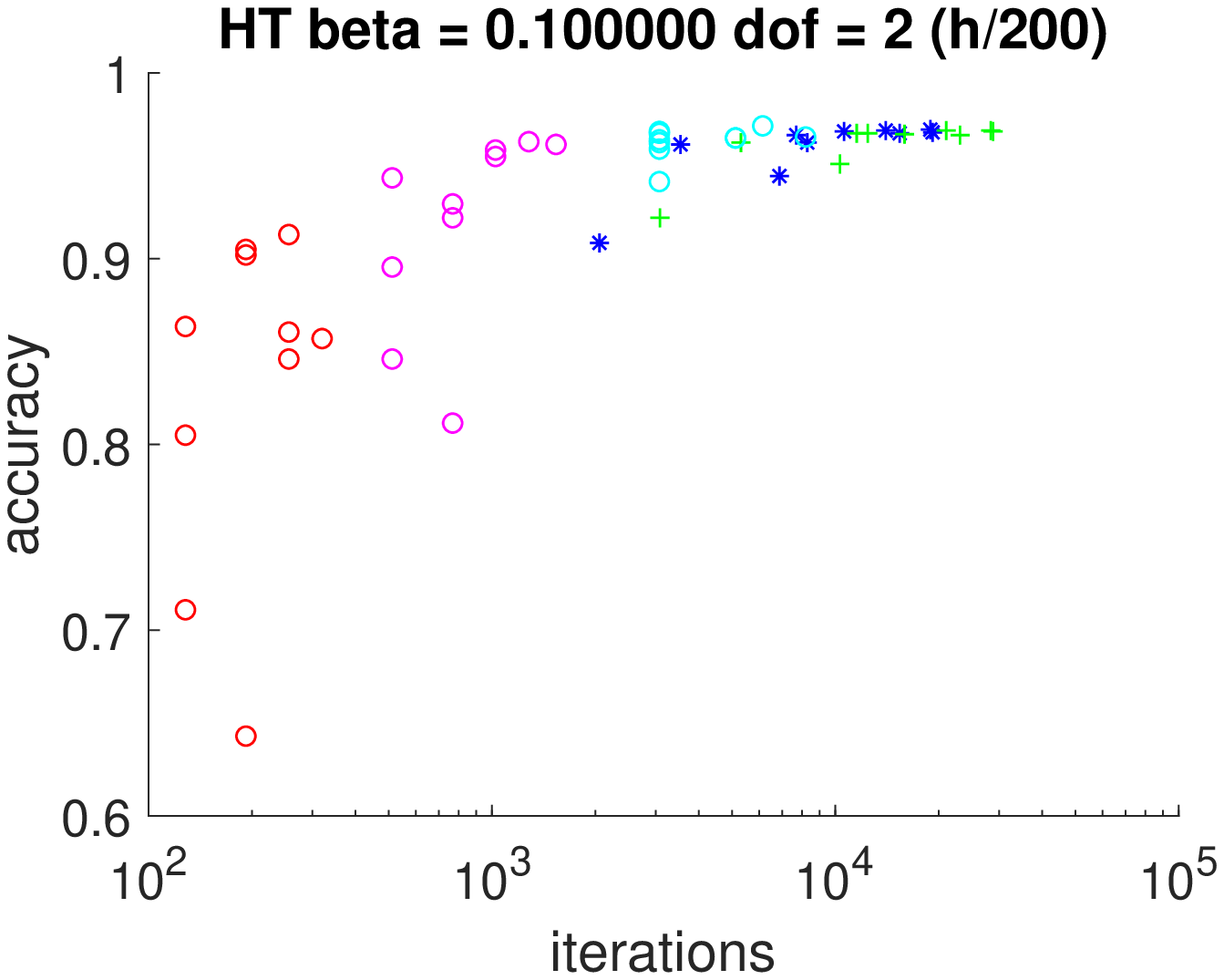} \qquad       \includegraphics[height=2.5cm]{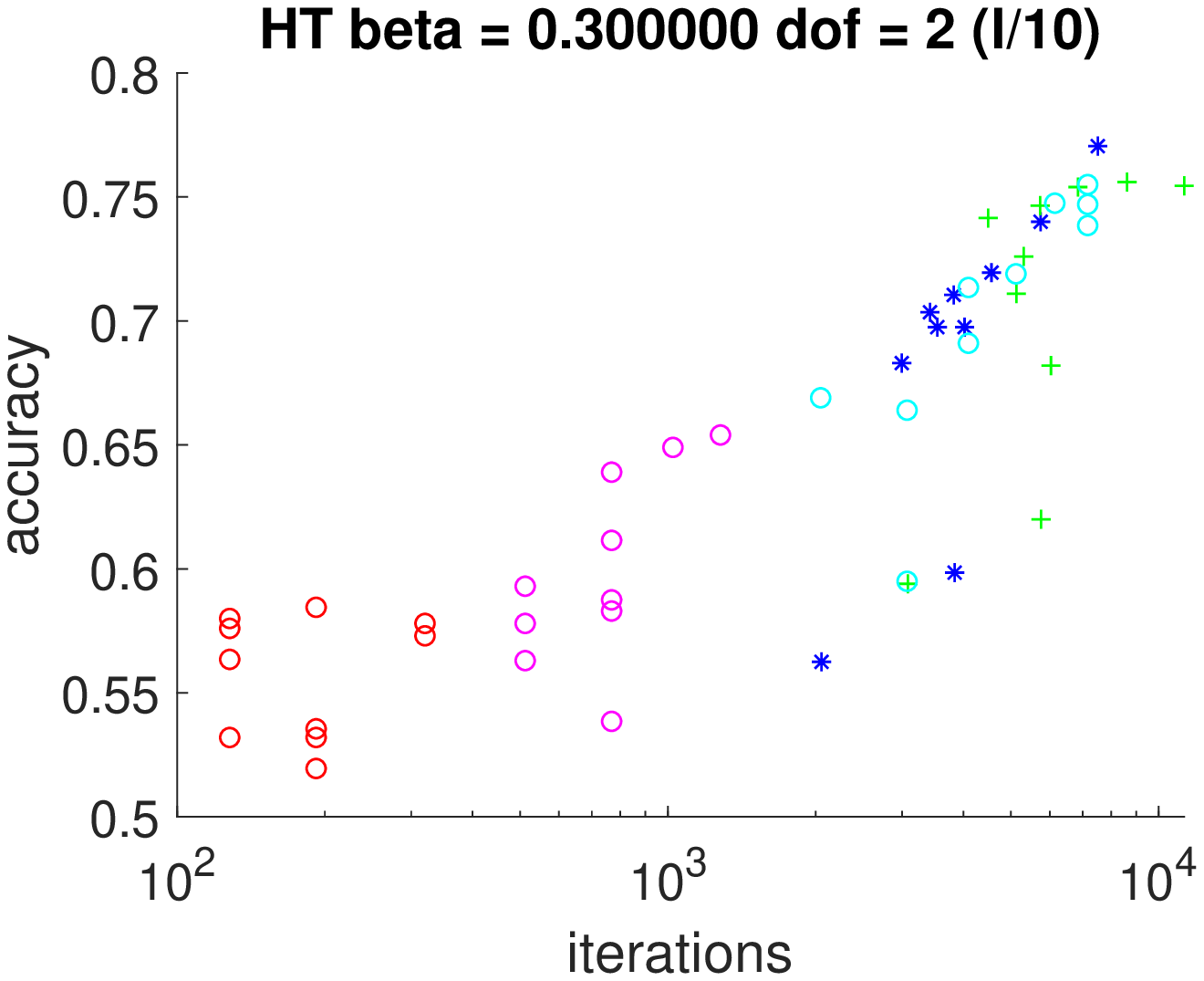} \qquad \includegraphics[height=2.5cm]{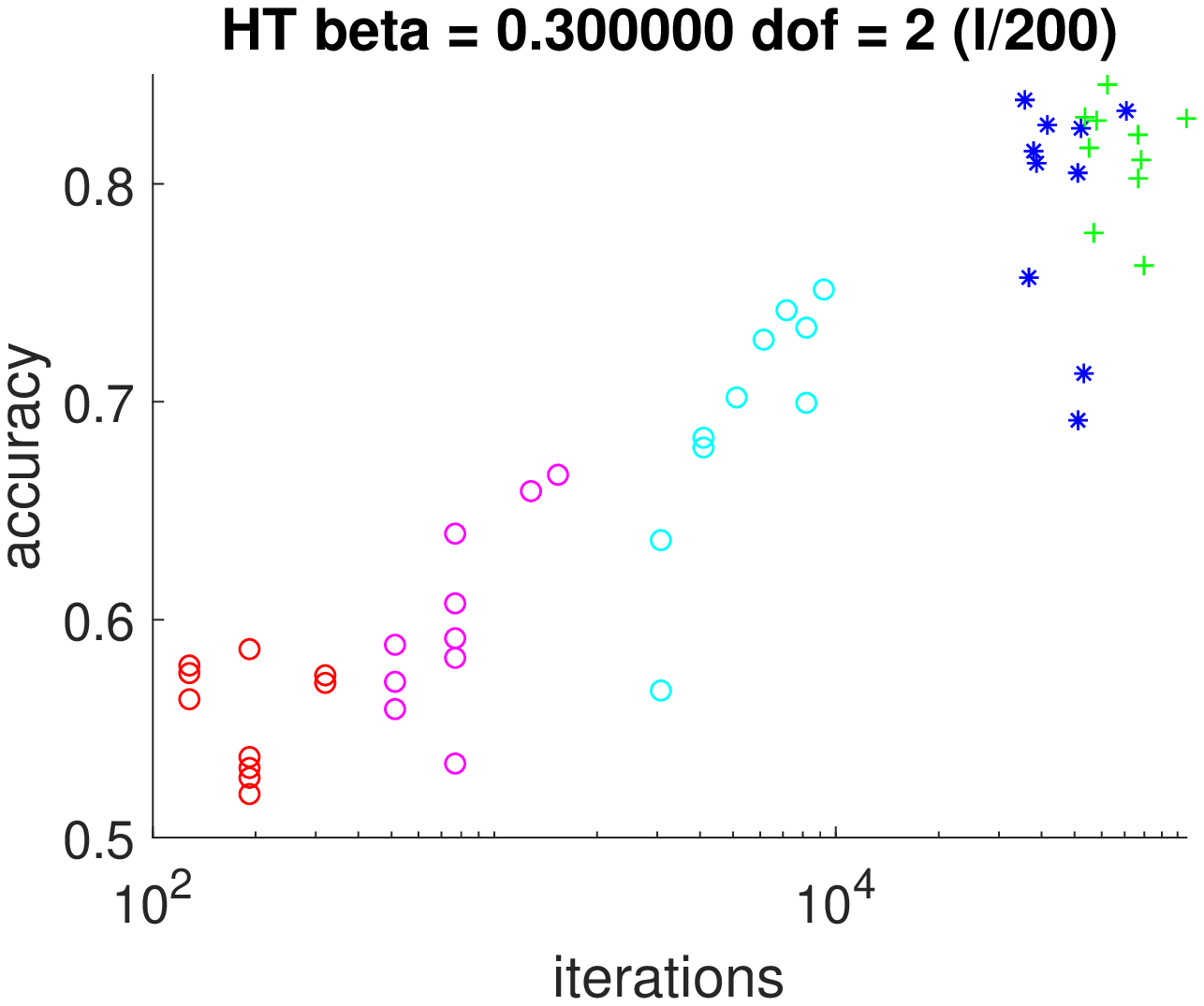} \qquad \includegraphics[height=2.5cm]{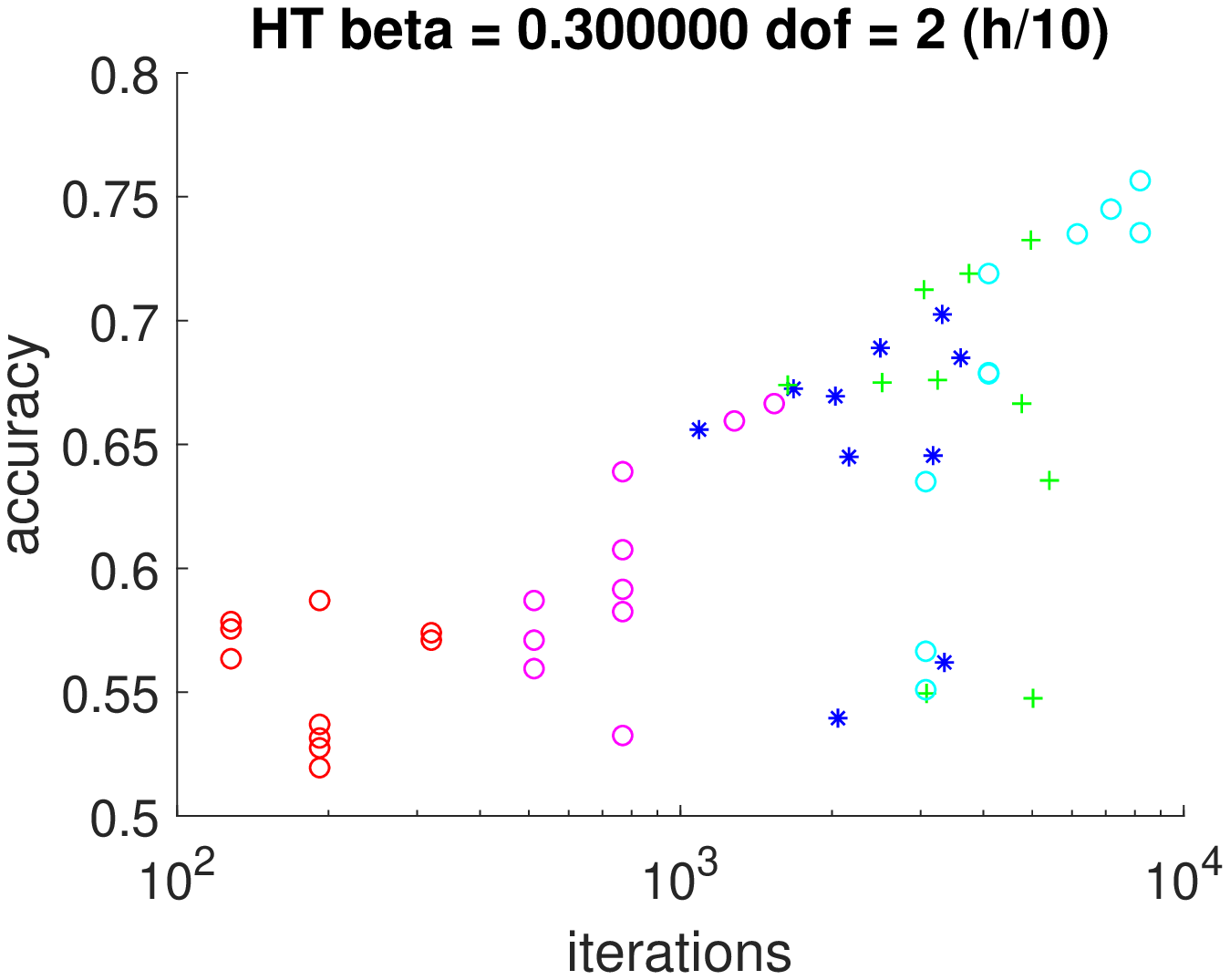} \qquad \includegraphics[height=2.5cm]{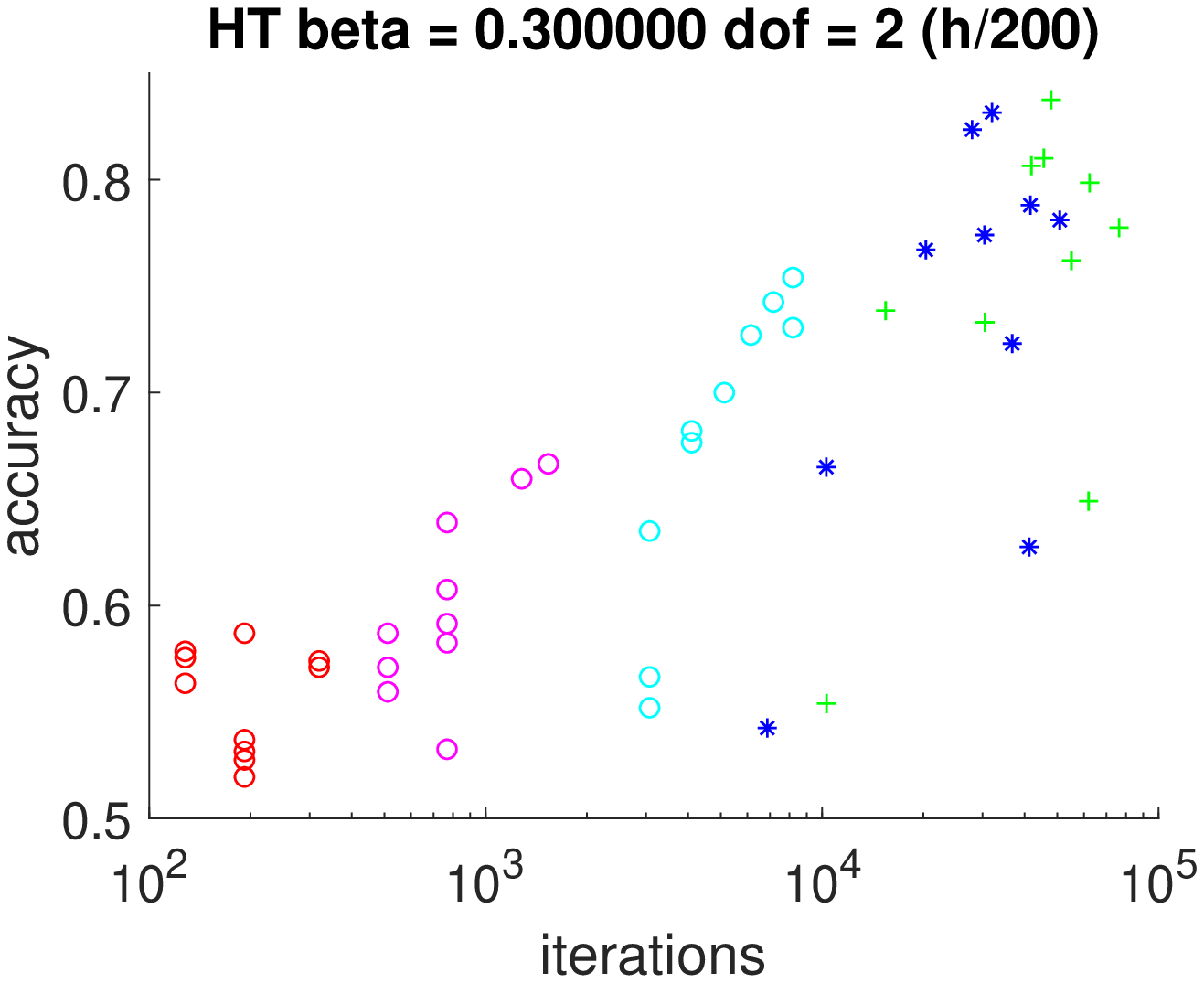} \\ \vspace{2 mm}
    \includegraphics[height=2.5cm]{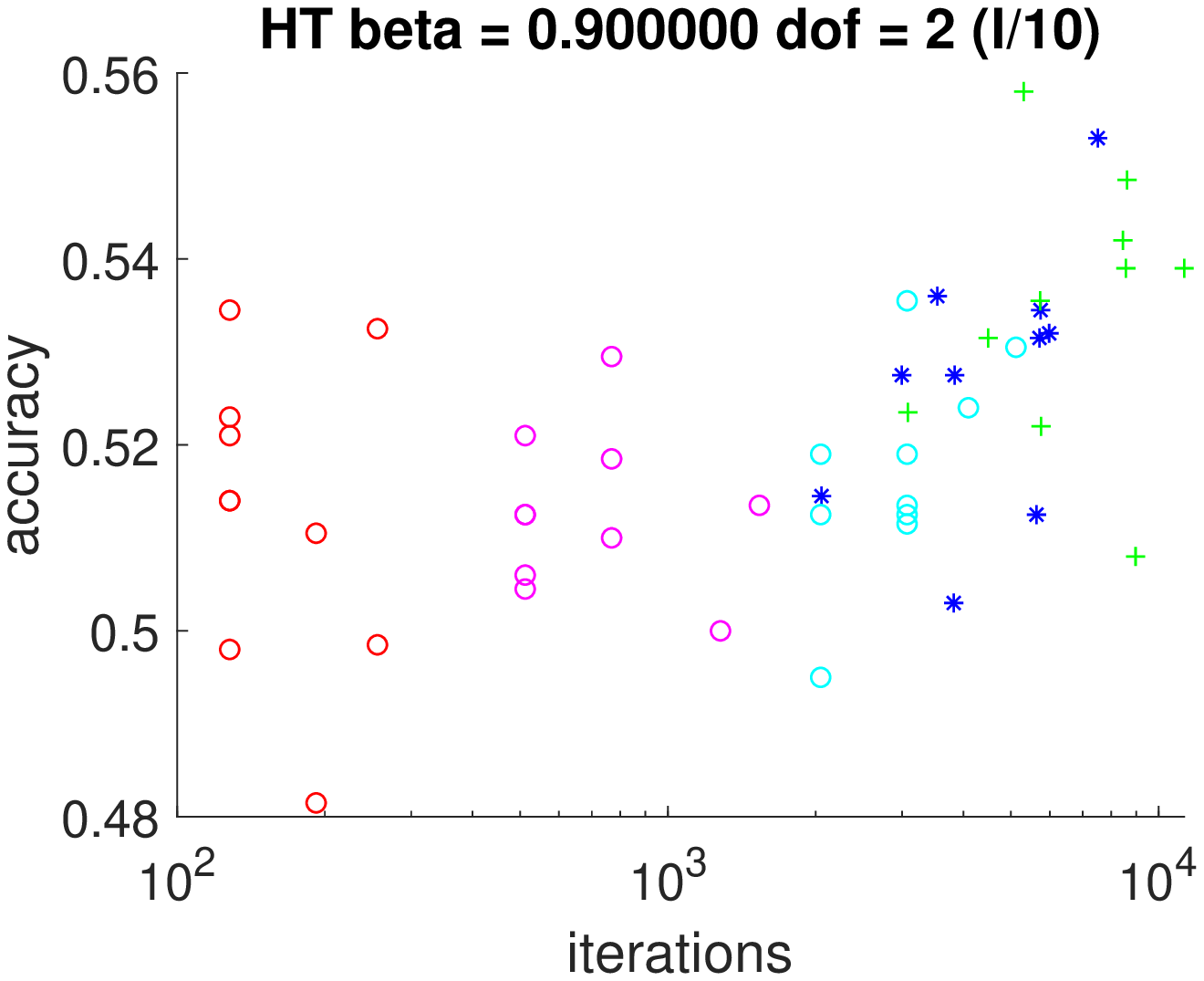} \qquad \includegraphics[height=2.5cm]{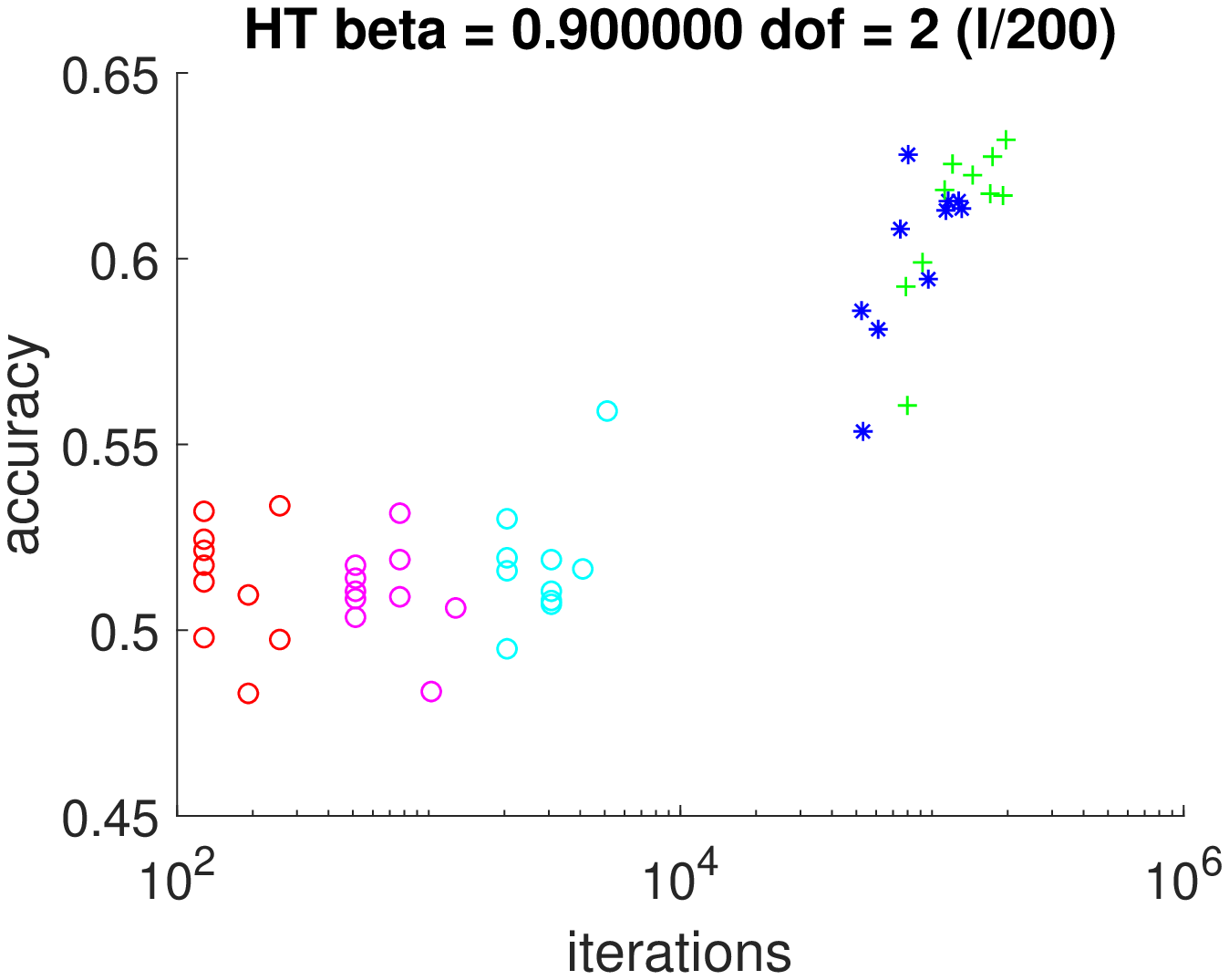} \qquad \includegraphics[height=2.5cm]{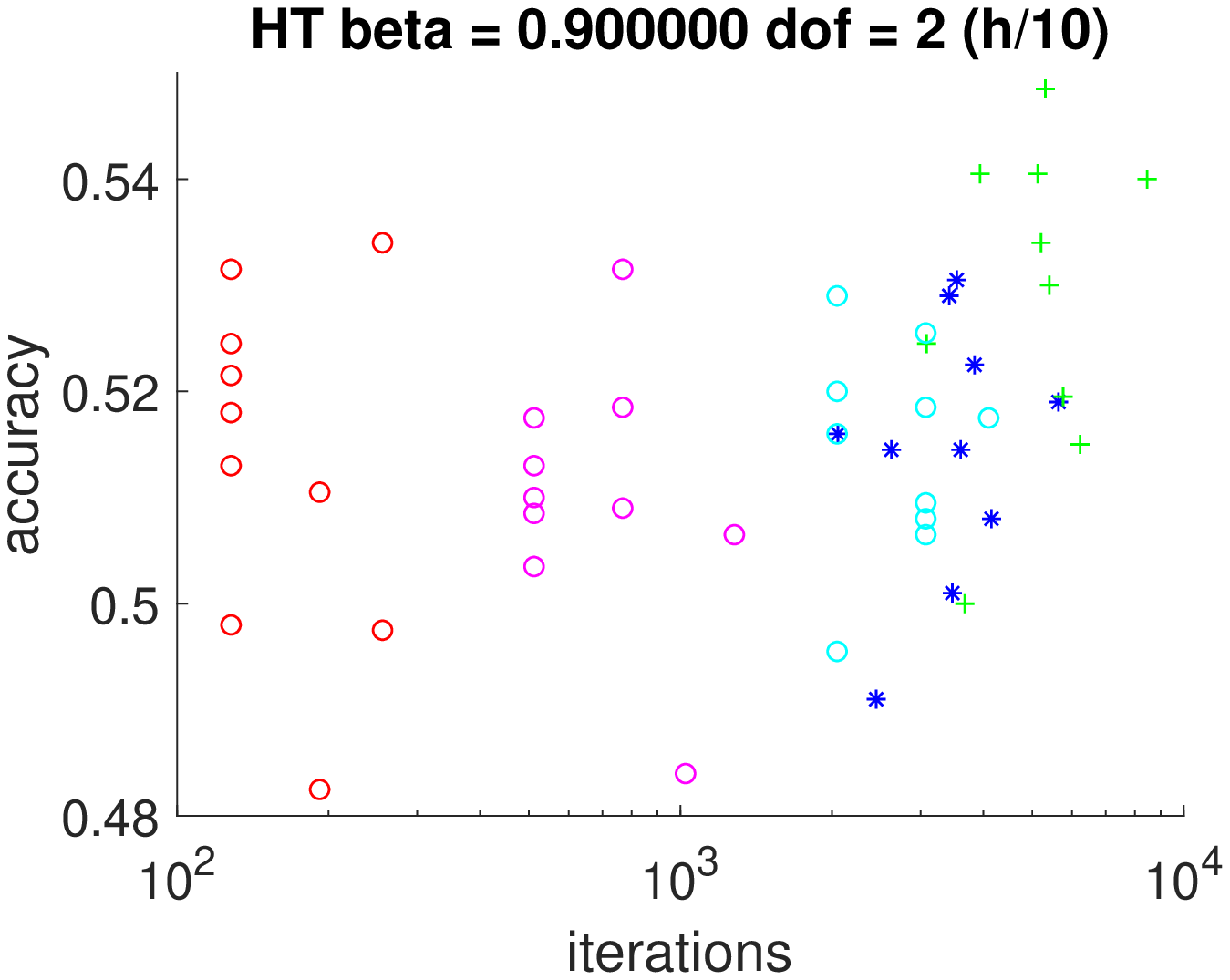} \qquad \includegraphics[height=2.5cm]{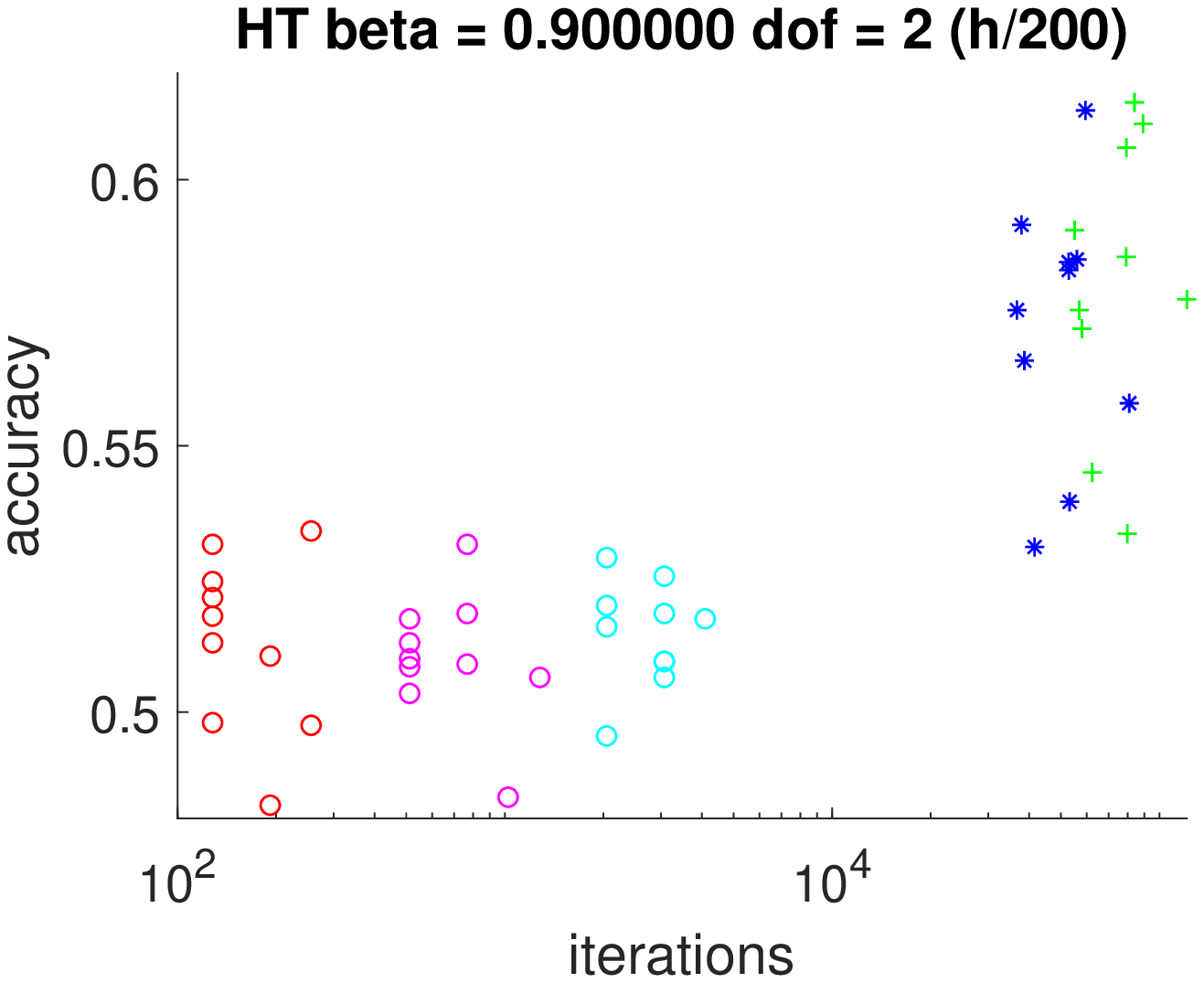} \\
    \end{center}
        \caption{Tests on the student-t distribution (heavy tailed) with two degrees of freedom and the indicated value of parameter $\beta$.  See the caption of Fig.~\ref{fig:normal_scatter} for explanation of the plots.}
    \label{fig:ht_scatter}
\end{figure}

\paragraph{Heavy-tailed distribution.}
We consider the student t-distribution with two degrees of freedom.  This
distribution is heavy-tailed since some of its
higher moments are infinite.  

The two classes were generated as follows.  For $\bzeta$ in the 0-class, each of the $d$ entries of $\bzeta$ is chosen as
$\beta\eta$, where $\beta$ is varied in the experiments and
$\eta$ is drawn from the student t-distribution with two degrees of freedom.  For the 1-class, $\bzeta$ is chosen in the same way except that the first entry is incremented by 1. Fig.~\ref{fig:ht_scatter} shows our performance against SVS.
The results in this table show similar trends as in the normally distributed case.  One difference is that the accuracy achieved by our termination test \eqref{eq: practical_termination_test} is more spread out presumably because of the heavy-tailed nature of the data set.

\subsection{Experiments with real data} \label{sec:comp-realistic}

\paragraph{MNIST handwritten digits.} We compared our termination test on the MNIST handwritten digit set \citep{MNIST} ($d=784$, no preprocessing of the data other than centering between the two means).  Two trials are shown: distinguishing 1 from 8 (easy case) and distinguishing 7 from 9 (more difficult case).  The test runs are obtained by running through the training data in different randomized orders.  The plots in Fig.~\ref{fig:mnist_scatter} show similar trends as before.  As expected, the accuracy is overall higher for 
$\tilde\alpha=1/200$ than for $\tilde\alpha=1/10$.

\begin{figure}
    \centering
    \includegraphics[height=2.5cm]{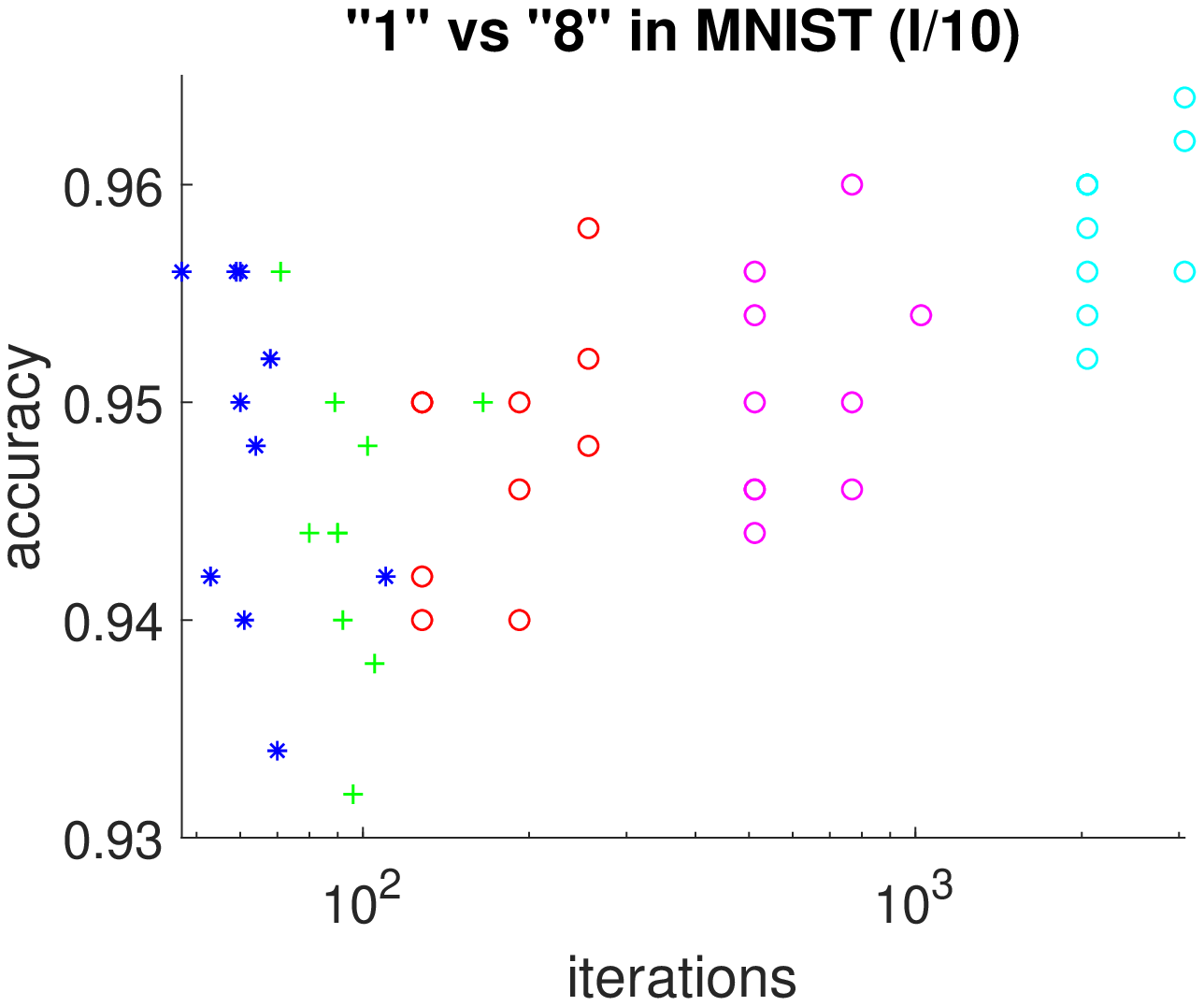} \qquad     \includegraphics[height=2.5cm]{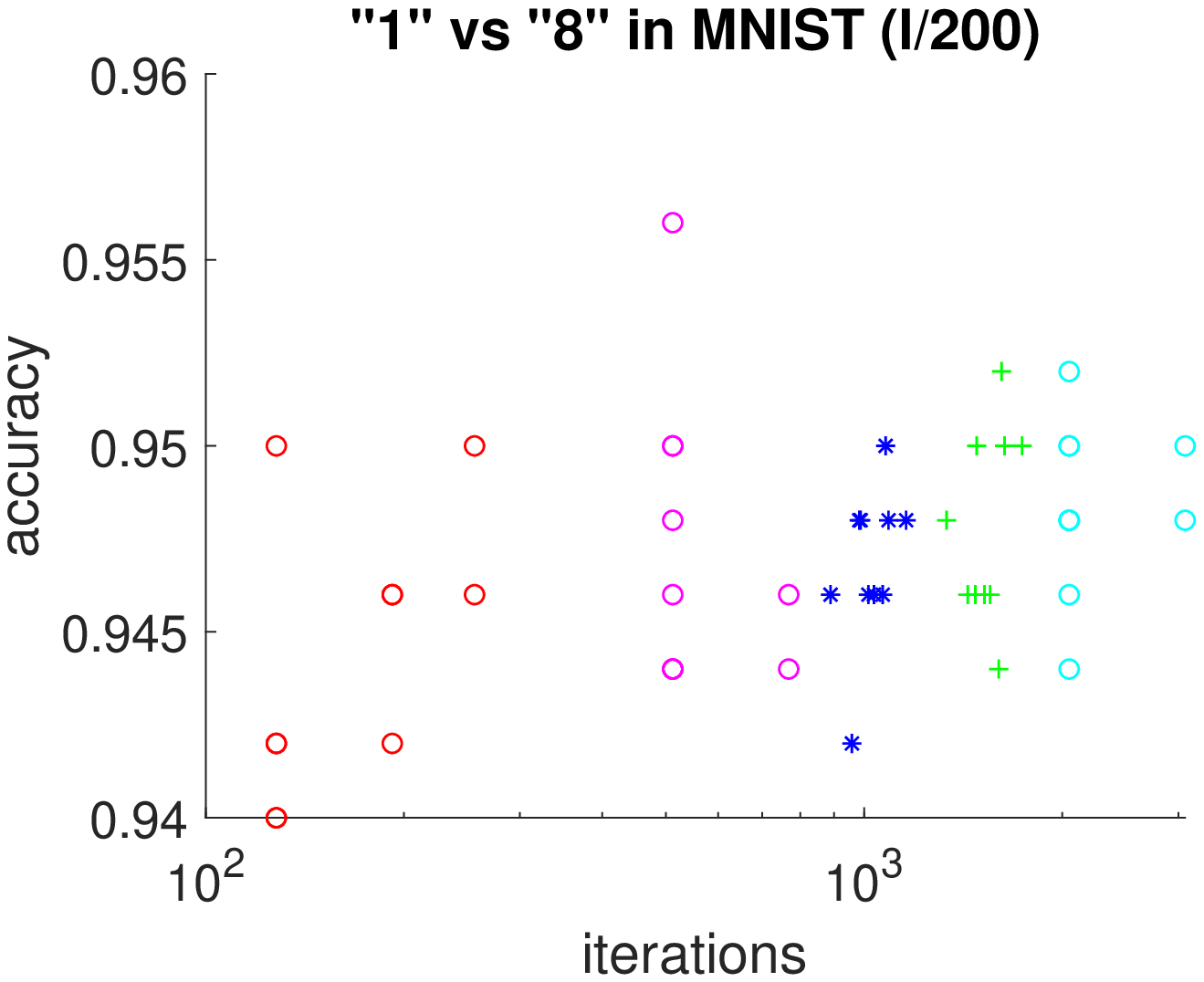} \qquad  \includegraphics[height=2.5cm]{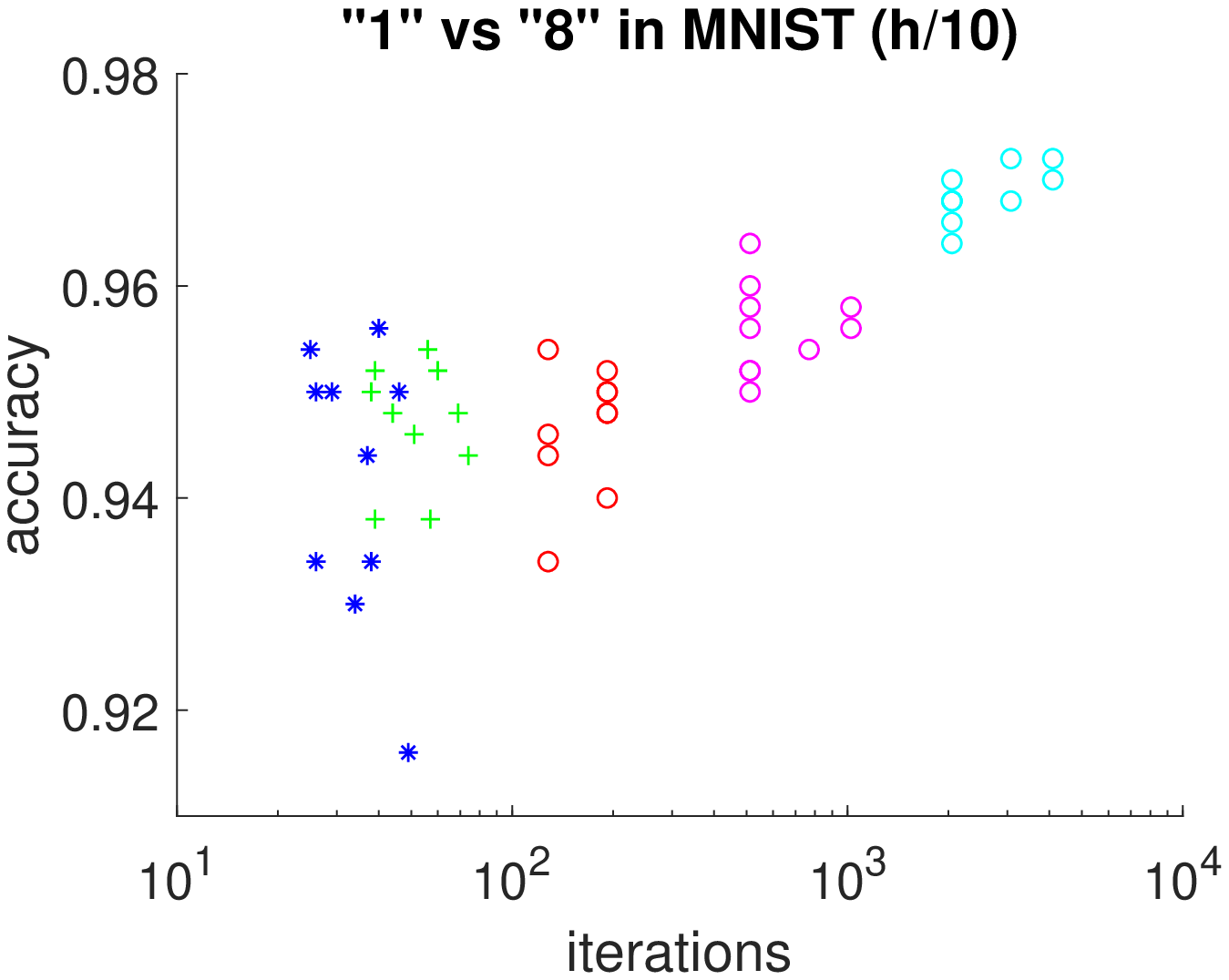} \qquad     \includegraphics[height=2.5cm]{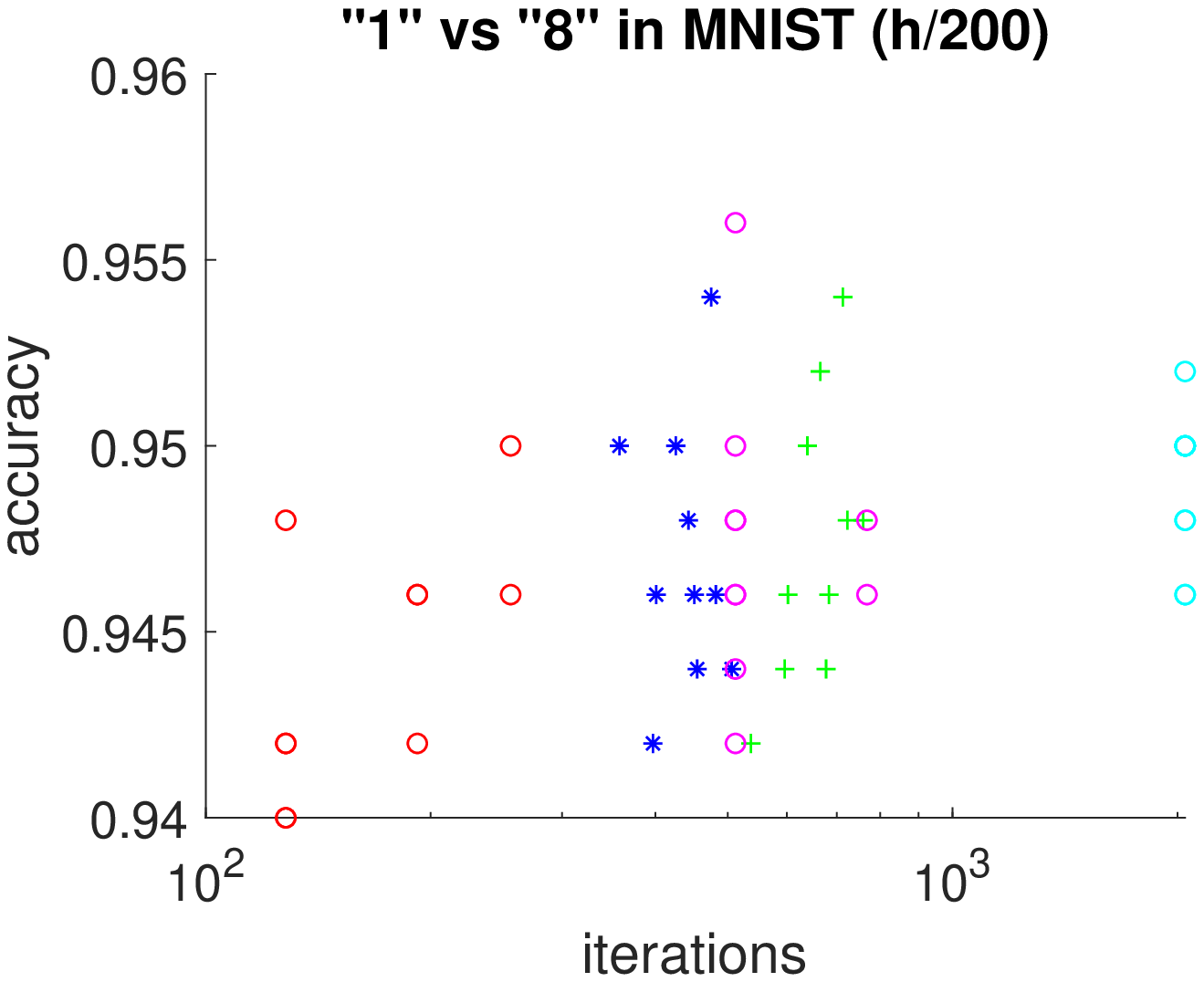} \\  \vspace{2 mm} \includegraphics[height=2.5cm]{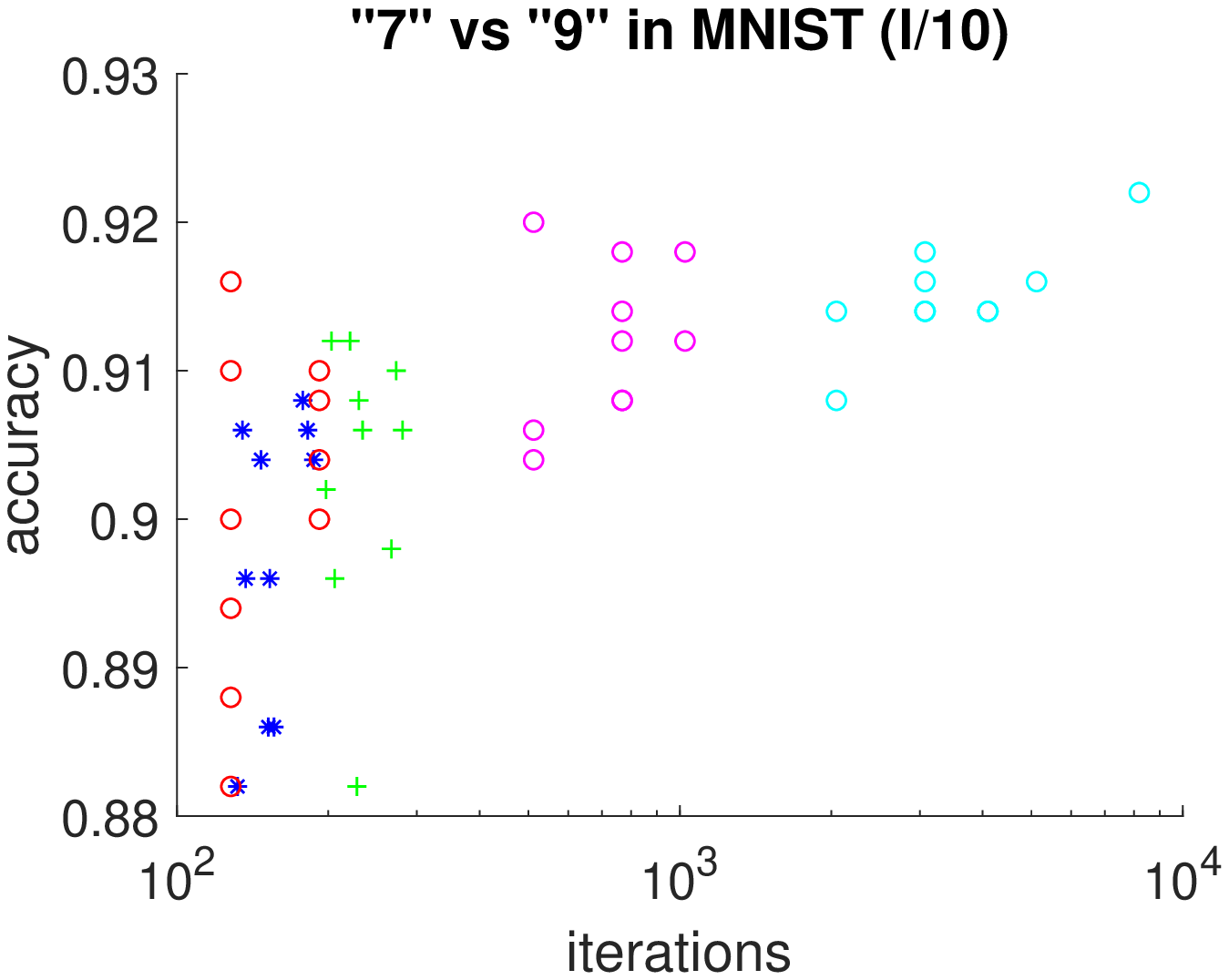} \qquad     \includegraphics[height=2.5cm]{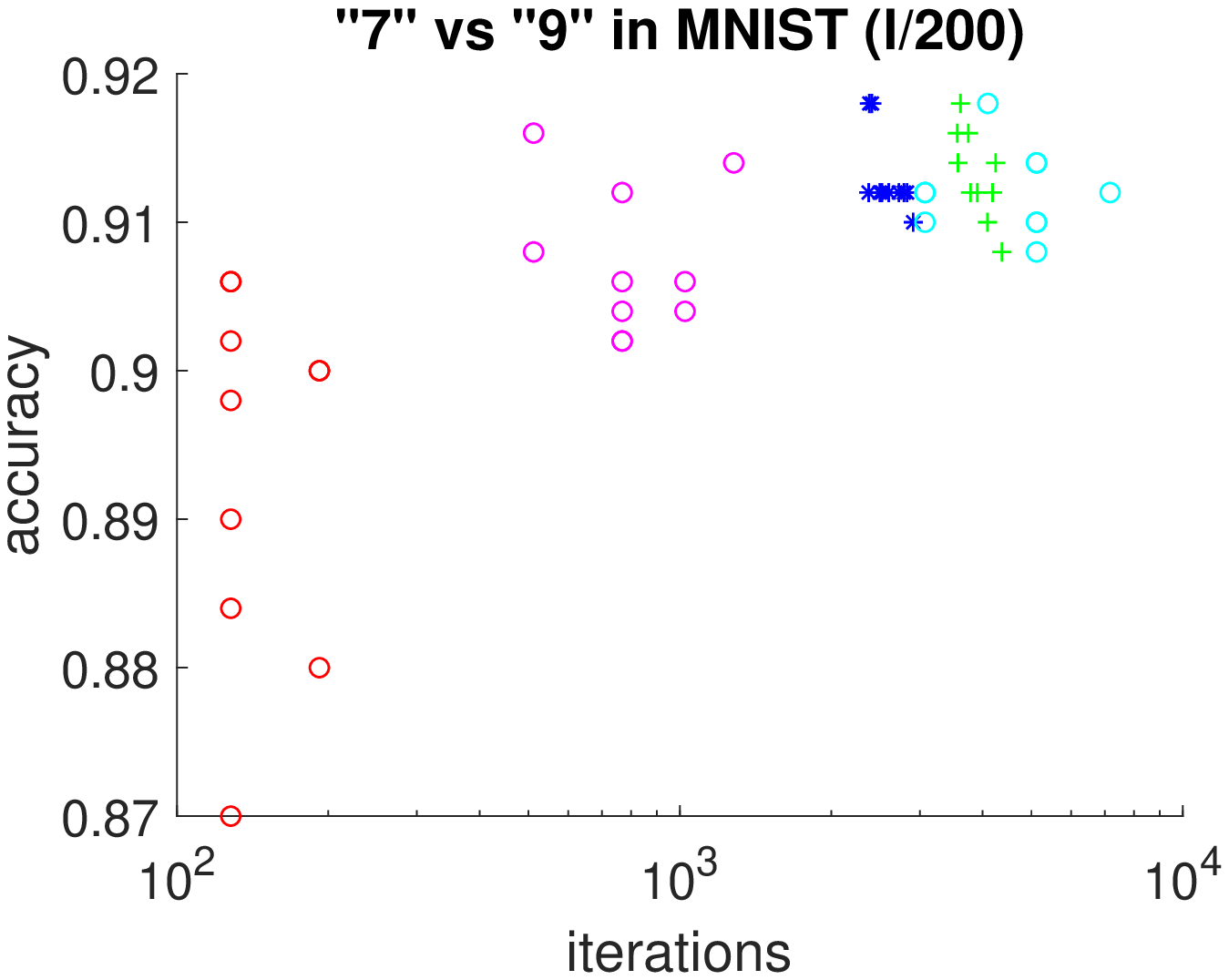} \qquad \includegraphics[height=2.5cm]{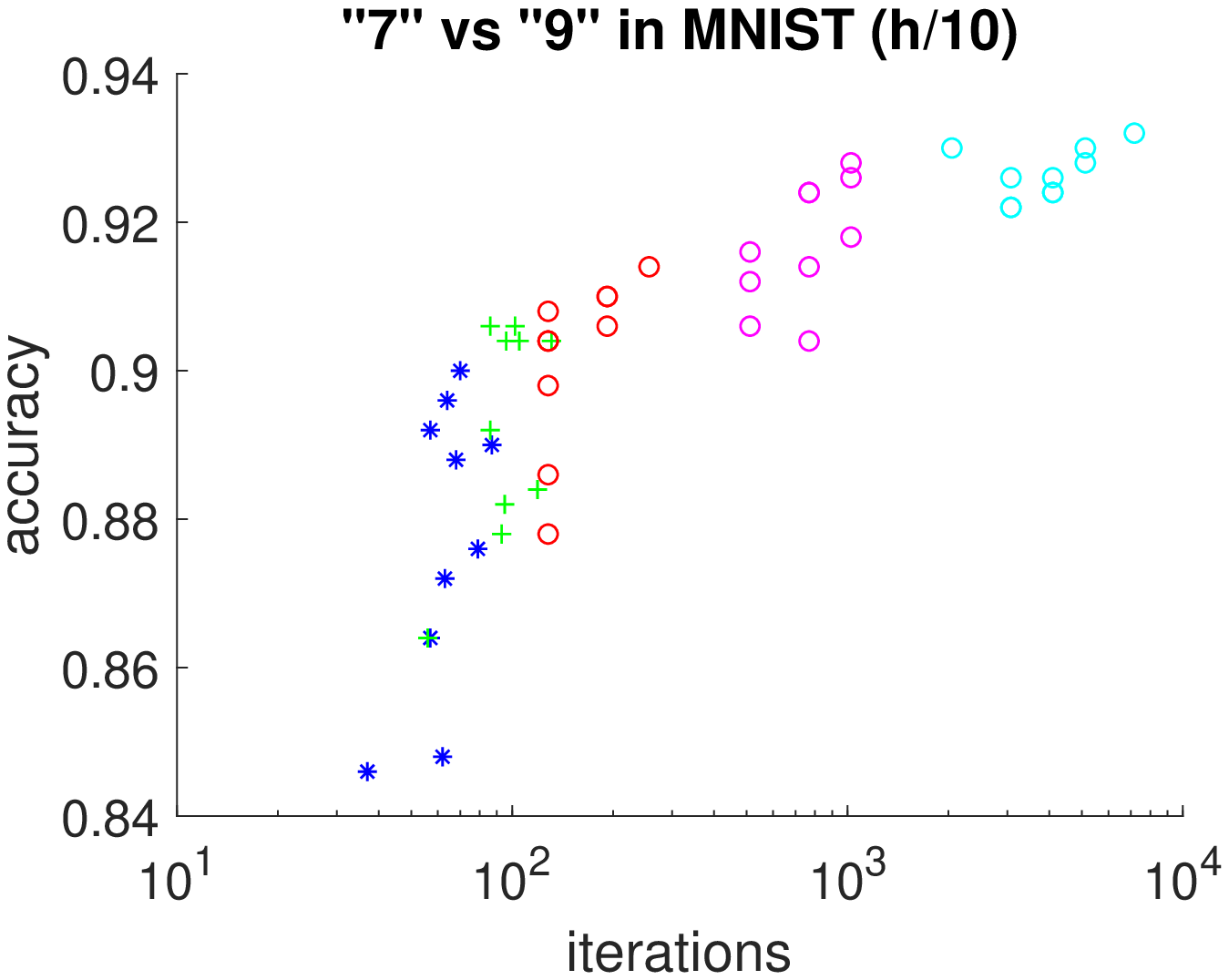} \qquad     \includegraphics[height=2.5cm]{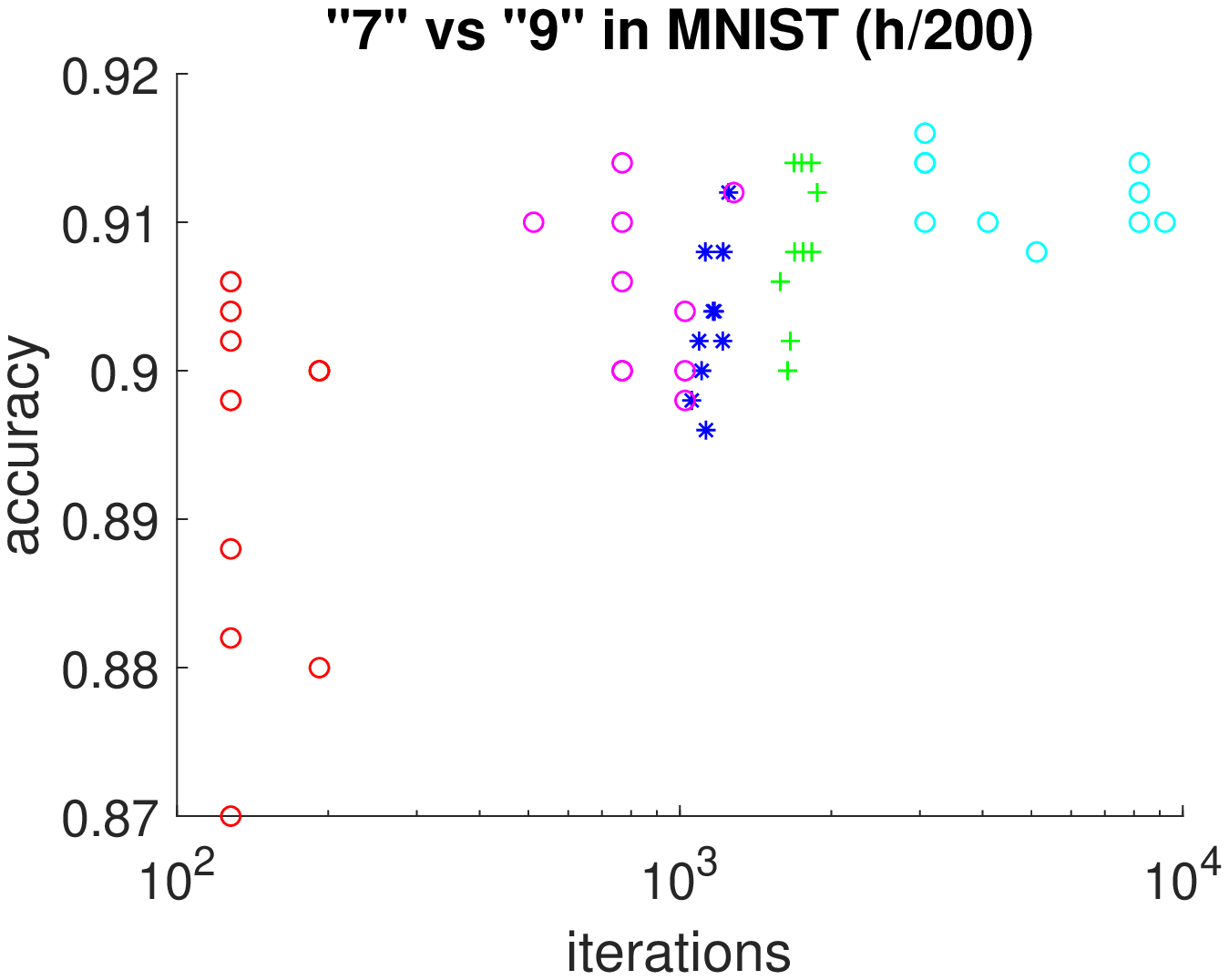} \\
    \caption{Tests on the MNIST handwritten digit data set for discerning ``1'' from ``8'' and ``7'' from ``9'' for both hinge and logistic, and for both $\tilde\alpha=1/10$ and $\tilde\alpha=1/200$.  Refer to the caption of Fig.~\ref{fig:normal_scatter} for the key to the plots.}
    \label{fig:mnist_scatter}
\end{figure}

\paragraph{CIFAR-10 image set.} We compared our termination test on the CIFAR-10 \citep{cifar10} ($d=3072$, no preprocessing of the data other than centering between the two means as described earlier).  Two trials are shown: distinguishing deer from airplanes and frogs from trucks. As in MNIST, test runs are obtained by running through the training data in different randomized orders.

\begin{figure}
    \centering
    \includegraphics[height=2.5cm]{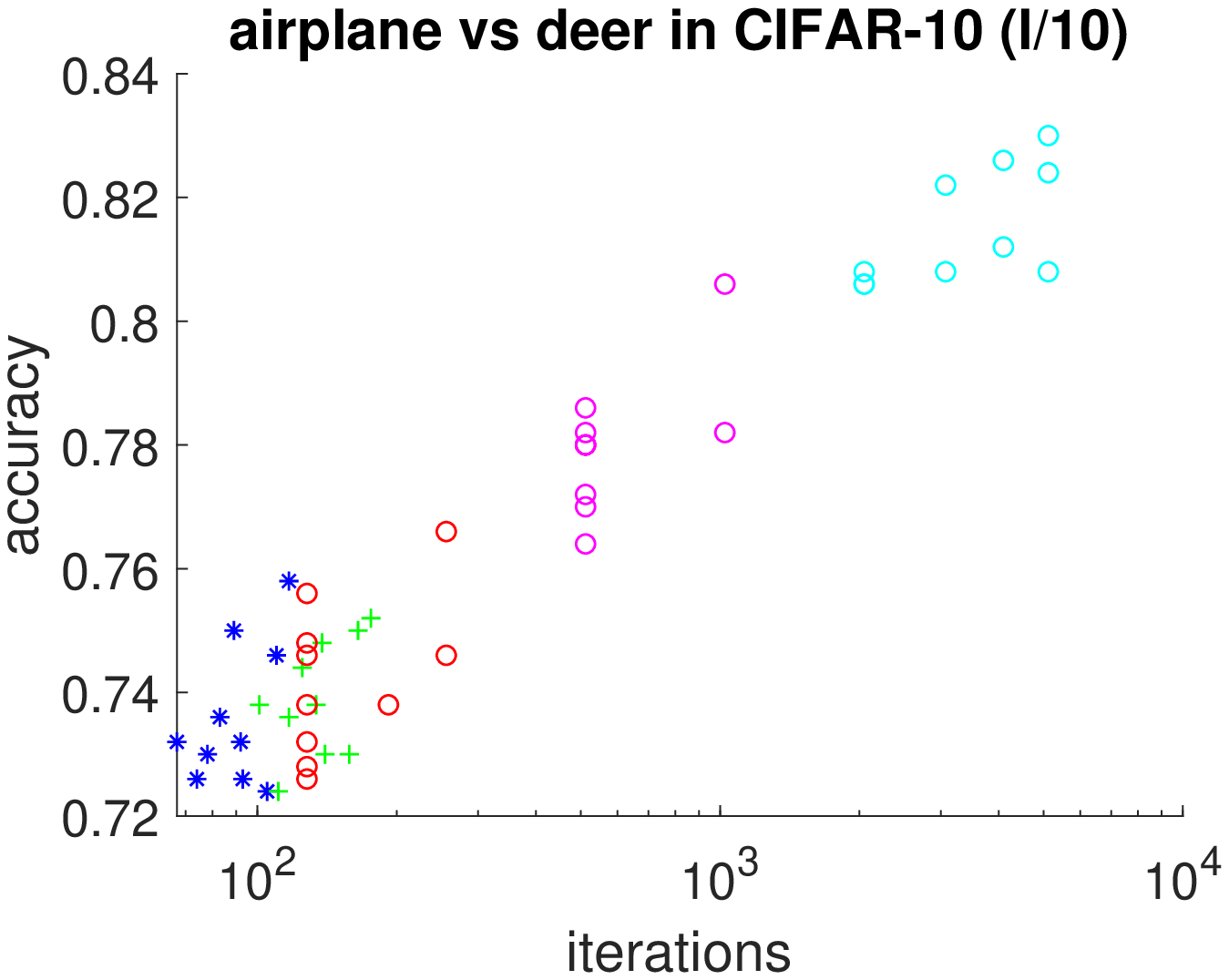} \qquad \includegraphics[height=2.5cm]{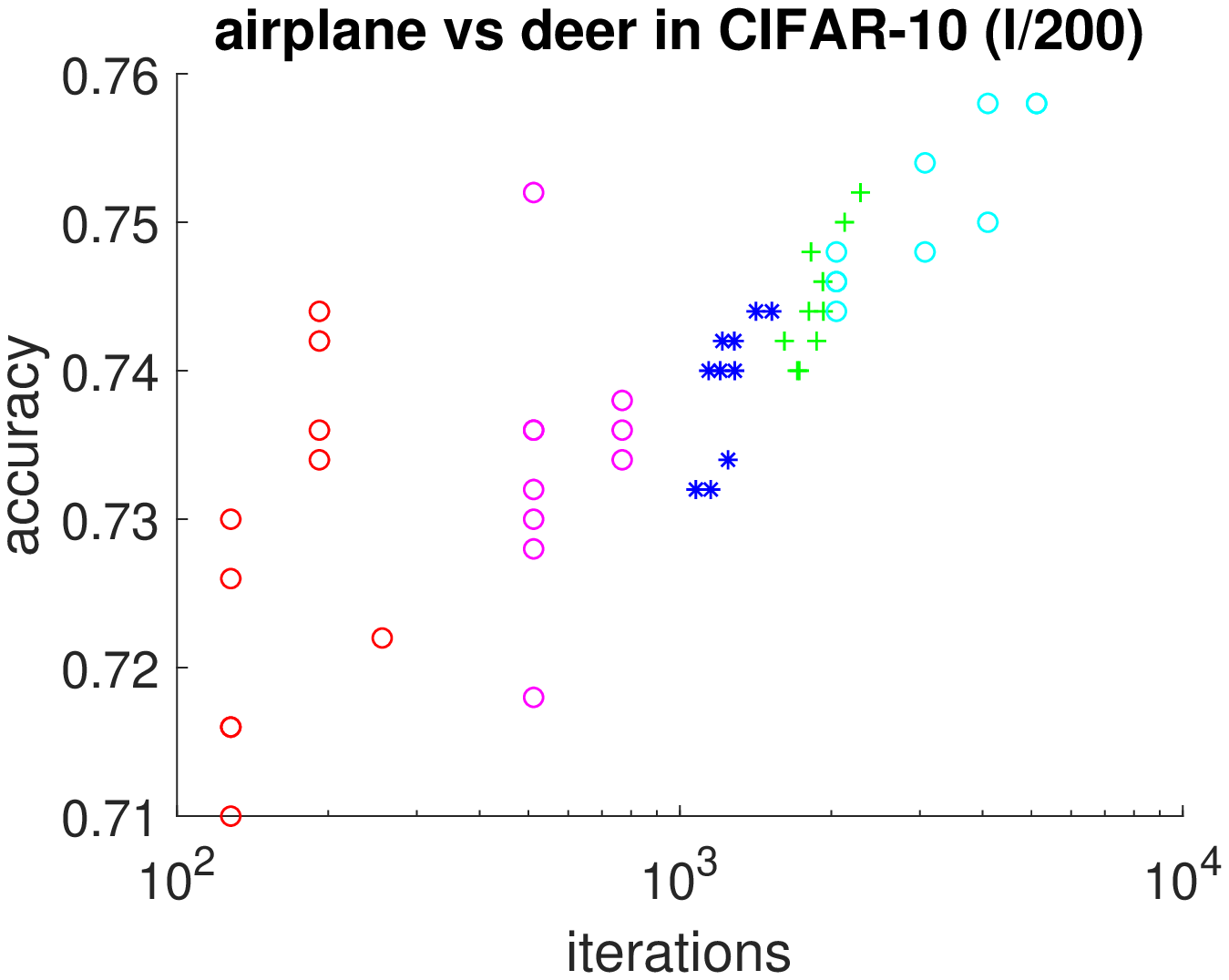} \qquad \includegraphics[height=2.5cm]{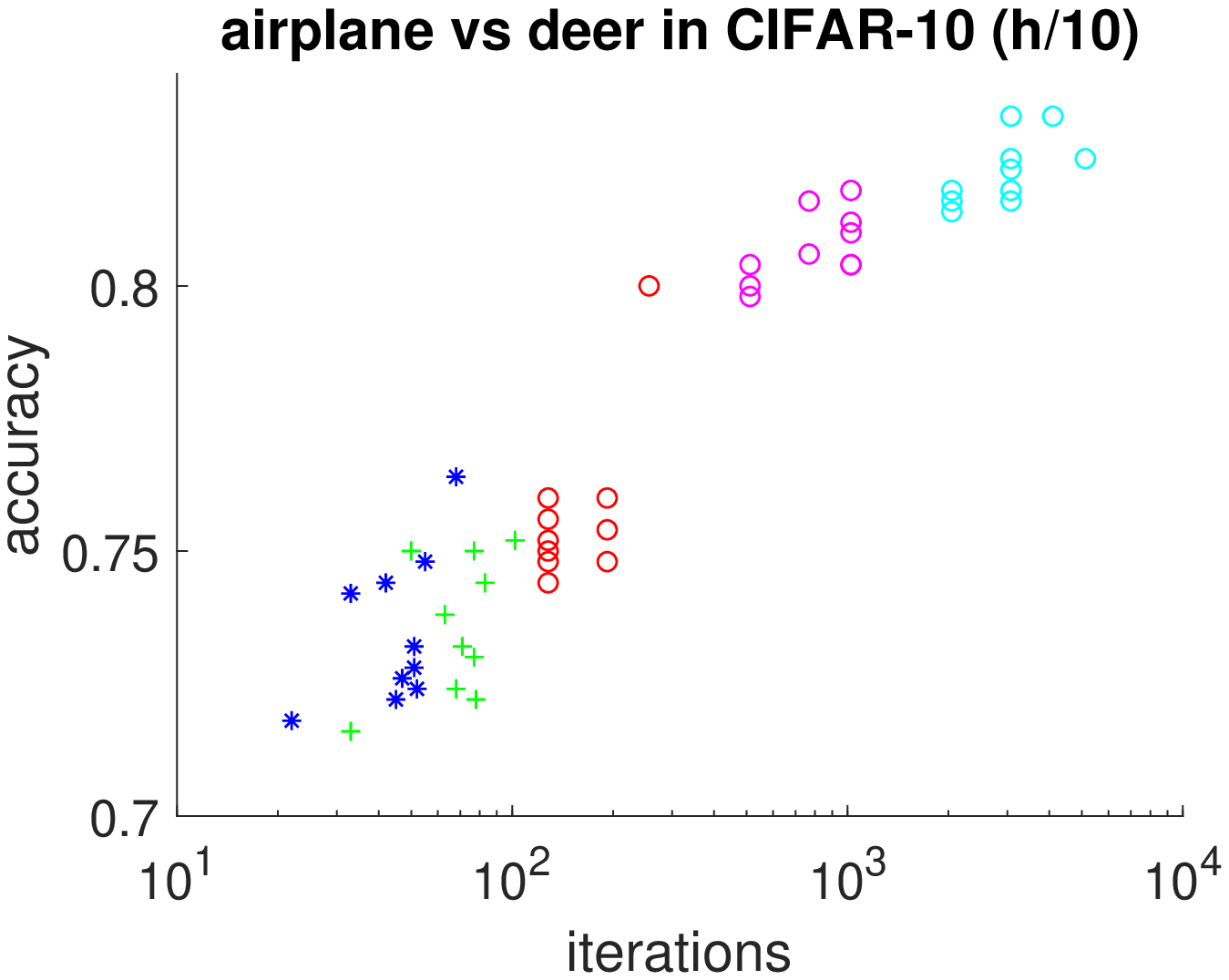} \qquad \includegraphics[height=2.5cm]{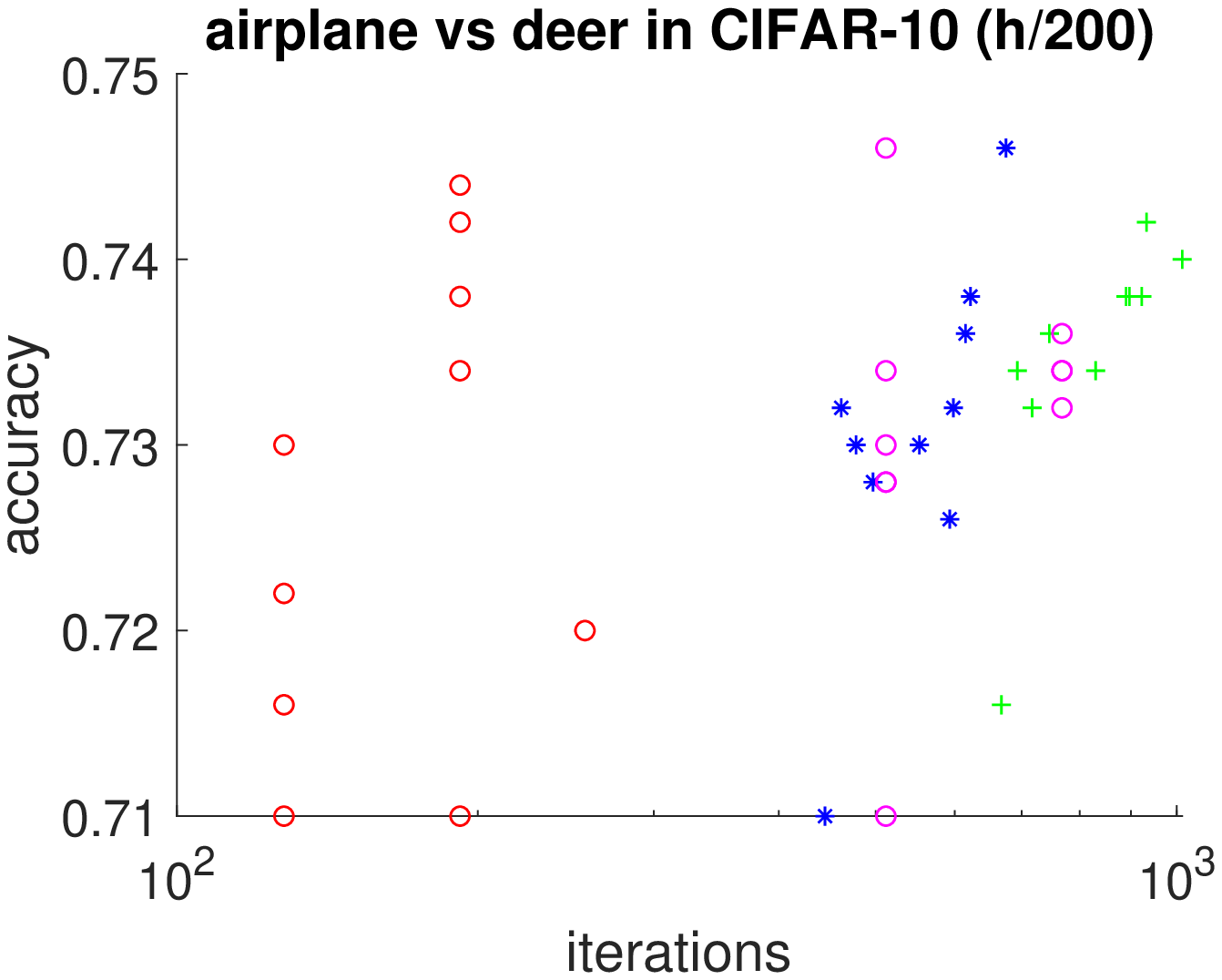} \\ \vspace{2 mm}
\includegraphics[height=2.5cm]{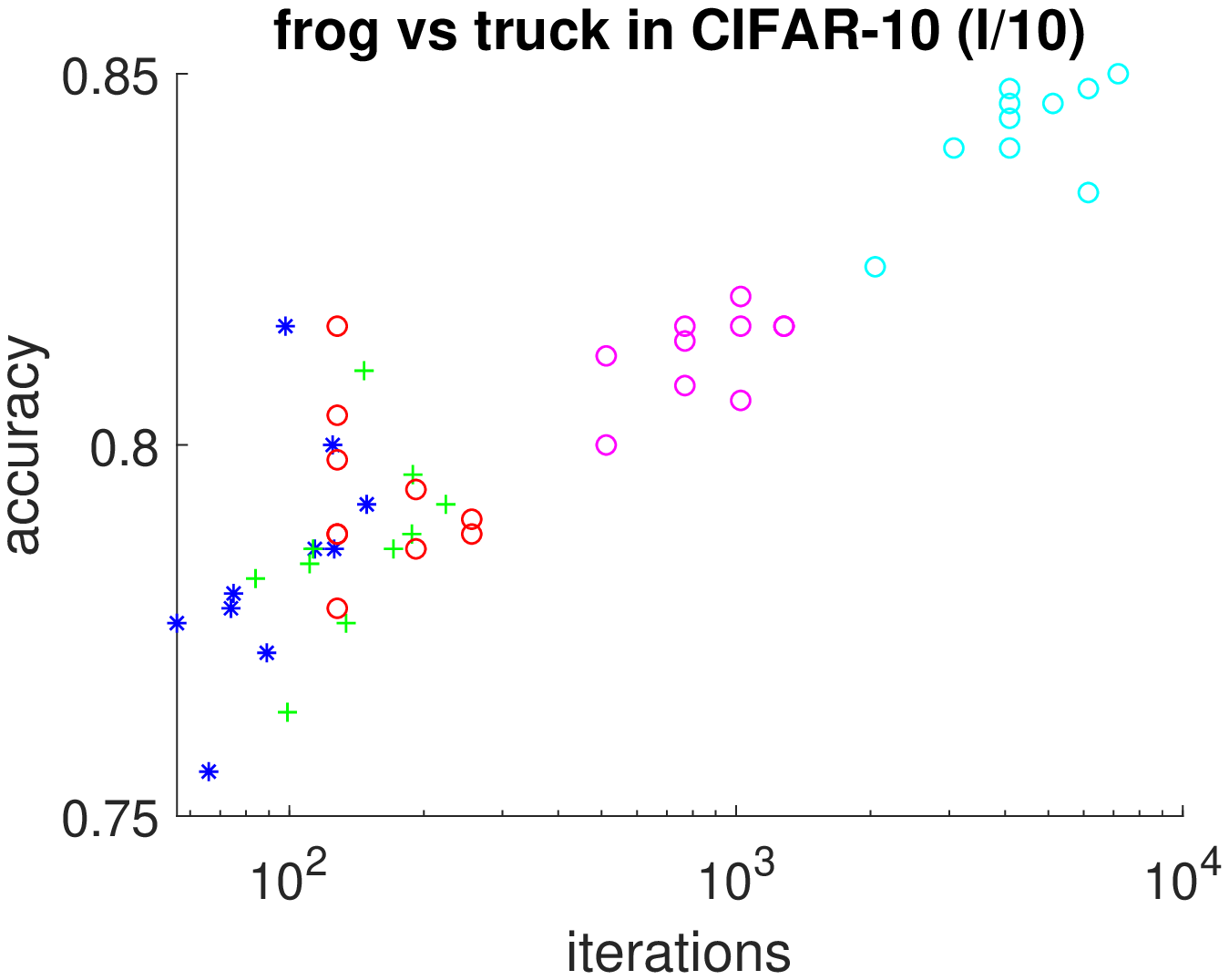} \qquad
 \includegraphics[height=2.5cm]{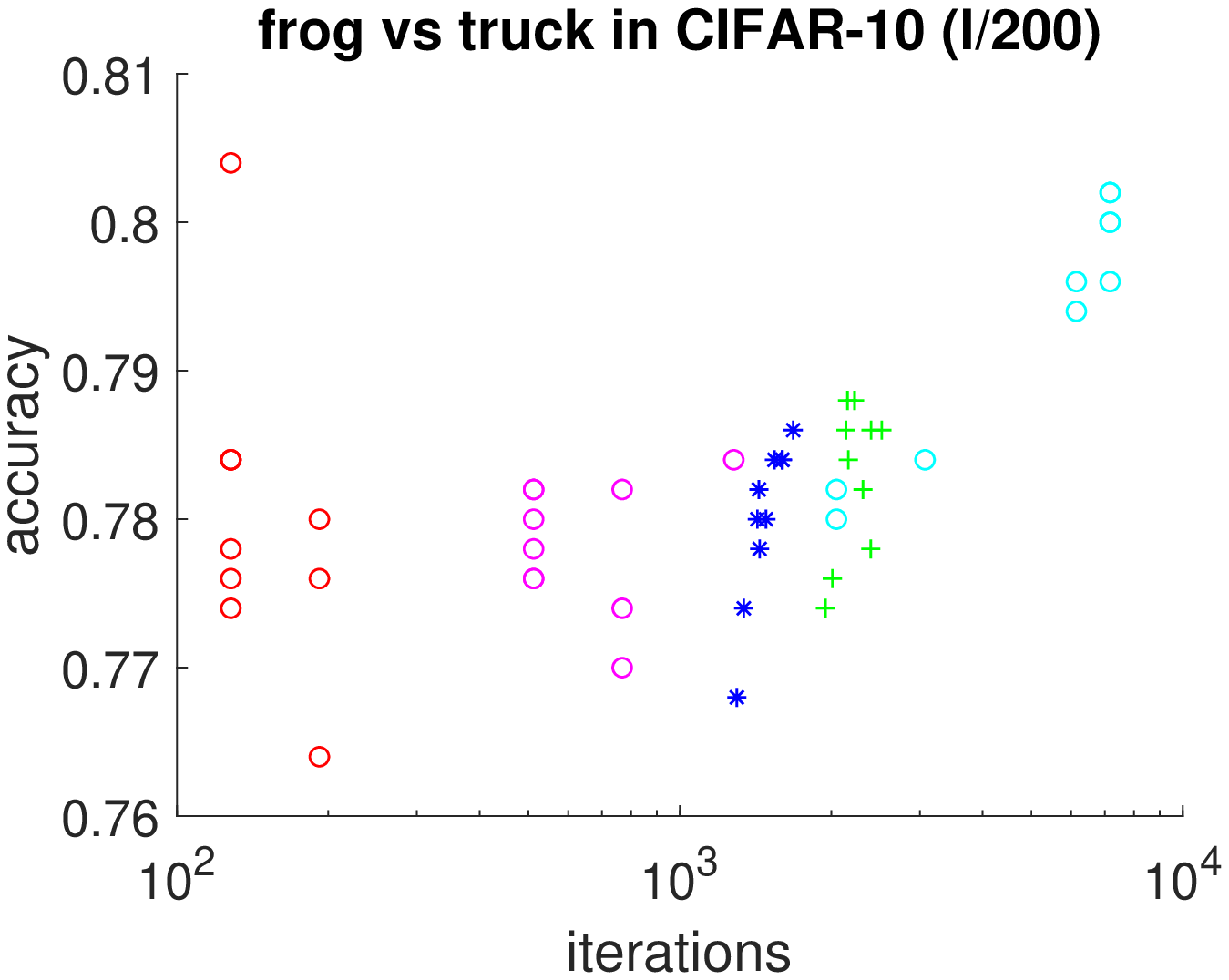} \qquad \includegraphics[height=2.5cm]{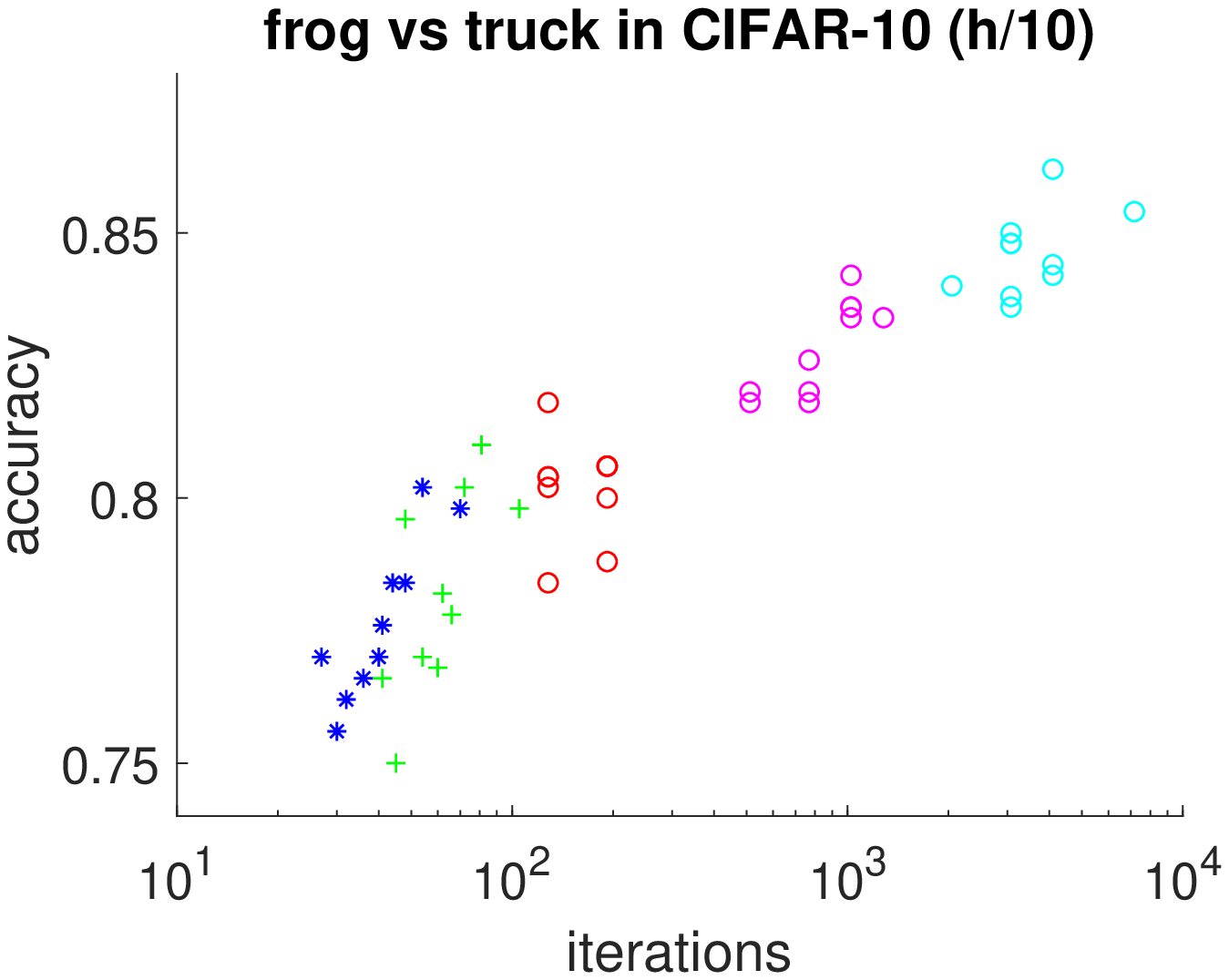} \qquad \includegraphics[height=2.5cm]{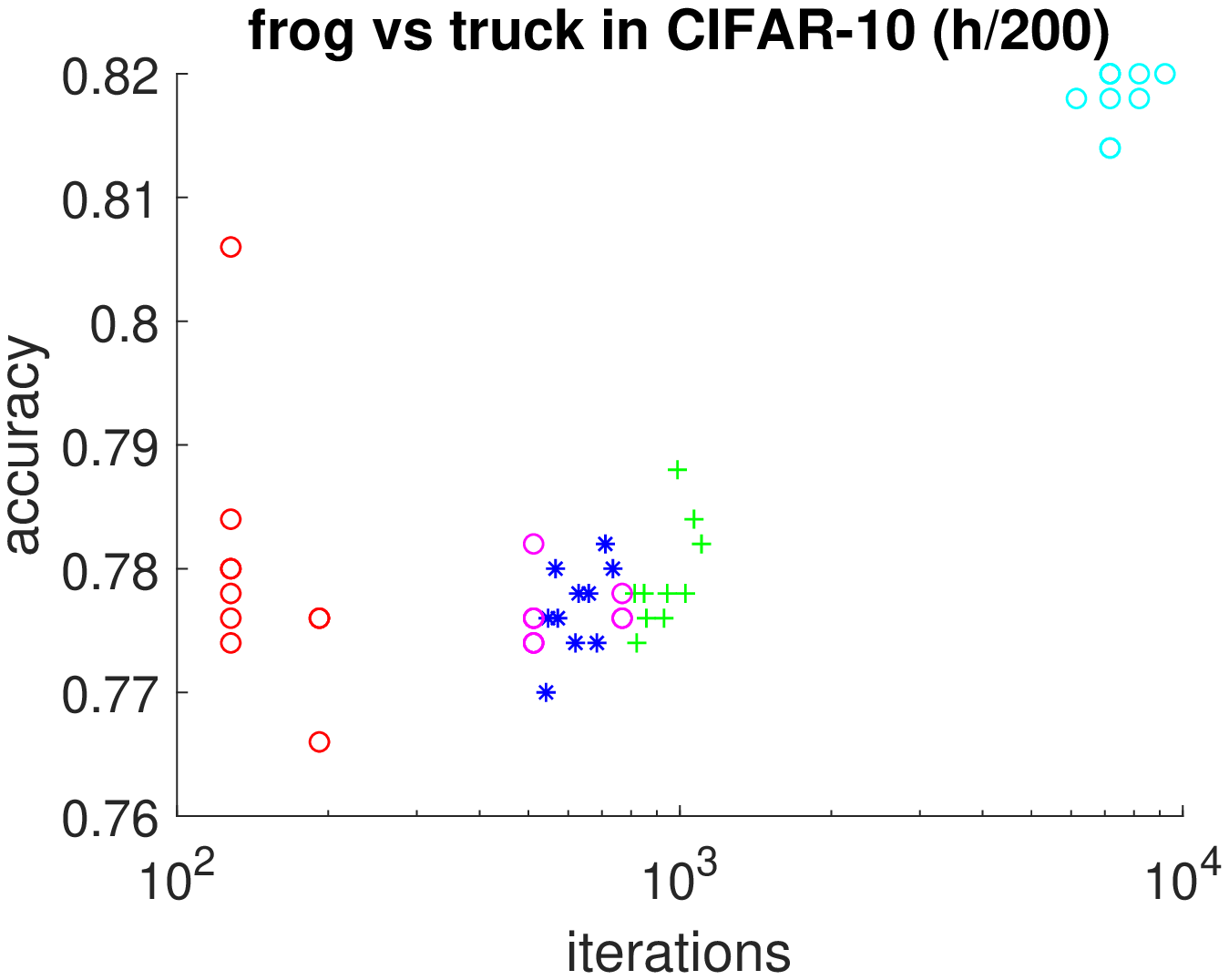}
        \caption{Tests on the CIFAR-10 image set for two tasks, for logistic and hinge losses, and for $\tilde\alpha=1/10$ and $\tilde\alpha=1/200$.   Refer to the caption of Fig.~\ref{fig:normal_scatter} for the key to the plots.  The plot in the first row, right, does not include cyan circles because the training data was exhausted before the SVS test could activate for $p=512$.}
    \label{fig:cifar_scatter}
\end{figure}

\section{Conclusions}
We have proposed a simple and computationally free termination test for SGD for binary classification, supported by both theoretical and experimental results.
The theoretical results show that the test will stop SGD after a finite time with a bound on the expected accuracy of the resulting classifier.  The bounds that we proved are weaker than what we observed in our experiments.  Therefore, the first obvious question left open by this work is whether the theoretical bounds can be improved.

In our experimental results,
the plots in Figs.~\ref{fig:normal_scatter} through \ref{fig:cifar_scatter} show a consistent pattern that
\eqref{eq: practical_termination_test} achieves low accuracy but is faster than SVS for $\tilde\alpha=1/10$, while it achieves higher accuracy with more iterations when $\tilde\alpha=1/200$.  This is useful behavior in practice, compared to SVS, since it puts the accuracy/iterations tradeoff in the hands of the user who selects the stepsize $\tilde\alpha$. Another benefit of \eqref{eq: practical_termination_test} apparent from all plots is that the number of iterations is more consistent across random trials, which is beneficial in the case that SGD is used as a subproblem of a larger computation.  

This work did not explore regularization via early stopping.  As mentioned in the introduction, experiments showed that as SGD iterations continued, the accuracy on the test set eventually levels off but does not decrease significantly, {\em i.e.}, SGD for binary classification is not prone to overfitting.  Because the test accuracy never shows marked decline, there is no opportunity for early stopping to regularize.  However, we know of other settings in which early stopping has a strong regularizing effect ({\em e.g.}, conjugate gradient iterations for image deconvolution, already known in \cite{CG_implicit_regularization}), so if \eqref{eq: practical_termination_test} is extended beyond binary classification in future work, there will likely also be an opportunity to explore regularization.

\bibliographystyle{plainnat}
\bibliography{bib}

\end{document}